\documentclass[11pt,twoside]{amsart}

\usepackage[utf8]{inputenc}
\usepackage[T1]{fontenc}
\usepackage{amssymb,amsmath,amsthm,mathtools,amstext}
\usepackage{hyperref}
\usepackage{xcolor}
\usepackage{booktabs}
\usepackage{graphicx}
\usepackage{pgfplots}\pgfplotsset{compat=1.18}
\usepackage{amsaddr}

\newtheorem{theor}{Theorem}[section]
\newtheorem{lem}[theor]{Lemma}
\newtheorem{prop}[theor]{Proposition}
\newtheorem{cor}[theor]{Corollary}
\newtheorem{fact}[theor]{Fact}

\newtheorem{defin}[theor]{Definition}
\newtheorem{notation}[theor]{Notation}
\newtheorem{exam}[theor]{Example}

\newtheorem{rem}[theor]{Remark}

\newcommand{\rng}{\mathrm{rng}}

\newcommand{\mc}{\mathcal}
\newcommand{\mb}{\mathbf}
\newcommand{\mbb}{\mathbb}

\newcommand{\es}{\emptyset}

\newcommand{\mcA}{\mathcal{A}}
\newcommand{\mcB}{\mathcal{B}}

\newcommand{\mcD}{\mathcal{D}}

\newcommand{\mcG}{\mathcal{G}}

\newcommand{\mbG}{\mathbf{G}}

\newcommand{\mbU}{\mathbf{U}}

\newcommand{\mbW}{\mathbf{W}}
\newcommand{\mbX}{\mathbf{X}}
\newcommand{\mbY}{\mathbf{Y}}

\newcommand{\mbbE}{\mathbb{E}}

\newcommand{\mbbP}{\mathbb{P}}
\newcommand{\mbbM}{\mathbb{M}}
\newcommand{\mbbN}{\mathbb{N}}
\newcommand{\mbbR}{\mathbb{R}}

\newcommand{\eff}{f}
\newcommand{\Eff}{F}

\title[Random coloured digraphs defined by a Markov logic network]{Random coloured digraphs defined by a \\ Markov logic network}


\author{Yasmin Tousinejad and Vera Koponen}

\address{Department of Mathematics, Uppsala University, Sweden.}
\email{yasmin.tousinejad@math.uu.se}

\address{Department of Mathematics, Uppsala University, Sweden.}
\email{vera.koponen@math.uu.se}

\date{19 November, 2026}

\begin{document}

\begin{abstract}
{\bf Background:} A Markov Logic Network (MLN) $\mbbM$ is a probabilistic relational model used in Statistical Relational Artificial
Intelligence for defining a probability distribution $\mbbP_D^\mbbM$ on the set $\mbW_D$
of possible worlds with domain $D$ for an arbitrary finite domain $D$. 
An MLN consists of first-order formulas, which we call soft constraints, 
with associated weights which are nonnegative real numbers. The behaviour of the
probability distribution, including the effect of the weights, on different domain sizes is poorly understood.\\
{\bf Objectives:} We approach the problem in a systematic way by 
studying sequences $(\mbbP_n^\mbbM(\psi) : n = 1, 2, 3, \ldots)$ where $\psi$ is a first-order sentence
and $\mbbP_n^\mbbM(\psi)$ is the probability that $\psi$ holds in $\mbW_D$ where $D = [n] := \{1, \ldots, n\}$.\\
{\bf Results:} We consider a first-order language that can express a property $P(x)$ and relation $R(x, y)$ and show that
all of the following cases can be realized for an MLN $\mbbM$ such that each of its soft constraints is a boolean combination of
$P(x)$, $P(y)$ and $R(x, y)$ in which both $x$ and $y$ appears:
\begin{enumerate}
\item For every first-order sentence $\psi$, $\mbbP_n^\mbbM(\psi)$ converges to 0 or 1 as $n\to\infty$.
\item For every first-order sentence $\psi$, $\mbbP_n^\mbbM(\psi)$ converges as $n\to\infty$ but not necessarily to 0 or 1.
\item There is a first-order sentence $\psi$ such that $\mbbP_n^\mbbM(\psi)$ does not converge as $n\to\infty$.
\end{enumerate}
In fact, all three cases can be realized even after fixing the soft constraints and, for any $\varepsilon > 0$, 
by varying their weights by less than $\varepsilon$.
We also show that, given any MLN as described above, we can efficiently determine which case it belongs to by 
inspecting the MLN.
Then we show that if the weights are scaled by $1/n$ then we end up in case~(1) no matter how the weights are chosen.\\
{\bf Conclusion:} The behaviour of MLNs on large domains can be complex, even for simple languages and 
soft constraints. Arbitrarily small
changes of the weights can radically change the behaviour on large domains.
This speaks against a reasonably simple general theory of domain size asymptotics for MLNs.
It can happen that a scaling of the weights by domain size simplifies the situation,
but at the cost of the weights becoming unimportant in the limit.

\end{abstract}

\maketitle

\tableofcontents

\section{Introduction}

\subsection{Probabilistic relational models}

{\em Probabilistic graphical models (PGMs)} have an established role in statistical/\-probabilistic AI \cite{Koller}
and can be used to define probability distributions on the set of possible states of a system, which we call the set of
possible worlds. More precisely a possible world consists of objects (in principle of any kind, including abstract objects) 
and a specification of which properties the objects have (or not have) and how different objects are related (or not related).
For the PGM approach to work out we must specify exactly which objects we consider, which is sometimes called
a {\em closed domain assumption}.

But often the closed domain assumption is too restrictive.
For example, we may have learned a PGM from examples of objects coming from one domain $D$ 
and in this case the PGM should define a probability distribution on the set of possible worlds with domain $D$
which reflects the observed properties of the examples (of objects from $D$) that the PGM was learned from.
But what if we wish to use the PGM on another domain of objects, possibly with many more, or many fewer, objects?
And how does a PGM learned from examples of possible worlds over one domain make sense on possible worlds on another domain?

To overcome this hurdle, researchers in the field of {\em Statistical Relational AI (SRAI)} 
\cite{RKNP, GT} have considered
formal specifications that define, for {\em every finite domain} $D$ of objects, a probability distribution on the set of
possible worlds of objects with domain $D$, which we denote by $\mbW_D$.
Such formal specifications have been called by various names, such as 
probabilistic relational models, lifted (probabilistic) graphical models,
parametrized (probabilistic) graphical models, and template (probabilistic) graphical models
\cite{BKNP, RKNP, GT, KMG}.

So for any finite domain (set of objects) $D$
a {\em probabilistic relational model (PRM)} $\mbbM$ defines a probability distribution $\mbbP_D^\mbbM$
on the set $\mbW_D$ of possible worlds with finite domain $D$. 
A possible world with domain $D$ can be represented mathematically by a $\sigma$-structure $\mcD$
with domain $D$ in the sense of first-order logic, where $\sigma$ a signature (also called language or vocabulary)
consisting of symbols representing the ``atomic relations/properties'' that we are interested in.
Since only the size, or cardinality, of the domain $D$ influences the nature of the 
distribution $\mbbP_D^\mbbM$ defined by a PRM
$\mbbM$ on $\mbW_D$ we will always let $D = [n] := \{1, \ldots, n\}$ for some positive integer $n$;
one can think of the objects in the original domain, whatever it is, as being labelled by $1, \ldots, n$.
Let $\mbW_n := \mbW_{[n]}$
and for a PRM $\mbbM$ we let $\mbbP_n^\mbbM := \mbbP_{[n]}^\mbbM$,
so $\mbbP_n^\mbbM$ is the distribution on $\mbW_n$ that is specified by $\mbbM$.

\subsection{Knowledge transfer}

Let $\psi$ denote a statement which may be true or false in a possible world $\mcA \in \mbW_n$
and let $\mbbP_n^\mbbM(\psi)$ denote the probability (using $\mbbP_n^\mbbM$)
that $\psi$ is true in a possible world from $\mbW_n$. 
Suppose that we have computed $\mbbP_n^\mbbM(\psi)$, or an approximation of it, which is computationally expensive
in general for large $n$ since the size of $\mbW_n$ is at least $2^n$ 
(provided that $\sigma$ is nonempty which is the interesting case).
The question now arises whether we can generalize our knowledge about $\mbbP_n^\mbbM(\psi)$ to domains of other sizes.
This is the question about {\em knowledge transfer for
events expressed by first-order sentences}. More precisely: do we have $\mbbP_n\mbb(\psi) = \mbbP_m(\psi)$,
or at least $\mbbP_n^\mbbM(\psi) \approx \mbbP_m^\mbbM(\psi)$,
for other domain sizes $m$?  
The question is also relevant for learning a PRM. 
Suppose that we have learnt a PRM $\mbbM$ from examples with domain size $n$.
This means that the distribution $\mbbP_n^\mbbM$ on $\mbW_n$ reflects the distribution of various properties
of possible worlds in $\mbW_n$. 
But what justifies that, for $m \neq n$, the distribution $\mbbP_m^\mbbM$ on $\mbW_m$
 has similar properties as $\mbbP_n^\mbbM$. 
 For example, for a statement $\psi$ of interest, do we have $\mbbP_n^\mbbM(\psi) \approx \mbbP_m^\mbbM(\psi)$
 where $\mbbP_n^\mbbM(\psi)$ is the learned probability of $\psi$ on possible worlds with domain size $n$
 and $m$ is another domain size.

Since large domain sizes pose computational difficulties, and also frequently occur in applications, we consider the question 
of whether $\mbbP_n^\mbbM(\psi) \approx \mbbP_m^\mbbM(\psi)$ for large $n$ and $m$.
Moreover, since we wish that the the difference between $\mbbP_n^\mbbM(\psi)$ and $\mbbP_m^\mbbM(\psi)$
becomes ignorable in the limit we ask whether, for any $\varepsilon > 0$,
$|\mbbP_n^\mbbM(\psi) - \mbbP_m^\mbbM(\psi)| < \varepsilon$ for all sufficiently large $n$ and $m$.
This is equivalent to asking whether $\lim_{n\to\infty}\mbbP_n^\mbbM(\psi)$ exists.
In order to make the question mathematically precise we need to make precise what kind of statements $\psi$ we have in mind.
A natural choice is to let $\psi$ be a sentence (formula without free variables) of first-order logic as it
is a widely known and widely used language in mathematics, computer science, artificial intelligence (AI), and (among other things)
in applications of logic such as formal methods and automated theorem proving.
For a PRM $\mbbM$ let us say that the sequence $(\mbbP_n^\mbbM : n \in \mbbN^+)$, where $\mbbN^+$ is the
set of positive integers, satisfies a {\em convergence law for first-order logic} if 
$\lim_{n\to\infty} \mbbP_n^\mbbM(\psi)$ exists for every first-order sentence $\psi$ 
(that uses the signature that $\mbbM$ considers).
So now the question of knowledge transfer for a PRM $\mbbM$ can be given the following precise form:
\begin{enumerate}
\item[] Does $(\mbbP_n^\mbbM : n \in \mbbN^+)$ satisfy a convergence law for first-order logic?
\end{enumerate}
A positive answer implies that for every first-order sentence $\psi$ and $\varepsilon > 0$, if $n$ is sufficiently large
then $\mbbP_n^\mbbM(\psi)$ is an estimate of $\mbbP_m^\mbbM(\psi)$ with error at most $\varepsilon$ for all $m \geq n$,
so if this error marginal is acceptable then we can reason with the estimate $\mbbP_n^\mbbM(\psi)$
of $\mbbP_m^\mbbM(\psi)$ on all domains of size
$m \geq n$ instead of carrying out the heavy computation of $\mbbP_m^\mbbM(\psi)$ for each $m \geq n$ that we consider.

Note that since $\forall x (x=x)$ is true in all possible worlds and 
$\neg \forall x (x=x)$ is false in all possible worlds it follows that
$\lim_{n\to\infty} \mbbP_n^\mbbM\big(\forall x (x=x)\big) = 1$ and 
$\lim_{n\to\infty} \mbbP_n^\mbbM\big(\neg\forall x (x=x)\big) = 0$.
Therefore 0 and 1 are always limits of $\mbbP_n^\mbbM(\psi)$ as $n\to\infty$ for some first-order sentence $\psi$.
In surprisingly many contexts it has been shown, within the areas of finite model theory and probabilistic combinatorics,
that if $\mbbP_n$ is a probability distribution on $\mbW_n$ for all $n \in \mbbN^+$,
then $\lim_{n\to\infty} \mbbP_n(\psi)$ is either 0 or 1 for
every first-order sentence $\psi$. In this case we say that $(\mbbP_n : n \in \mbbN^+)$ satisfies a
{\em 0-1 law for first-order logic}.
So for a PRM $\mbbM$ we can distinguish three exclusive cases:
\begin{enumerate}
\item $(\mbbP_n^\mbbM : n \in \mbbN^+)$ satisfies a 0-1 law for first-order logic, or
\item $(\mbbP_n^\mbbM : n \in \mbbN^+)$ satisfies a convergence law, but not 0-1 law, for first-order logic, or 
\item $(\mbbP_n^\mbbM : n \in \mbbN^+)$ does not even satisfy a convergence law for first-order logic.
\end{enumerate}

It would be wonderful to have a general theory about first-order 0-1 laws, convergence laws, and absence of these
for PRMs. Since known convergence (and nonconvergence) results in finite model theory and combinatorics have
been proved by quite different methods for different types of randoms structures
(see e.g. \cite{Bur, Com, Fag, Gle, KPR, KopDeg, SS}) and since PRMs can define a great variety of 
probability distributions it seems unlikely to us
that a united and reasonably simple theory of convergence phenomena for all PRMs will be found.
However, for {\em directed} PRMs such as various kinds of relational, or lifted, Bayesian networks,
rather general positive results about convergence of first-order or other formal sentence have been found. 
Each such result applies to all relational/lifted Bayesian networks (RBNs) with a specified common architecture
and shows that for such RBNs $\mbbM$ we are in case (1) or (2) above
\cite{CM, Jae98a, Kop20, KW1, KW2, KT, Kop26, Wei21, Wei24}.

On the other hand, for {\em undirected} PRMs, such as Markov logic networks 
(\cite{DL, RD} and Definition~\ref{definition of MLN} below), very little is known concerning 
convergence phenomena.
As far as we know, before this study no convergence law for first-order logic or result refuting a convergence law first-order logic, has ever been proved for a single PRM, let alone for a family of similar PRMs.
The closest thing in the literature that we have found
are a few results showing that for a Markov logic network $\mbbM$ of a specific kind and a specific 
sentence $\psi$ that depends on $\mbbM$, the limit $\lim_{n\to\infty} \mbbP_n^\mbbM(\psi)$ exists
\cite{KopMLN, Mittal, Poole, Wei25}.
We will show that already for a simple signature and a fairly simple family of Markov logic networks all
three cases (1) -- (3) can appear, and an ever so small modification of the Markov logic network
may have the effect that we move from one case to another.
We also show that if the weights of the Markov logic network (of the kind considered here) are scaled by the domain size,
then the sequence of distributions satisfies a 0-1 law for first-order logic, so we are in case (1), no matter what the weights are.
Before going further we introduce the basic notation and terminology that will be used.

\subsection{Notation and terminology}\label{Notation and terminology}

\noindent
$\mbbN$ and $\mbbN^+$ denote the set of nonnegative, respectively, positive integers.
For all $n \in \mbbN^+$ we let $[n] := \{1, \ldots, n\}$.
Let $\exp(x) := e^x$.
If $S$ is a set then $|S|$ denotes its cardinality, also informally called its size.
Finite sequences will be denoted by $\bar{s}$ for some letter $s$ and the length of 
$\bar{s}$ will be denoted by $|\bar{s}|$.

If a sequence of logical variables is denoted by $\bar{x}$ (say) then it is assumed that all entries in $\bar{x}$ are different.
By a {\em signature} we mean a finite set of relation symbols.
If $\sigma$ is a signature then $FO(\sigma)$ denotes the set of first-order formulas that 
can be formed by using symbols in $\sigma$
and the identity symbol `$=$'.
A first-order formula is called {\em quantifier-free} if it contains no quantifier, or equivalently, if it is
constructed by using only from atomic formulas and the connectives $\neg, \wedge, \vee, \rightarrow$, and $\leftrightarrow$.
A formula that is constructed from other formulas and (only) $\neg, \wedge, \vee, \rightarrow$, and $\leftrightarrow$
(applied as many times as one likes)
is called a {\em boolean combination} of those other formulas.
If a formula in $FO(\sigma)$ is denoted by $\varphi(\bar{x})$ (where $\bar{x}$ is a finite sequence of logical variables)
then it is assumed that $\bar{x}$ lists the free variables that occur in $\varphi(\bar{x})$.
First-order structures will be denoted by calligraphic letters $\mcA$, $\mcB$, etc.
If $\varphi(x_1, \ldots, x_k) \in FO(\sigma)$ and $\mcA$ is a $\sigma$-structure with domain $A$, then let
\[
\varphi(\mcA) := \{ (a_1, \ldots, a_k) \in A^k : \mcA \models \varphi(a_1, \ldots, a_k)\}.
\]
If $k = 1$ then any 1-tuple $(a)$ is identified with the single coordinate `$a$' in it.
So if $k=1$ then $\varphi(\mcA) := \{a \in A : \mcA \models \varphi(a)\}$.
If $R \in \sigma$ then $R^\mcA$ denotes its interpretation in the $\sigma$-structure $\mcA$,
and it follows that $R(\mcA) = R^\mcA$.
In particular, if $R \in \sigma$ has arity 1 then $R(\mcA)$ and $R^\mcA$ denote the set of $a \in A$
such that $\mcA \models R(a)$.
In some situations it is convenient to view $R^\mcA$ as a function from $A^k$ to $\{0, 1\}$ (if $R$ has arity $k$)
defined by $R^\mcA(a_1, \ldots, a_k) = 1$ if $\mcA \models R(a_1, \ldots, a_k)$ and 
$R^\mcA(a_1, \ldots, a_k) = 0$ otherwise.

Once a signature $\sigma$ has been fixed then we let, for all $n \in \mbbN^+$, $\mbW_n$ be the the set of
all $\sigma$-structures (``possible worlds'') with domain $[n]$.

\subsection{Markov logic networks}

Intuitively speaking, a Markov logic network (first conceived in \cite{RD}) 
is determined by a finite number of weighted ``soft-constraints'' 
which are statements, for example 
\[
(R(x, y) \wedge P(x)) \rightarrow P(y),
\] 
that may be violated by tuples 
(2-tuples in case of the example)
of elements from the domain of
a possible world, but if more tuples 
violate a soft constraint in a possible world then that world is less likely, but still
has positive probability. The weight of a soft constraint quantifies how much violations of the soft constraint are penalized.
The formal definition is as follows:

\begin{defin}\label{definition of MLN}{\rm
A {\em Markov logic network (MLN)} over a signature $\sigma$
is a (finite) set 
\[
\mbbM=\{(\varphi_i(\bar x_i), w_i) :  1\le i\le t\}
\]
where $\bar{x}_i = (x_1, \ldots, x_{k_i})$, $\varphi_i(\bar{x}_i) \in FO(\sigma)$
and $w_i \geq 0$ is a real number.
We call each $\varphi_i(\bar{x}_i)$ a {\em soft constraint} of $\mbbM$ and we call $w_i$ the {\em weight} of $\varphi_i(\bar{x}_i)$.

For every $n \in \mbbN^+$ and $\mcA \in \mbW_n$ we define
\begin{align*}
&\mbbP_n^\mbbM(\mcA)
:=
\frac{
\exp\bigl(\sum_{i=1}^t w_i |\varphi_i(\mcA)|\bigr)
}{
\sum_{\mcB\in\mbW_n}
\exp\bigl(\sum_{i=1}^t w_i |\varphi_i(\mcB)|\bigr)
}
\end{align*}
where (as follows from Section~\ref{Notation and terminology})
$|\varphi_i(\mcA)|$ is the number of $k_i$-tuples of $(a_1, \ldots, a_{k_i}) \in [n]^{k_i}$ such that
$\mcA \models \varphi_i(a_1, \ldots, a_{k_i})$.
}\end{defin}

\noindent
Due to the division by $\sum_{\mcB\in\mbW_n} \exp\bigl(\sum_{i=1}^t w_i |\varphi_i(\mcB)|\bigr)$ above we
have $\sum_{\mcA \in \mbW_n} \mbbP_n^\mbbM(\mcA) = 1$ and it follows
that $\mbbP_n^\mbbM$ is a probability distribution on $\mbW_n$.
One can show (as has been done before) that even if we would allow the weights to be negative
we would not be able to define more distributions on $\mbW_n$ by the above construction. 
Neither does the choice of base (in this case $e$) influence what distributions we can define on $\mbW_n$
by the above definition, provided that the base is larger than 1.

We are leaning towards believing that there is no general and reasonably
simple theory of convergence phenomena for all MLNs, partly because 
convergence phenomena for various kinds of hard constraints have been studied in finite model theory and combinatorics
with specific outcomes by specific methods for the different types of hard constraints considered, 
and it seems unlikely to us that ``softifying'' the constraints would simplify the situation that much.
Below we take a step in this direction by showing that already for a simple signature and rather simple soft constraints
all three cases (1) -- (3) above do occur, even if we fix the soft constraints that we consider and only let their weights vary. 
We also show that it is computationally easy to determine, for the
the MLNs $\mbbM$ considered here, to which of the three cases $(\mbbP_n^\mbbM : n \in \mbbN^+)$ belongs.
But this does not mean that for all other MLNs than those considered here it is easy to determine
in which of the three cases we are. We aim at showing this in future studies.
We will also see that ever so small changes of the weights may move $(\mbbP_n^\mbbM : n \in \mbbN^+)$
from one of the three cases above to another one.
Furhermore, we show that if we scale the weights by the domain size then we always end up in case~(1)
for the MLNs that we consider.

\subsection{Related work}\label{Related work}

In \cite{KopMLN} Koponen studies various aspects of domain size asymptotics of MLNs, in particular the influence
in the limit of the weights of soft constraints. 
But no convergence law for all first-order sentences is proved, or refuted, in \cite{KopMLN}.
Malhotra and Serafini \cite{MS} have studied projective MLNs \cite{JS} with relation symbols of arity at most 2 and with
only soft constraints with at most 2 free variables. But they do not consider convergence laws of first-order logic
and the MLNs considered here are only projective in special cases.

It has been shown, first by Poole et al. \cite{Poole} and then more generally 
(and with some corrections) by Mittal et al. \cite{Mittal} that for some MLNs $\mbbM$ there is an atomic formula
$Q(x)$ such that $\lim_{n\to\infty} \mbbP_n^\mbbM(Q(1))$ exists and is independent of the weights of $\mbbM$.
Thus, for the specific MLNs considered, weight of the single soft constraint considered
has negligible influence on the probability of $Q(1)$ being true for large $n$.
To avoid that weights may lose influence in the limit,
variants of MLNs have been proposed,
such as Adaptive MLNs by Jain et al. \cite{JBB} and Domain-size Aware MLNs by Mittal et al. \cite{Mittal},
which scale the weights by dividing them by some number that depends on the domain size.
But Weitkämper \cite{Wei25} has recently shown that there are scaled MLNs $\mbbM$ such that the asymptotic
probability of some atomic ground formula is independent of the weights of $\mbbM$, so the scaling does not
help in general. 
Our 0-1 law for first-order logic in the case of scaled weights 
(Theorem~\ref{0-1 law in scaled and general case}) 
also shows that scaling does not necessarily give the weights more influence
on limit probabilities.
None of the above cited studies proves or refutes a complete convergence law for first-order logic and some MLN
(scaled or not), but they only consider convergence in probability of a specific atomic formula.
The study \cite{ChenWM}  considers the question about knowledge transfer for MLNs but
is not directly concerned with convergence of probabilities of first-order formulas.

For {\em directed} graphical models (unlike MLNs which are undirected), and various logics, there are a number of 
quite general logical convergence results, and related results. This includes
relational Bayesian networks \cite{Jae98a},
relational Bayesian network specifications \cite{CM},
lifted Bayesian networks \cite{Kop20, KW1},
probabilistic logic programs \cite{Wei21},
PLA-networks \cite{KW2, KT, Kop26}, and
functional lifted Bayesian networks \cite{Wei24}.
The proofs of these results rely on the fact that it is possible to use induction on 
the longest directed path in the underlying directed acyclic graph of the directed graphical model.
There are also convergence laws on graph neural networks
\cite{Adam-Day1, Adam-Day2, Adam-Day3}
and multilayer feedforward neural networks \cite{KopNN}.

\section{Random coloured digraphs and the main results}

\noindent
{\em In the rest of this article we fix a signature $\sigma = \{P, R\}$ where $P$ and $R$ have arities 1 and 2, respectively.
Let $\mbW_n$ be the set of all $\sigma$-structures with domain $[n]$.}
Informally we can think of $P$ as a colour, so elements satisfying $P$ are coloured and others are uncoloured,
and $R$ as a directed edge relation (with loops allowed).
We first state three relatively nontechnical theorems. 
All of them follow from the more technical and detailed analysis that will follow.

\begin{theor}\label{all three cases can be realized}
There are MLNs $\mbbM_1$, $\mbbM_2$, and $\mbbM_3$ over $\sigma$ such that all soft constraints of each $\mbbM_k$,
$k = 1, 2, 3$, are boolean combinations of $P(x)$, $P(y)$ and $R(x, y)$ in which both $x$ and $y$ occur, and 
\begin{enumerate}
\item $(\mbbP_n^{\mbbM_1} : n \in \mbbN^+)$ satisfies a 0-1 law for first-order logic, 
\item $(\mbbP_n^{\mbbM_2} : n \in \mbbN^+)$ satisfies a convergence law, but not 0-1 law, for first-order logic, and
\item $(\mbbP_n^{\mbbM_3} : n \in \mbbN^+)$ does not satisfy a convergence law for first-order logic.
\end{enumerate}
Moreover, for every $\varepsilon > 0$, 
we can find $\mbbM_1$, $\mbbM_2$ and $\mbbM_3$ as above and such that all of them have exactly the same soft
constraints, and if $\varphi(x, y)$ is any one of these common soft constraints and $w_1, w_2$, respectively $w_3$ is
the weight of $\varphi(x, y)$ in $\mbbM_1$, $\mbbM_2$, respectively $\mbbM_3$, then the numbers
$|w_1 - w_2|$, $|w_1 - w_3|$, and $|w_2 - w_3|$ are less than $\varepsilon$. 
\end{theor}

\begin{proof}
This follows from 
Theorem~\ref{convergence laws} and 
Example~\ref{example of cases}.
\end{proof}

\noindent
Thus, for every $\varepsilon > 0$, there is an MLN $\mbbM$ such that $(\mbbP_n^\mbbM : n \in \mbbN^+)$
can move from one of the cases in Theorem~\ref{all three cases can be realized} to any of the other two
by changing its weights by less than $\varepsilon$.

\begin{theor}\label{procedure for determining the case}
Let
$\mbbM = \big\{ \big(\varphi_1(x, y), w_1\big), \ldots, \big(\varphi_t(x, y), w_t\big) \big\}$
be an MLN over $\sigma$ where each $\varphi_i(x, y)$, $i = 1, \ldots, t$, is a
boolean combination of $P(x)$, $P(y)$ and $R(x, y)$ that involves both variables.
Then there is a procedure which uses only the data given by $\mbbM$ and which determines to which
of the cases (1) -- (3) in Theorem~\ref{all three cases can be realized} $(\mbbP_n^\mbbM : n \in \mbbN^+)$ belongs.
\end{theor}

\begin{proof}
This follows from
Lemma~\ref{getting a canonical MLN},
the definition of the polynomial $\eff$ in Section~\ref{canonical MLNs} (which depends only on $\mbbM$),
and Theorem~\ref{convergence laws}.
\end{proof}

\noindent
Roughly speaking the procedure mentioned in 
Theorem~\ref{procedure for determining the case} is as follows:
Given $\mbbM$ as in Theorem~\ref{procedure for determining the case},
transform it to a ``canonical'' MLN $\mbbM'$ with exactly 8 soft constraints, according to the proof of 
Lemma~\ref{getting a canonical MLN} below. 
Then $\mbbP_n^{\mbbM} = \mbbP_n^{\mbbM'}$ for all $n$ as stated in 
Lemma~\ref{getting a canonical MLN}.
From the weights of the 8 soft constraints of $\mbbM'$ one constructs a polynomial $\eff(\alpha)$ of degree at most 2.
Next one finds the (at most two) $\alpha \in [0, 1]$ which maximize $f$ when restricted to $[0, 1]$, and for such $\alpha$
compute $\eff'(\alpha)$. The maximizers $\alpha$ of $\eff$ restricted to $[0, 1]$ and the sign of $\eff'(\alpha)$
will determine to which of the three cases (1) -- (3) $(\mbbP_n^\mbbM : n \in \mbbN^+)$ belongs,
as stated by Theorem~\ref{convergence laws} below.

The next result gives a certain limitation of the influence of the weights on sequences
of the form $(\mbbP_n^\mbbM(\psi) : n \in \mbbN^+)$ where $\psi$ is a sentence and $\mbbM$ an MLN
of the kind considered here.

\begin{theor}\label{limit of sentence depends only on the case}
Let
\begin{align*}
&\mbbM = \big\{ \big(\varphi_1(x, y), w_1\big), \ldots, \big(\varphi_t(x, y), w_t\big) \big\}
\quad \text{ and } \\ 
&\mbbM' = \big\{ \big(\varphi'_1(x, y), w'_1\big), \ldots, \big(\varphi'_t(x, y), w'_s\big) \big\}
\end{align*}
be  MLNs over $\sigma$ where each $\varphi_i(x, y)$, $i = 1, \ldots, t$, 
and each $\varphi'_i(x, y)$, $i = 1, \ldots, s$, is a boolean combination of $P(x)$, $P(y)$ and $R(x, y)$.
Suppose that 
\begin{enumerate}
\item both $(\mbbP_n^\mbbM : n \in \mbbN^+)$ and $(\mbbP_n^{\mbbM'} : n \in \mbbN^+)$ satisfy a 0-1 law
for first-order logic, $\lim_{n\to\infty} \mbbP_n^\mbbM(\exists x P(x)) = \lim_{n\to\infty} \mbbP_n^{\mbbM'}(\exists x P(x))$,
and $\lim_{n\to\infty} \mbbP_n^\mbbM(\exists x \neg P(x))$ \\
$= \lim_{n\to\infty} \mbbP_n^{\mbbM'}(\exists x \neg P(x))$, or that
\item both $(\mbbP_n^\mbbM : n \in \mbbN^+)$ and $(\mbbP_n^{\mbbM'} : n \in \mbbN^+)$ satisfy a convergence law
for first-order logic but none of them satisfy a 0-1 law for first-order logic.
\end{enumerate}
Then, for every first-order sentence $\psi \in FO(\sigma)$, 
$\lim_{n\to\infty} \mbbP_n^\mbbM(\psi) = \lim_{n\to\infty} \mbbP_n^{\mbbM'}(\psi)$.
\end{theor}

\begin{proof}
This follows from
Lemma~\ref{getting a canonical MLN}
and Theorem~\ref{limits depend only on cases}.
\end{proof}

\noindent
Theorems~\ref{procedure for determining the case} and~\ref{limit of sentence depends only on the case}
together imply that the soft constraints and weights of an MLN $\mbbM$ of the kind considered determines 
to which of the cases (1) -- (3) $(\mbbP_n^\mbbM : n \in \mbbN^+)$ belongs, but in cases (1) and (2) above
the soft constraints and weights have {\em no} further influence on the limit probabilities of first-order sentences,
besides determining $\lim_{n\to\infty} \mbbP_n^\mbbM(\exists x P(x))$ 
and $\lim_{n\to\infty} \mbbP_n^\mbbM(\exists x \neg P(x))$ in case~(1) above.

Earlier studies have considered ``scalings'' of the weights with the domain size raized to some nonnegative integer
\cite{JBB, Mittal, Wei25}.
If we use the scalings considered in \cite{Mittal, Wei25} to an MLN  $\mbbM$ 
such that its soft constraints are  boolean combinations of $P(x)$, $P(y)$ and $R(x, y)$ in which all three atomic
formulas occur as subformulas,
then we get the
probability distribution $\widetilde{\mbbP}_n^\mbbM$ on $\mbW_n$ defined below, where all weights are
scaled by $1/n$ where $n$ is the domain size.

\begin{defin}\label{definition of scaled distribution}{\rm
Let $\mbbM =\{(\varphi_i(\bar x_i), w_i) :  1\le i\le t\}$ be an MLN.
For all $n \in \mbbN^+$ and $\mcA \in \mbW_n$, define
\[
\widetilde{\mbbP}_n^\mbbM(\mcA) :=
\frac{\exp\bigl(\frac{1}{n}\sum_{i=1}^t w_i |\varphi_i(\mcA)|\bigr)}
{\sum_{\mcB\in\mbW_n}
\exp\bigl(\frac{1}{n}\sum_{i=1}^t w_i |\varphi_i(\mcB)|\bigr)}
\]
}\end{defin}

\begin{theor}\label{0-1 law in scaled and general case}
Let
$\mbbM = \big\{ \big(\varphi_1(x, y), w_1\big), \ldots, \big(\varphi_t(x, y), w_t\big) \big\}$
be an MLN over $\sigma$ where each $\varphi_i(x, y)$, $i = 1, \ldots, t$, is a boolean combination of
$P(x)$, $P(y)$ and $R(x, y)$.
Then $(\widetilde{\mbbP}_n^\mbbM : n \in \mbbN^+)$ satisfies a 0-1 law for first-order logic.
Moreover, for every first-order sentence $\psi \in FO(\sigma)$ the limit 
$\lim_{n\to\infty} \widetilde{\mbbP}_n^\mbbM(\psi)$ is independent of the soft constraints and weights of $\mbbM$.
In other words, either 
$\lim_{n\to\infty} \widetilde{\mbbP}_n^\mbbM(\psi) = 0$ for all $\mbbM$ as above, or 
$\lim_{n\to\infty} \widetilde{\mbbP}_n^\mbbM(\psi) = 1$
for all $\mbbM$ as above.
\end{theor}

\begin{proof}
This follows from
Lemma~\ref{getting a canonical MLN}
and Theorem~\ref{0-1 law in scaled case}.
\end{proof}

\noindent
Theorem~\ref{0-1 law in scaled and general case} tells that as $n \to \infty$
the weights completely loose influence on probabilities of first-order sentences, so the
situation becomes simpler than in the unscaled case of the previous results.
The technical proofs also show that for every MLN $\mbbM$ as in 
Theorem~\ref{0-1 law in scaled and general case} and for any pair $(a, b) \in [n]^2$ the probability,
using $\widetilde{\mbbP}_n^\mbbM$, that $R(a, b)$ holds
tends to 1/2 as $n \to \infty$. 
In \cite{JBB, Mittal, Wei25} scaling of weights (of various kinds) were introduced in order to give the weights influence on 
limit probabilities, but here we see that this is not always the case, as was also noted by
Weitkämper \cite{Wei25}.

As mentioned above one can transform every MLN over $\sigma$ such that all of its soft constraints
are boolean combinations of $P(x)$, $P(y)$ and $R(x, y)$ in which both variables occur
into a ``canonical'' MLN in the sense defined below:

\begin{defin}\label{definition of canonical MLN} {\bf (Canonical MLN)} {\rm
For the rest of the article let
\begin{align*}
\varphi_1(x,y)  &:=  P(x)\land  P(y) \land  R(x,y),\\
\varphi_2(x,y)  &:=  P(x)\land  P(y) \land \neg  R(x,y),\\
\varphi_3(x,y)  &:=  \neg  P(x) \land \neg  P(y) \land  R(x,y),\\
\varphi_4(x,y)  &:=  \neg  P(x) \land \neg  P(y) \land \neg  R(x,y),\\
\varphi_5(x,y)  &:=  P(x) \land \neg  P(y) \land  R(x,y), \\
\varphi_6(x,y)  &:=  P(x) \land \neg  P(y) \land \neg  R(x,y), \\
\varphi_7(x,y)  &:=  \neg  P(x) \land  P(y) \land  R(x,y), \\
\varphi_8(x,y)  &:=  \neg  P(x)\land  P(y) \land \neg  R(x,y).
\end{align*}
By a {\em canonical} MLN we mean (in this particular study) an MLN of the form
\[
\mbbM = \{(\varphi_i(x, y), w_i) : i = 1, \ldots, 8\}
\]
for any choice of weights $w_1, \ldots, w_8 \geq 0$.
}\end{defin}

\begin{lem}\label{getting a canonical MLN}
Let
$\mbbM = \big\{ \big(\theta_1(x, y), w_1\big), \ldots, \big(\theta_t(x, y), w_t\big) \big\}$
be an MLN over $\sigma$ where each $\theta_i(x, y)$, $i = 1, \ldots, t$, is a boolean combination
of $P(x)$, $P(y)$ and $R(x, y)$.
Then there is a canonical MLN $\mbbM'$ such that $\mbbP_n^\mbbM = \mbbP_n^{\mbbM'}$ for all $n \in \mbbN^+$.
\end{lem}

\begin{proof}
Let $\mbbM = \big\{ \big(\theta_1(x, y), w_1\big), \ldots, \big(\theta_t(x, y), w_t\big) \big\}$
be an MLN over $\sigma$ where each $\theta_i(x, y)$, $i = 1, \ldots, t$, is a boolean combination
of $P(x)$, $P(y)$ and $R(x, y)$.
We show how to transform $\mbbM$ into a canonical MLN $\mbbM'$ that determines the same distribution on $\mbW_n$.
Without loss of generality we may assume that each $\theta_i(x, y)$ is satisfiable in some $\sigma$-structure,
because otherwise we can remove $\big(\theta_i(x, y), w_i\big)$ and the new MLN determines the
same distribution as the original one.

For each $i = 1, \ldots, 8$, let $v_i$ the the sum of all $w_j$ ($j = 1, \ldots, t$) such that
\[
\varphi_i(x, y) \quad \text{ implies } \quad \theta_j(x, y).
\]
If no such $j$ exists then $v_i = 0$.
Now $\mbbM' :=  \{(\varphi_i(x, y), v_i) : i = 1, \ldots, 8\}$ is a canonical MLN.
We claim that $\mbbP_n^\mbbM = \mbbP_n^{\mbbM'}$ for all $n$.
Note that for any $\mcA \in \mbW_n$ and $(a, b) \in [n]$, there is a unique $i \in \{1, \ldots, 8\}$ 
such that $\mcA \models \varphi_i(a, b)$.
Take any $i \in \{1, \ldots, 8\}$. Then $v_i = w_{j_1} + \ldots + w_{j_l}$ where $j_1, \ldots, j_l$ enumerates
(without repetition) all indices $j$ such that $\varphi_i(x, y)$ implies $\theta_j(x, y)$.
Every $(a, b) \in [n]$ such that $\mcA \models \varphi_i(a, b)$ contributes with 
weight $v_i$ according to $\mbbM'$, and the same $(a, b)$ contributes with weight $w_{j_1} + \ldots + w_{j_l}$
according to $\mbbM$. As $v_i = w_{j_1} + \ldots + w_{j_l}$ we get $\mbbP_n^\mbbM = \mbbP_n^{\mbbM'}$.
\end{proof}

\section{Detailed results for canonical MLNs}\label{canonical MLNs}

\noindent
By Lemma~\ref{getting a canonical MLN}, in order to study all
MLNs with soft contraints being boolean combinations of $P(x)$, $P(y)$ and $R(x, y)$ in which 
both variables occur it suffices to study all
canonical MLNs, so below we restrict our attention to canonical MLNs in the sense of 
Definition~\ref{definition of canonical MLN}.
Indeed, we need to consider canonical MLNs to clearly see the different cases that we can end up in.
Note that the only thing that distinguishes two different canonical MLNs are their respective 8 weights,
so we essentially consider properties of a canonical MLN $\mbbM$ as its 8-tuple of weights moves
in $\big(\mbbR^{\geq 0}\big)^8$.

Fix an arbitrary canonical MLN 
\[
\mbbM := \big\{\big(\varphi_1(x, y), w_1\big), \ldots, \big(\varphi_8(x, y), w_8\big) \big\}.
\]
For the rest of the article we make the following abbreviation
\[
\mbbP_n := \mbbP_n^\mbbM.
\]
Using the weights of $\mbbM$ we will define, in a few steps, a polynomial $\eff$ the properties of which will
determine the asymptotic behaviour of $(\mbbP_n : n \in \mbbN^+)$.

\begin{notation}\label{notation C-ij}{\rm Let
\begin{align*}
&C_{12}:=e^{w_1}+e^{w_2},\qquad C_{34}:=e^{w_3}+e^{w_4},  \\
&C_{56}:=e^{w_5}+e^{w_6} ,
\  \text{and} \ \ C_{78}:=e^{w_7}+e^{w_8}.
\end{align*}
}\end{notation}

\begin{defin}\label{definition of phase function}{\rm
(a) For $\alpha\in \mbbR$ define the {\em phase function} $\eff$ by
\begin{equation}\label{eq:def-psi}
\eff(\alpha) :=  \alpha^2\ln C_{12}+(1-\alpha)^2\ln C_{34}
+\alpha(1-\alpha)\ln (C_{56}C_{78}).
\end{equation}
(b) Define the following constants:
\begin{align*}
&c_0 := \ln C_{34},\\
&c_1 := \ln (C_{56}C_{78}) - 2\ln C_{34} = \ln \frac{C_{56}C_{78}}{(C_{34})^2}, \text{ and}\\
&c_2 := \ln C_{12} + \ln C_{34} - \ln (C_{56}C_{78}) = \ln \frac{C_{12}C_{34}}{C_{56}C_{78}}.
\end{align*}
}\end{defin}

\noindent
Thus $\eff(\alpha)$ is a polynomial of degree at most 2 and by straightforward calculations we get
\begin{align}\label{normal form of phase function}
&\eff(\alpha) = c_2 \alpha^2 + c_1 \alpha + c_0. 
\end{align}

\begin{defin}\label{definition of maximiser set}{\rm
Define $\Eff^{\max} = \underset{\alpha\in[0,1]}{\arg\max} \, \eff(\alpha)$,
or equivalently, if $\beta = \max\{\eff(\alpha) : \alpha \in [0, 1]\}$ 
(where the maximum exists since  $\eff$ is continous and $[0, 1]$ is a closed interval)
then
$\Eff^{\max} = \{\alpha \in [0, 1] : \eff(\alpha) = \beta\}$.
}\end{defin}

\subsection{Results}

We first state how properties of $\eff$ influence the most likely bounds
of the number coloured elements in the limit,
because these bounds will determine the other limit properties of consideration in this study.

\begin{theor}\label{the colour distribution, exact formulation} 
{\bf (Asymptotic colour distribution)}
\begin{enumerate}
\item If $\Eff^{\max} = [0, 1]$ (meaning that $\eff$ is constant)
then, for all $\varepsilon > 0$,
\[
\lim_{n\to\infty} \mbbP_n \Big(
  \big\{\mcA \in \mb W_n :
    (1/2 - \varepsilon)n
      \le \big|P^\mcA\big|
      \le (1/2 + \varepsilon)n
  \big\}
\Big) = 1.
\]

\item If $\Eff^{\max} = \{\alpha^\star\}$ where $\alpha^\star \in (0, 1)$, 
then, for all $\varepsilon > 0$,
\[
\lim_{n\to\infty} \mbbP_n \Big(
  \big\{\mcA \in \mb W_n :
    (\alpha^\star - \varepsilon)n
      \le \big|P^\mcA\big|
      \le (\alpha^\star + \varepsilon)n
  \big\}
\Big) = 1.
\]

\item If $\Eff^{\max} = \{0\}$ and $\eff'(0) = 0$ 
then, for all $\varepsilon \in (0, \tfrac{1}{2})$,
\begin{equation*}
\begin{aligned}
\lim_{n\to\infty} \mbbP_n \Big(
  \big\{\mcA \in \mb W_n :
    \big\lfloor\tfrac{1 - \varepsilon}{2|c_2|}\ln n \big\rfloor
      \le \big|P^\mcA\big|
      \le \big\lceil \tfrac{1 + \varepsilon}{2|c_2|}\ln n \big\rceil
  \big\}
\Big)
\ = \ 1.
\end{aligned}
\end{equation*}

\item If  $\Eff^{\max} = \{1\}$ and $\eff'(1) = 0$ 
then, for all $\varepsilon \in (0, \tfrac{1}{2})$,
\begin{align*}
\lim_{n\to\infty} \mbbP_n \Big(
  \big\{\mcA \in \mbW_n : \ 
    &n - \big\lceil \tfrac{1 + \varepsilon}{2|c_2|}\ln n \big\rceil
      \le \big|P^\mcA\big| \le \\
     &n - \big\lfloor \tfrac{1 - \varepsilon}{2|c_2|}\ln n \big\rfloor
  \big\}
\Big)
\ = \ 1.
\end{align*}

\item If $\Eff^{\max} = \{0\}$ and $\eff'(0) < 0$, then 
$\lim_{n\to\infty} \mbbP_n\big(\mbW_n(0)\big) = 1$. 

\item If $\Eff^{\max} = \{1\}$ and $\eff'(1) > 0$, then 
$\lim_{n\to\infty} \mbbP_n\big(\mbW_n(n)\big) = 1$.

\item If $\Eff^{\max} = \{0, 1\}$, then 
\[
\lim_{n\to\infty} \mbbP_n\big(\mbW_n(0)\big) = \lim_{n\to\infty} \mbbP_n\big(\mbW_n(n)\big) = 1/2.
\]
\end{enumerate}
\end{theor}

\begin{proof}
Part~(1) is given by Lemma~\ref{colour proportion when psi is constant}.
Part~(2) is given by Proposition~\ref{thm:concentration}.
Parts~(3) and~(4) are given by Lemma~\ref{lem:caseC-window-rigorous}.
Parts~(5) and~(6) are given by Lemma~\ref{lem:pinning}.
Part~(7) is given by Lemma~\ref{distribution of P when 0 and 1 are maxpoints}.
\end{proof}

\noindent
Figure~\ref{fig:seven-phase-regimes}
illustrates (schematically) the shape of $\eff$ in each case of 
Theorem~\ref{the colour distribution, exact formulation}.

\begin{figure}[t]
\centering
\begin{tikzpicture}[
  x=1cm,
  y=1cm,
  >=stealth,
  font=\footnotesize,
  line cap=round,
  line join=round,
  panel/.style={draw=black!55, rounded corners=2pt, line width=0.35pt},
  axis/.style={->, line width=0.45pt},
  curve/.style={line width=1pt},
  guide/.style={densely dashed, line width=0.35pt},
  title/.style={font=\footnotesize\bfseries, align=center},
  hyp/.style={font=\scriptsize, align=center, text width=3.45cm, inner sep=1pt},
  conclusion/.style={font=\scriptsize, align=center, text width=3.45cm, inner sep=1pt},
  maxdot/.style={circle, fill=black, inner sep=1.15pt}
]

\def\panelbox{%
  \draw[panel] (-0.50,-1) rectangle (3.70,1.76);
}
\def\axes{%
  \draw[axis] (0,0) -- (3.25,0) node[right] {$\alpha$};
  \draw[axis] (0,0) -- (0,1.12);
  \node[anchor=east] at (-0.04,1.03) {$\eff$};
  \node[below] at (0,0) {$0$};
  \node[below] at (3,0) {$1$};
}

\begin{scope}[shift={(0,0)}]
  \panelbox
  \axes
  \node[title] at (1.60,2) {Case (1)};
  \node[hyp] at (1.60,1.24) {$\Eff^{\max}=[0,1]$};
  \draw[curve] (0,0.78) -- (3,0.78);
  \draw[guide] (1.5,0) -- (1.5,0.78);
\end{scope}

\begin{scope}[shift={(6.55,0)}]
  \panelbox
  \axes
  \node[title] at (1.60,2) {Case (2)};
  \node[hyp] at (1.60,1.18)
    {$\Eff^{\max}=\{\alpha^\star\}$};
  \draw[curve, domain=0:3, samples=100, smooth, variable=\x]
    plot ({\x},{0.96-0.18*(\x-1.72)*(\x-1.72)});
  \node[maxdot] at (1.72,0.96) {};
  \draw[guide] (1.72,0) -- (1.72,0.96);
  \node[below] at (1.72,0) {$\alpha^\star$};
\end{scope}

\begin{scope}[shift={(0,-4)}]
  \panelbox
  \axes
  \node[title] at (1.60,2) {Case (3)};
  \node[hyp] at (1.60,1.24)
    {$\Eff^{\max}=\{0\}$,\quad $\eff'(0)=0$};
  \draw[curve, domain=0:3, samples=100, smooth, variable=\x]
    plot ({\x},{0.95-0.105*\x*\x});
  \node[maxdot] at (0,0.95) {};
\end{scope}

\begin{scope}[shift={(6.55,-4)}]
  \panelbox
  \axes
  \node[title] at (1.60,2) {Case (4)};
  \node[hyp] at (1.60,1.24)
    {$\Eff^{\max}=\{1\}$,\quad $\eff'(1)=0$};
  \draw[curve, domain=0:3, samples=100, smooth, variable=\x]
    plot ({\x},{0.95-0.105*(3-\x)*(3-\x)});
  \node[maxdot] at (3,0.95) {};
\end{scope}

\begin{scope}[shift={(0,-8)}]
  \panelbox
  \axes
  \node[title] at (1.60,2) {Case (5)};
  \node[hyp] at (1.60,1.24)
    {$\Eff^{\max}=\{0\}$,\quad $\eff'(0)<0$};

  \draw[curve, domain=0:3, samples=120, smooth, variable=\x]
    plot ({\x},{0.95 - 0.15*\x - 0.035*\x*\x});

  \node[maxdot] at (0,0.95) {};

\end{scope}

\begin{scope}[shift={(6.55,-8)}]
  \panelbox
  \axes
  \node[title] at (1.60,2) {Case (6)};
  \node[hyp] at (1.60,1.24)
    {$\Eff^{\max}=\{1\}$,\quad $\eff'(1)>0$};

  \draw[curve, domain=0:3, samples=120, smooth, variable=\x]
    plot ({\x},{0.95 - 0.15*(3-\x) - 0.035*(3-\x)*(3-\x)});

  \node[maxdot] at (3,0.95) {};
  
\end{scope}

\begin{scope}[shift={(3.275,-12)}]
  \panelbox
  \axes
  \node[title] at (1.60,2) {Case (7)};
  \node[hyp] at (1.60,1.24) {$\Eff^{\max}=\{0,1\}$};
  \draw[curve, domain=0:3, samples=100, smooth, variable=\x]
    plot ({\x},{0.36+0.25*(\x-1.5)*(\x-1.5)});
  \node[maxdot] at (0,0.9225) {};
  \node[maxdot] at (3,0.9225) {};
\end{scope}

\end{tikzpicture}
\caption{Schematic shapes of the phase function $\eff$ in the seven regimes of
Theorem~\ref{the colour distribution, exact formulation}. }
\label{fig:seven-phase-regimes}
\end{figure}

\noindent
Below we state exactly how each case of 
Theorem~\ref{the colour distribution, exact formulation}
determines whether $(\mbbP_n^\mbbM : n \in \mbbN^+)$ satisfies a 0-1 law for first-order logic,
or a convergence law, but not 0-1 law, for first-order logic, or does not even satisfy a convergence law.

\begin{theor}\label{convergence laws} {\bf (Logical convergence laws)} 
\begin{enumerate}
\item If $\Eff^{\max} = [0, 1]$ 
then $(\mbbP_n : n \in \mbbN^+)$ satisfies 0-1 law for first-order logic.

\item If $\Eff^{\max} = \{\alpha^\star\}$ where $\alpha^\star \in (0, 1)$, 
then  $(\mbbP_n : n \in \mbbN^+)$ satisfies a 0-1 law for first-order logic.

\item If $\Eff^{\max} = \{0\}$ and $\eff'(0) = 0$ then
$(\mbbP_n : n \in \mbbN^+)$ does {\em not} satisfy a convergence law for first-order logic.

\item If $\Eff^{\max} = \{1\}$ and $\eff'(1) = 0$ then 
$(\mbbP_n : n \in \mbbN^+)$ does {\em not} satisfy a convergence law for first-order logic.

\item If $\Eff^{\max} = \{0\}$ and $\eff'(0) < 0$, then 
$(\mbbP_n : n \in \mbbN^+)$ satisfies a 0-1 law for first-order logic.

\item If $\Eff^{\max} = \{1\}$ and $\eff'(1) > 0$, then 
$(\mbbP_n : n \in \mbbN^+)$ satisfies a 0-1 law for first-order logic.

\item If $\Eff^{\max} = \{0, 1\}$ then 
$(\mbbP_n : n \in \mbbN^+)$ satisfies a convergence law,
but not 0-1 law, for first-order logic, and for 
every first-order sentence $\varphi$, $\lim_{n\to\infty}\mbbP_n(\varphi) \in \{0, \frac{1}{2}, 1\}$.
\end{enumerate}
\end{theor}

\begin{proof}
Parts~(1) and~(2) are given by Lemma~\ref{0-1 laws in cases 1 and 2}.
Parts~(3) and~(4) are proved in Appendix~\ref{proof of nonconvergence}.
Parts~(5) and~(6) are given by Lemma~\ref{0-1 laws in cases 5 and 6}.
Part~(7) is given by Lemma~\ref{convergence law in case 7}.
\end{proof}

\noindent
Theorems~\ref{the colour distribution, exact formulation} 
and~\ref{convergence laws} 
show that it is the most likely size of the set of coloured elements, as $n \to \infty$, that determines
whether $(\mbbP_n : n \in \mbbN^+)$ satisfies a 0-1 law, or convergence law but not 0-1 law, or not even a convergence law.

\begin{theor}\label{limits depend only on cases}
Suppose that $\mbbM$ and $\mbbM'$ are canonical MLNs over $\sigma$
such that both $(\mbbP_n^\mbbM : n \in \mbbN^+)$
and $(\mbbP_n^{\mbbM'} : n \in \mbbN^+)$ belong to the same case among the cases (1), (2), (5), (6), or (7)
of Theorem~\ref{convergence laws}.
Then, for every first order sentence $\psi \in FO(\sigma)$, 
$\lim_{n\to\infty}\mbbP_n^\mbbM(\psi) = \lim_{n\to\infty} \mbbP_n^{\mbbM'}(\psi)$
\end{theor}

\begin{proof}
This follows from the fact that the proofs of the corresponding convergence, or 0-1, laws in
lemmas~\ref{0-1 laws in cases 1 and 2}, \ref{0-1 laws in cases 5 and 6},
and~\ref{convergence law in case 7} only uses information about the almost sure, for large $n$,
proportion, or number, of coloured elements in the corresponding case
(as described in Theorem~\ref{the colour distribution, exact formulation}).
\end{proof}

\noindent
Figure~\ref{fig:uv-case-diagram} visualises 
the seven cases in 
Theorems~\ref{the colour distribution, exact formulation}
and~\ref{convergence laws} 
by considering a 2-dimensional plane that represents the 7 cases
by a translation of the 8 weights into only two coordinates.
The justification of the correspondence between the 7 cases ans the 2-dimensional plane is given by 
Lemma~\ref{prop:maximiser}
below.

\begin{figure}[t]
\centering
\includegraphics[width=\textwidth]{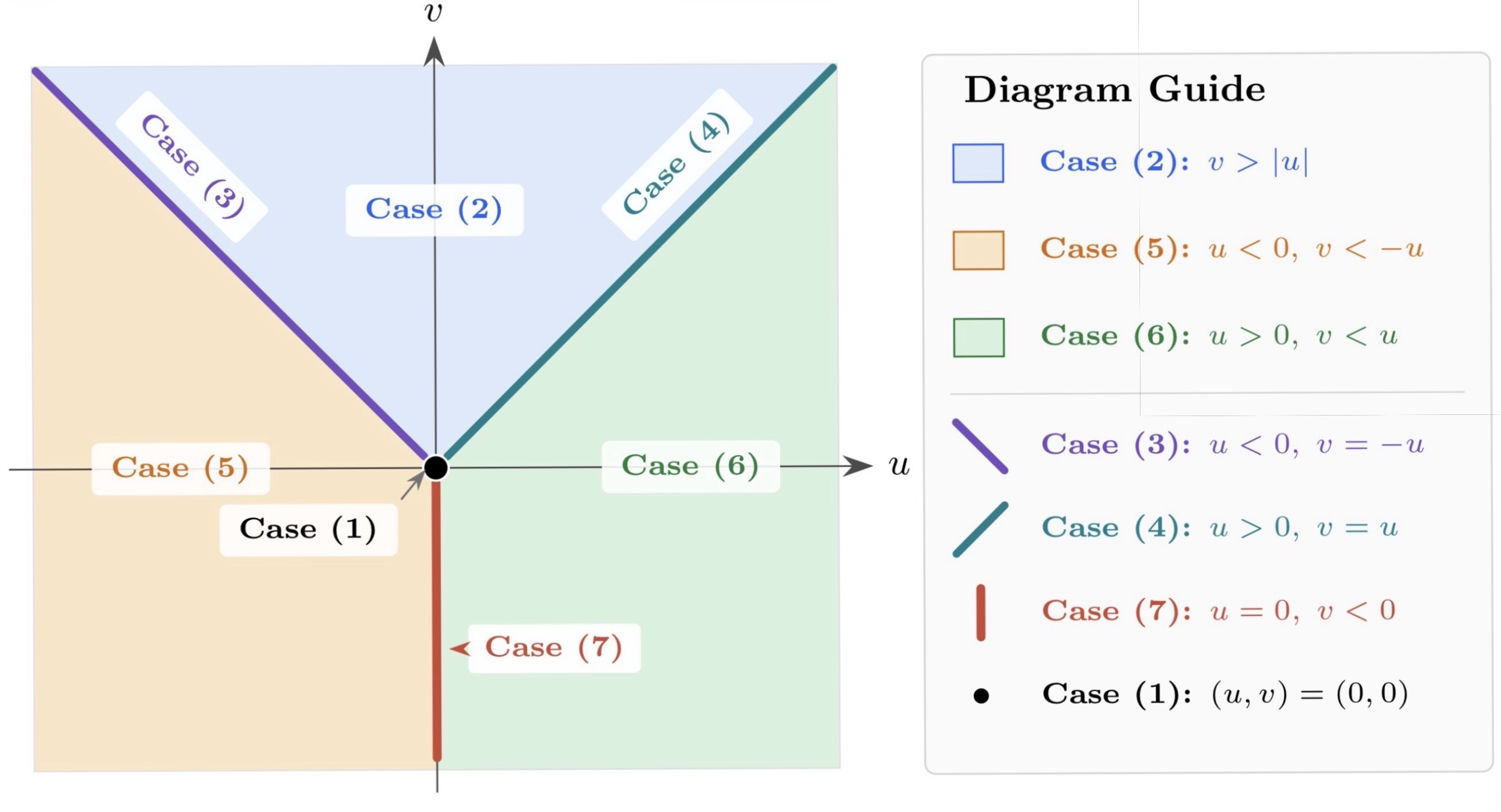}
\caption{
Let 
$u=\ln(C_{12}/C_{34})$ and
$v=\ln\big(C_{56}C_{78}/(C_{12}C_{34})\big)$.
With these coordinates we can visualise the seven cases
(1) -- (7) in
Theorems~\ref{the colour distribution, exact formulation}
and~\ref{convergence laws}
as regions in the plane of points $(u, v)$.
In particular, cases (3) and (4) correspond to the lines $v=-u$ and $v=u$
(where $v > 0$), respectively.
Case (1) corresponds to the point $(0, 0)$.}
\label{fig:uv-case-diagram}
\end{figure}

The next example illustrates how we can determine to which case $(\mbbP_n : n \in \mbbN^+)$ belongs
by turning the weights into $(u, v)$-coordinates according to 
Figure~\ref{fig:uv-case-diagram}.
It also illustrates that we can move from one case to another by an ever so small change of the weights.

\begin{exam}\label{example of cases}{\rm
Recall Notation~\ref{notation C-ij}.
As in Figure~\ref{fig:uv-case-diagram},
let 
$u=\ln(C_{12}/C_{34})$ and
$v=\ln\big(C_{56}C_{78}/(C_{12}C_{34})\big)$.
Let $\varepsilon >  0$ be as small as we like.

(a) Let $w_1 = w_2 = 2\varepsilon$,  $w_3 = w_4 = \varepsilon$, and 
$w_5 = w_6 = w_7 = w_8 = 2\varepsilon$.
Straightforward calculations now give 
$u = \varepsilon$ and $v = \varepsilon$.
According to 
Figure~\ref{fig:uv-case-diagram} 
we are in case (4) of  Theorem~\ref{convergence laws} so $(\mbbP_n : n \in \mbbN^+)$
does not satisfy a convergence law for first-order logic.

(b) Let $w_1 = w_2 = 2\varepsilon$, $w_3 = w_4 = \varepsilon$, and
$w_5 = w_6 = w_7 = w_8 = 3\varepsilon/2$.
Then $u = \varepsilon$ and $v = 0$.
According to 
Figure~\ref{fig:uv-case-diagram} 
 we are in case (6) of Theorem~\ref{convergence laws} so $(\mbbP_n : n \in \mbbN^+)$ satisfies a 0-1
law for first-order logic. 

(c) Let $w_1 = w_2 = w_3 = w_4 = 2\varepsilon$ and $w_5 = w_6 = w_7 = w_8 = 3\varepsilon/2$.
Then $u = 0$ and $v = -\varepsilon$.
According to Figure~\ref{fig:uv-case-diagram} 
 we are in case (7) of Theorem~\ref{convergence laws} so $(\mbbP_n : n \in \mbbN^+)$
satisfies a convergence law, but not 0-1 law, for first-order logic.
}\end{exam}

\noindent
These above examples show that we can move between convergence to only 0 or 1 and convergence in general and 
nonconvergence with ever so small changes of the weights.

\begin{lem}\label{prop:maximiser} {\bf (Justification of the $(u, v)$-coordinates in Figure~\ref{fig:uv-case-diagram})}
\begin{itemize}
\item[(a)] If $c_2 < 0$ 
(equivalently, $C_{12}C_{34} < C_{56}C_{78}$) then $\Eff^{\max}=\{\alpha^\star\}$ is a singleton and
\[
\alpha^\star=
\begin{cases}
-\frac{c_1}{2c_2}, & \text{if } C_{56}C_{78}>\max\{(C_{12})^2, (C_{34})^2\} ,\\
0, & \text{if } (C_{34})^2 \ge C_{56}C_{78},\\
1, & \text{if } (C_{12})^2 \ge C_{56}C_{78}.
\end{cases}
\]
\item[(b)] If $c_2 > 0$ (equivalently, $C_{12}C_{34} > C_{56}C_{78}$) then 
\[
\Eff^{\max}=
\begin{cases}
\{0\}, & \text{if } C_{34}>C_{12},\\
\{1\}, & \text{if } C_{12}>C_{34},\\
\{0,1\}, & \text{if } C_{12}=C_{34}.
\end{cases}
\]
\item[(c)] If $c_2 = 0$ (equivalently, $C_{12}C_{34} = C_{56}C_{78}$) then 
\[
\Eff^{\max}=
\begin{cases}
\{0\}, & \text{if } C_{56}C_{78} < (C_{34})^2,\\
\{1\}, & \text{if } C_{56}C_{78} > (C_{34})^2,\\
[0,1], & \text{if } C_{56}C_{78} = (C_{34})^2 \ (\text{equivalently } C_{12} = C_{34}).
\end{cases}
\]
\end{itemize}
\end{lem}

\begin{proof}
\emph{(a) Concave case: suppose that $c_2<0$.}
Here $\alpha^\star := -\frac{c_1}{2c_2}$ is the unique critical point of $\eff$,
and $\eff$ is strictly concave on $\mbbR$, hence it has a unique maximum on $[0,1]$. 
If $\alpha^\star \in(0,1)$, then the unique maximiser on $[0,1]$ is $\alpha^\star$; if $\alpha^\star\le 0$ 
(respectively $\alpha^\star\ge 1$) then the unique maximiser on $[0,1]$ is 0 (respectively 1).

Since $c_2<0$ we have
\begin{align*}
&\alpha^\star>0 \ \Longleftrightarrow\ \frac{c_1}{2c_2}<0\ \Longleftrightarrow \ c_1>0
\ \Longleftrightarrow \\ 
&\ln \frac{C_{56}C_{78}}{(C_{34})^2} > 0 \ \Longleftrightarrow \ 
C_{56}C_{78} > (C_{34})^2,
\end{align*}
and
\begin{align*}
\alpha^\star<1
&\Longleftrightarrow
\alpha^\star-1
=
\frac{-c_1-2c_2}{2c_2}<0 \\
&\Longleftrightarrow
-c_1-2c_2>0
\qquad\text{since } c_2<0 \\
&\Longleftrightarrow
c_1+2c_2<0 \\
&\Longleftrightarrow
\ln(C_{56}C_{78})-2\ln C_{34}
+2\bigl(\ln C_{12}+\ln C_{34}-\ln(C_{56}C_{78})\bigr)<0 \\
&\Longleftrightarrow
2\ln C_{12}-\ln(C_{56}C_{78})<0 \\
&\Longleftrightarrow
2\ln C_{12}<\ln(C_{56}C_{78}) \\
&\Longleftrightarrow
(C_{12})^2<C_{56}C_{78}.
\end{align*}
Thus
\[
0<\alpha^\star<1\quad\Longleftrightarrow\quad C_{56}C_{78} > \max\{(C_{12})^2, (C_{34})^2\}.
\]

Similarly (remembering that $c_2 < 0$),
\[
\alpha^\star\le 0\ \Longleftrightarrow\ \frac{c_1}{2c_2}\geq 0\ \Longleftrightarrow\ c_1\leq 0
\ \Longleftrightarrow\ C_{56}C_{78} \leq (C_{34})^2,
\]
and
\begin{align*}
&\alpha^\star\ge 1\ \Longleftrightarrow\ \alpha^\star-1= \frac{-c_1 - 2c_2}{2c_2} \ge 0
\ \Longleftrightarrow\ c_1 + 2c_2 \ge 0 \\
&\Longleftrightarrow\ \ln(C_{12})^2 - \ln (C_{56}C_{78}) \geq 0
\ \Longleftrightarrow\ (C_{12})^2 \geq C_{56}C_{78}.
\end{align*}

\smallskip
\noindent\emph{(b) Convex case: suppose that $c_2>0$.}
Now $\eff$ is convex on $\mbb R$, so the critical point $\alpha^\star := -\frac{c_1}{2c_2}$ 
is a global \emph{minimum}; hence on $[0,1]$ the maximum is attained at one, or both, endpoints. 
The comparison is between
\[
\eff(0)=\ln C_{34}\quad\text{and}\quad \eff(1)=\ln C_{12}.
\]
Therefore $\Eff^{\max} = \{0\}$ if $C_{34}>C_{12}$, $\Eff^{\max} = \{1\}$ if $C_{12}>C_{34}$, 
and $\Eff^{\max} = \{0,1\}$ if $C_{12}=C_{34}$.

\smallskip
\noindent\emph{(c) Linear case: suppose that $c_2=0$.}
Then $\eff(\alpha) = c_1\alpha + c_0$ where $c_1 = \ln \frac{C_{56}C_{78}}{(C_{34})^2}$.
It follows that if $C_{56}C_{78} > (C_{34})^2$ then the maximum of $\eff$ restricted to $[0, 1]$ is attained in 1,
if $C_{56}C_{78} < (C_{34})^2$ then the maximum of $\eff$ restricted to $[0, 1]$ is attained in 0,
and if $C_{56}C_{78} = (C_{34})^2$ then $\eff$ is constant and all points are maxima.
In the last case we have $C_{12}C_{34} = C_{56}C_{78} = (C_{34})^2$ so $C_{12} = C_{34}$.
\end{proof}

\noindent
Now we turn to the distribution  $\widetilde{\mbbP}_n^\mbbM$ on $\mbW_n$
from Definition~\ref{definition of scaled distribution}, where the weights are scaled by the domain size. 
As the next two results show, in this context the situation is simpler than in the unscaled case.

\begin{theor}\label{the colour proportion in scaled case}
For every choice of the weights $w_1, \ldots, w_8 \geq 0$, there is $\varepsilon > 0$ such that
\[
\lim_{n\to\infty} \widetilde{\mbbP}_n \Big(
  \big\{\mcA \in \mb W_n :
    \varepsilon n
      \le \big|P^\mcA\big|
      \le (1 - \varepsilon)n
  \big\}
\Big) = 1.
\]
\end{theor}

\begin{proof}
This is Corollary~\ref{corollary for colour concentration} proved in
Appendix~\ref{sec:one-over-n-scaling}.
\end{proof}

\begin{theor}\label{0-1 law in scaled case}
For every choice of the weights $w_1, \ldots, w_8 \geq 0$, 
$(\widetilde{\mbbP}_n : n \in \mbbN^+)$ satisfies a 0-1 law for first-order logic.
For every sentence $\varphi \in FO(\sigma)$, $\lim_{n\to\infty} \widetilde{\mbbP}_n(\varphi)$ does {\em not}
depend on the choice of weights.
\end{theor}

\begin{proof}
This is Corollary~\ref{cor:logical-0-1 law in scaled case} proved in
Appendix~\ref{sec:one-over-n-scaling}.
\end{proof}

\subsection{Outline of proofs}

In Section~\ref{edge independence} we show that the probability that $R(a, b)$ holds depends
only on whether $a$ respectively $b$ are coloured or not, and that,
conditioned on a colouring of all vertices, the event that $R(a, b)$ holds is 
independent from the event that $R(a', b')$ holds for $(a', b') \neq (a, b)$.
As one can guess from the identical cases in 
Theorem~\ref{the colour distribution, exact formulation} 
and Theorem~\ref{convergence laws},
whether $(\mbbP_n : n \in \mbbN^+)$  satisfies a 0-1 law, or convergence law but not 0-1 law, or
not even a convergence law,
is determined by almost sure asymptotic distribution of coloured vertices.
We show this in Section~\ref{sec:logical convergence laws}.
That is, provided that we know that the almost sure asymptotic distribution of coloured vertices
is as in any one of cases (1), (2), (5), (6) or (7) in Theorem~\ref{the colour distribution, exact formulation}
we can use well-known metods to derive convergence laws.
More precisely, in cases (1) and (2) we show that so-called extension axioms hold asymptotically almost surely. 
In cases (5) and (6) we show that extension axioms which imply that no, respectively all, vertices are coloured
hold asymptotically almost surely.
In case (7) we show that extension axioms which imply that no, respectively all, vertices are coloured hold
with probability 1/2 in the limit.
By well-known results in model theory (see e.g. \cite[Sections~3 and~4]{EF}),
for every first-order sentence $\psi$, either it or its negation is a consequence of some finite set of extension axioms and therefore
the probability that $\psi$ holds must tend to 1 or 0 as $n \to \infty$ in cases (1), (2), (5), and (6).
In case (7) the limit probability may also be 1/2.

The proof that a convergence law does not hold in  cases (3) and (4) of the mentioned theorems is
given in Appendix~\ref{proof of nonconvergence}
and relies on a result about almost sure interpretability of a fragment of arithmetic by
monadic second-order formulas in undirected graphs with fixed and independent edge probability \cite{KS}.
The idea is to first find a sentence $\psi$ in monadic second-order logic such that $\mbbP_n(\psi)$ does not converge as $n\to\infty$.
Then we use the assumption that we are in case (3) or (4) to ``translate'' $\psi$ into a sentence $\chi$
of first-order logic which is, with probability tending to 1 as $n\to\infty$, equivalent to $\psi$.
It follows that $\mbbP_n(\chi)$ does not converge.

Given that the above has been carried out it remains (in the case of unscaled weights) prove
Theorem~\ref{the colour distribution, exact formulation}.
As these proofs are rather long and technical (but use only basic calculus and probability theory),
and independent from the other results  presented, they
are postponed to appendices~\ref{sec:asymptotics} and~\ref{sec:limit}

Now we turn to the case when the weights are scaled by $1/n$.
Theorems~\ref{the colour proportion in scaled case} and~\ref{0-1 law in scaled case}
are proved in Appendix~\ref{sec:one-over-n-scaling}.
The proof of Theorem~\ref{0-1 law in scaled case} uses 
Theorem~\ref{the colour proportion in scaled case}
and the method of extension axioms already discussed.
In order to do this we first need to derive estimates of edge probabilities conditioned on colours,
as well as proving that edges are independent conditioned on any colouring.
We get similar results as in the unscaled case in Section~\ref{edge independence} but more work is needed.
The proof of Theorem~\ref{the colour proportion in scaled case} begins by following a similar pattern
as the proofs in appendices~\ref{sec:asymptotics} and~\ref{sec:limit} for unscaled weights, 
but with scaled weights the extreme cases present with unscaled weights never appear.

\subsection{Some more notation}\label{notation for Z etc}

Recall that we have fixed an arbitrary canonical MLN $\mbbM$.
It will be convenient to use the following notation in the proofs that follow later.

\begin{defin}\label{definition of mu}{\rm
For every $\mcA \in \mbW_n$, define 
\[
\mu_n(\mcA) := 
\exp\Big(\sum_{i=1}^8 w_i |\varphi_i(\mcA)| \Big)
\]
and for $\mbX \subseteq \mbW_n$ we let
\[
\mu_n(\mbX) := \sum_{\mcA \in \mbX} \mu_n(\mcA).
\]
We call $\mu_n(\mcA)$ (respectively $\mu_n(\mbX)$) the {\em (unnormalized) weight} of $\mcA$ (respectively $\mbX_n$).
}\end{defin}
It follows that for all $\mcA \in \mbW_n$
\[
\mbbP_n(\mcA) = \frac{\mu_n(\mcA)}{\mu_n(\mbW_n)}.
\]

\begin{notation}\label{notation for Z-n}{\rm
For $n \in \mbbN^+$ and $0 \leq m \leq n$, let
\begin{align*}
&Z_n := \mu_n\big(\mbW_n\big), \\
&\mbW_n(m) := \{ \mcA \in \mbW_n : |P^\mcA| = m\} \ \ \text{and} \\
&Z_n(m) := \mu_n\big(\mbW_n(m)\big).
\end{align*}
}\end{notation}

\noindent
Observe that $Z_n = \sum_{m=0}^n Z_n(m)$ and, if $\mbX, \mbY \subseteq \mbW_n$ and $\mbY \neq \es$ then
\[
\mbbP_n(\mbX) = \frac{\mu_n(\mbX)}{Z_n} \qquad \text{and} \qquad 
\mbbP_n(\mbX \ | \ \mbY) = \frac{\mu_n(\mbX \cap \mbY)}{\mu_n(\mbY)}.
\]

With the introduced notation we can also make an analogy to statistical mechanics.

\begin{rem}\label{remark about statistical mechanics}
{\bf (Analogy to statistical mechanics)} 
{\rm 
From the point of view of statistical mechanics we can view each $\mcA \in \mbW_n$ as a state of a system
with energy $E_\mcA := \sum_{i=1}^8 w_i \big|\varphi_i(\mcA)\big|$
(where a constant related to the temperature is omitted).
Then $\mu_n(\mcA) = \exp\big(E_\mcA\big)$ is the (simplified)
Boltzmann weight of $\mcA$, and $Z_n  = \sum_{\mcA \in \mbW_n} \mu_n(\mcA)$ is a normalizing constant,
also called ``partition function'', where $\ln Z_n$ and $\ln Z_n(m)$, 
for $0 \leq m \leq n$, strongly influence the behaviour of the whole system.
We will see, in Appendix~\ref{sec:asymptotics}, that $\ln Z_n(m) = n^2 \eff\big(\frac{m}{n}\big) + O(n)$ 
where $\eff$ can be viewed as the free energy density function, or phase function, of the system.
Since $n^2 \eff\big( \frac{m}{n} \big)$
is the leading term of $\ln Z_n(m)$ 
the shape of the free energy density function $\eff$ is crucial when we compare $\ln Z_n(m)$ and $\ln Z_n$
in order to estimate $\mbbP_n(\mbW_n(m))$.
See e.g. \cite[Chapters 5.3.2, 6.1, 6.2, 9.2]{Set}.
}\end{rem}

\section{Edge independence, conditioned on a colouring}\label{edge independence}

\noindent
The main result of this section, Proposition~\ref{lem:edge-factor}, 
is that, conditioned on a fixed (but arbitrary) colouring, the edges are independent of each other.
We also determine the probability of an edge $R(a, b)$ conditioned on the colours of $a$ and $b$.
This will be used in the proofs of the convergence, and nonconvergence, laws.

\begin{defin}\label{viewing the relation symbols as random variables}{\rm
(a) For any fixed $n$ and $a \in [n]$ we define the random variable
$P_a : \mbW_n \to \{0, 1\}$  by $P_a(\mcA) = P^\mcA(a)$.\\
(b) For any fixed $n$ and $(a, b) \in [n]^2$ we define the random variable
$R_{(a, b)} : \mbW_n \to \{0, 1\}$ by $R_{(a, b)}(\mcA) = R^\mcA(a, b)$.
}\end{defin}

\begin{rem}\label{remark on different notations for P-A}{\rm
Observe that with the explained notational conventions
(see Section~\ref{Notation and terminology}), for all $n$, $a, b \in [n]$, and $\mcA \in \mbW_n$,
\begin{align*}
&\mcA \models P(a) \ \Longleftrightarrow \ P^\mcA(a) = 1 \ \Longleftrightarrow \ P_a(\mcA) = 1 \text{ and } \\
&\mcA \models R(a, b)  \ \Longleftrightarrow \ R^\mcA(a, b) = 1 \ \Longleftrightarrow \ R_{(a, b)}(\mcA) = 1.
\end{align*}
}\end{rem}

\begin{notation}\label{not:typed}
Let
\begin{align*}
&p_{11} :=\dfrac{e^{w_1}}{e^{w_1}+e^{w_2}},\qquad
p_{00} :=\dfrac{e^{w_3}}{e^{w_3}+e^{w_4}}, \\
&p_{10} :=\dfrac{e^{w_5}}{e^{w_5}+e^{w_6}}, \ \text{and } \
p_{01} := \dfrac{e^{w_7}}{e^{w_7}+e^{w_8}}.
\end{align*}
Figure~\ref{fig:edge-probabilities-by-colour} gives a schematic picture of the four
edge probabilities.
\end{notation}

\begin{figure}[t]
\centering
\includegraphics[width=0.6\textwidth]{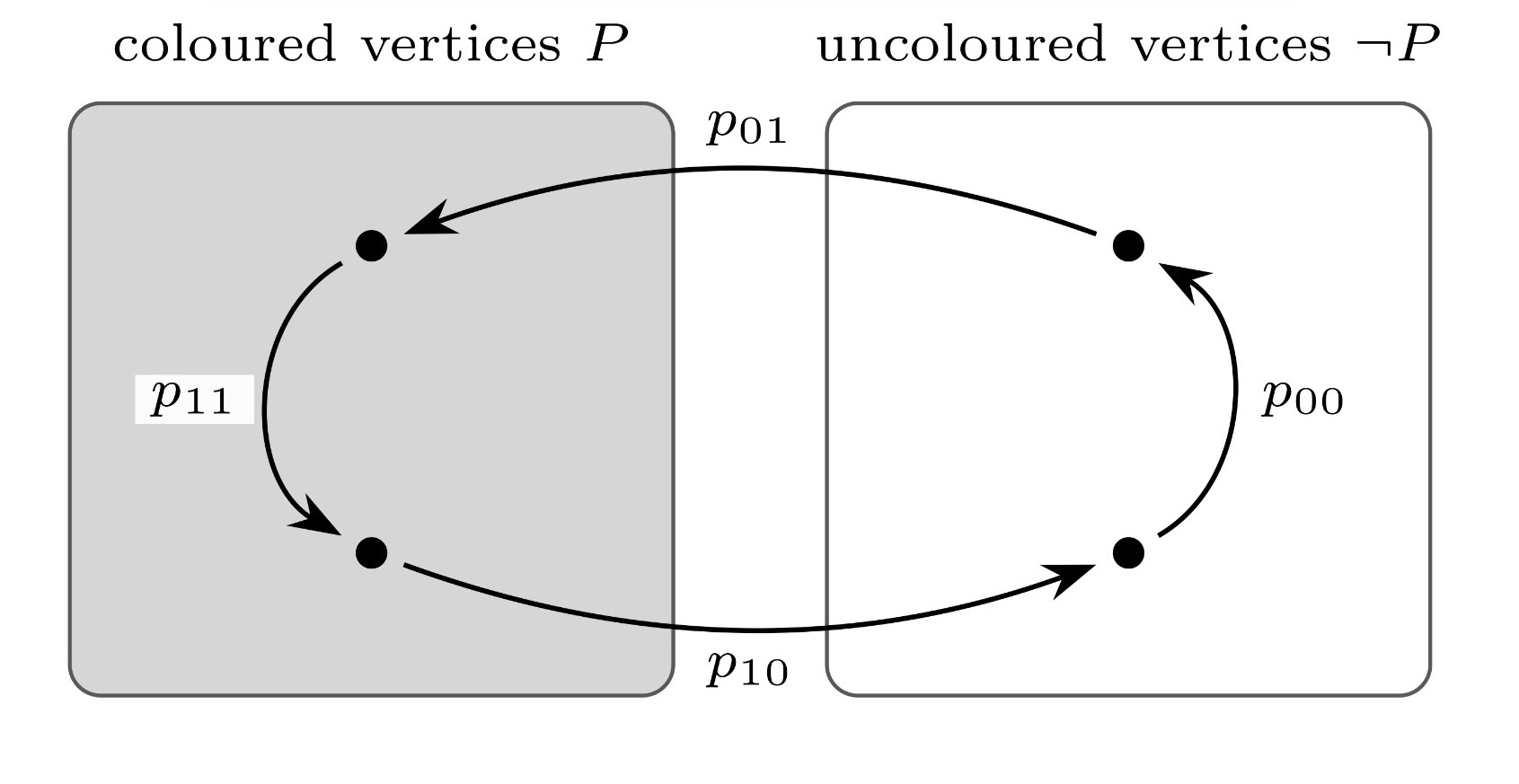}
\caption{Edge probabilities conditioned on a fixed colouring of the vertices.}
\label{fig:edge-probabilities-by-colour}
\end{figure}

\begin{prop}\label{lem:edge-factor}
Fix any $\bar{\alpha} = (\alpha_1, \ldots, \alpha_n) \in \{0, 1\}^n$ and let 
\[
\mbX_{\bar{\alpha}} = \big\{\mcA \in \mbW_n : P^\mcA(k) = \alpha_k \ \text{ for all } k = 1, \ldots, n \big\}.
\]
For every $(a, b) \in [n]^2$,
\[
\mbbP_n\big(R_{(a, b)} = 1  \ \big| \ \mbX_{\bar{\alpha}} \big) = 
\begin{cases}
p_{11} &\text{ if } \alpha_a = \alpha_b =1,\\
p_{00} &\text{ if } \alpha_a = \alpha_b =0,\\
p_{01} & \text{ if } \alpha_a = 0 \text{ and } \alpha_b = 1, \\
p_{10} & \text{ if } \alpha_a = 1 \text{ and } \alpha_b = 0.
\end{cases}
\]
Also, {\rm conditioned on  $\mbX_{\bar{\alpha}}$},
the random variable $R_{(a, b)}$  is independent from all random variables $R_{(a', b')}$,
where $(a', b') \neq (a, b)$.
\end{prop}

\begin{proof}
Fix $\bar{\alpha} \in\{0,1\}^n$. 
It follows from the definition of the formulas 
$\varphi_1(x, y),$ $\ldots$, $\varphi_8(x, y)$ that
for each ordered pair $(a,b)\in[n]^2$ and each $r\in\{0,1\}$ there is a unique index
$\iota(a,b,r)\in\{1,\ldots,8\}$ such that, 
for each $\mcA \in \mbX_{\bar{\alpha}}$,
\[
\text{$R^\mcA(a, b) = r$ if and only if $\mcA \models \varphi_{\iota(a,b,r)}(a, b)$.}
\]
For each $(a,b) \in [n]^2$ we have 
\[
\sum_{\ell=1}^8 w_\ell |\varphi_\ell(\mc A)|
=\sum_{(a,b) \in[n]^2} w_{\iota(a,b,R(a,b))}.
\]
It follows that for all $\mcA \in \mbX_{\bar{\alpha}}$,
\begin{align*}
\mu_n(\mc A)&=\exp \bigg(\sum_{\ell=1}^8 w_\ell |\varphi_\ell(\mc A)|\bigg)
= \exp\bigg(\sum_{(a,b) \in[n]^2} w_{\iota(a,b,R^\mcA(a,b))}\bigg) \\
&=\prod_{(a,b)\in[n]^2} e^{\,w_{\iota(a,b,R^\mcA(a,b))}}.
\end{align*}
Therefore
\[
\mu_n\big(\mbX_{\bar{\alpha}}\big) 
=\prod_{(a,b)\in[n]^2}\Big(e^{w_{\iota(a,b,0)}}+e^{w_{\iota(a,b,1)}}\Big).
\]
We also have
\begin{align*}
&\mu_n\big(\big\{\mcA \in \mbX_{\bar{\alpha}} : \mcA \models R(a,b)\big\}\big)
= \\
&e^{w_{\iota(a,b,1)}} \cdot
\prod_{\substack{(a',b')\in[n]^2 \\ (a',b') \neq (a,b)}}\Big(e^{w_{\iota(a',b',0)}}+e^{w_{\iota(a',b',1)}}\Big).
\end{align*}
 It follows that for each ordered pair $(a,b) \in [n]^2$,
\begin{align*}
&\mbbP_n\big(\big\{\mcA \in \mbX_{\bar{\alpha}} : \mcA \models R(a,b) \big\} \ \big| \  \mbX_{\bar{\alpha}}\big)
=  \dfrac{\mu_n\big( \big\{\mcA \in \mbX_{\bar{\alpha}} : \mcA \models R(a,b) \big\} \big)} 
{\mu_n\big(\mbX_{\bar{\alpha}}\big)} \\
= &\frac{e^{w_{\iota(a,b,1)}}}{e^{w_{\iota(a,b,0)}}+e^{w_{\iota(a,b,1)}}}
\prod_{\substack{(a',b')\in[n]^2 \\ (a',b') \neq (a,b)}}
\dfrac{e^{w_{\iota(a',b',0)}}+e^{w_{\iota(a',b',1)}}}{e^{w_{\iota(a',b',0)}}+e^{w_{\iota(a',b',1)}}} \\
= &\dfrac{e^{w_{\iota(a,b,1)}}}{e^{w_{\iota(a,b,0)}}+e^{w_{\iota(a,b,1)}}}.
\end{align*}
Therefore, recalling Notation~\ref{not:typed},
\begin{align*}
\mbbP_n\big(R_{(a,b)}=1 \ | \ \mbX_{\bar{\alpha}}\big) &=
\mbbP_n\big(\big\{\mcA \in \mbX_{\bar{\alpha}} : \mcA \models R(a,b) \big\} \ \big| \  \mbX_{\bar{\alpha}}\big) \\
&= 
\begin{cases}
p_{11}:=\frac{e^{w_1}}{e^{w_1}+e^{w_2}} & \text{ if } \alpha_a = \alpha_b =1,\\
p_{00}:=\frac{e^{w_3}}{e^{w_3}+e^{w_4}} & \text{ if } \alpha_a = \alpha_b =0,\\
p_{10}:=\frac{e^{w_5}}{e^{w_5}+e^{w_6}} & \text{ if } \alpha_a = 1 \text{ and } \alpha_b = 0, \\
p_{01}:=\frac{e^{w_7}}{e^{w_7}+e^{w_8}} & \text{ if } \alpha_a = 0 \text{ and } \alpha_b = 1.
\end{cases}
\end{align*}
By similar reasoning as above, if $(a_1, b_1), \ldots, (a_k, b_k) \in [n]^2$ are distinct pairs
and $\beta_1, \ldots, \beta_k \in \{0, 1\}$, then
\begin{align*}
&\mbbP_n\big(R_{(a_1, b_1)}=\beta_1, \ldots, R_{(a_k, b_k)} = \beta_k \ | \ \mbX_{\bar{\alpha}}\big) \\
= 
&\prod_{i=1}^k \dfrac{e^{w_{\iota(a_i, b_i, \beta_i)}}}{e^{w_{\iota(a_i, b_i, 0)}} + e^{w_{\iota(a_i, b_i, 1)}}}
\prod_{\substack{(a',b')\in[n]^2 \\ (a',b') \neq (a_i,b_i) \\ \text{for all } i = 1, \ldots, k}}
\dfrac{e^{w_{\iota(a',b',0)}}+e^{w_{\iota(a',b',1)}}}{e^{w_{\iota(a',b',0)}}+e^{w_{\iota(a',b',1)}}} \\
=
&\prod_{i=1}^k \mbbP_n\big(R_{(a_i, b_i)} = \beta_i \ \big| \ \mbX_{\bar{\alpha}}\big).
\end{align*}
It follows that the random variable $R_{(a, b)}$ is independent from all $R_{(a', b')}$ where $(a', b') \neq (a, b)$.
\end{proof}

\section{Proofs of convergence laws}\label{sec:logical convergence laws}

\noindent
In this section we prove Theorem~\ref{convergence laws} except cases (3) and (4) the (long and technical) proofs of which are postponed
to Appendix~\ref{proof of nonconvergence}.
{\em We assume that Theorem~\ref{the colour distribution, exact formulation}  holds}.
Theorem~\ref{the colour distribution, exact formulation} 
is proved in Appendices~\ref{sec:asymptotics} and \ref{sec:limit} and does not use any results from this section.

We will use the following basic probability theoretic fact, which follows easily from the definition 
of conditional probability:

\begin{fact}\label{dividing a conditional probability}
Suppose that $\mbbP$ is a probability measure on a set $\Omega$.
Let $X, Y \subseteq \Omega$ be measurable and suppose that 
$Y = Y_1 \cup \ldots \cup Y_m$ where $Y_i \cap Y_j = \es$ if $i \neq j$ and each $Y_i$ is measurable
and $\mbbP(Y_i) > 0$.
Let $I \subseteq [0, 1]$ be an interval.
If $\mbbP(X \ | \  Y_i) \in I$ for all $i = 1, \ldots, m$, then $\mbbP(X \ | \ Y) \in I$.
\end{fact}

\begin{defin}\label{definition of extension axiom}{\rm 
(a) By a {\em literal} we mean an atomic formula or a negation of an atomic formula.\\
(b) By a {\em maximal consistent conjunction of literals} we mean a  conjunction $\theta(x_1, \ldots, x_k)$
of literals which is satisfiable and such that, for all $i, j \in \{1, \ldots, k\}$,
\begin{itemize}
\item either $x_i = x_j$ or $x_i \neq x_j$ is a conjunct of $\theta(x_1, \ldots, x_k)$,
\item either $P(x_i)$ or $\neg P(x_i)$ is a conjunct of $\theta(x_1, \ldots, x_k)$, and
\item either $R(x_i, x_j)$ or $\neg R(x_i, x_j)$ is a conjunct of $\theta(x_1, \ldots, x_k)$.
\end{itemize}
(c) By an {\em extension axiom} we mean a sentence of the form 
(for any $k \in \mbbN$)
\[
 \forall x_1, \ldots, x_k \exists y \big(\theta_1(x_1, \ldots, x_k) \rightarrow \theta_2(x_1, \ldots, x_k, y)\big)
\]
where $\theta_1(x_1, \ldots, x_k)$ and $\theta_2(x_1, \ldots, x_k, y)$ are maximal consistent conjunctions of literals
such that $\theta_1(x_1, \ldots, x_k) \wedge \theta_2(x_1, \ldots, x_k, y)$ is satisfiable,
and, for all distinct $i, j \in \{1, \ldots, k\}$, $x_i \neq x_j$ and $x_i \neq y$ are conjuncts
of $\theta_2$.
We allow $k = 0$ so $\exists y \theta(y)$ is an extension axiom 
if $\theta(y)$ is a maximal consistent conjunction of literals.\\
(d) Let $T_{ext}$ be the set of all extension axioms.
}\end{defin}

\noindent
The following fact follows from the 0-1 law for first-order logic and the uniform probability distribution,
via Fagin's proof \cite{Fag},
and can also be found in \cite[Example~3.2.11, Corollary~4.1.3]{EF}.

\begin{fact}\label{the theory of extension axioms is complete}
(a) $T_{ext}$ has only infinite models, is $\omega$-categorical, and hence complete
(i.e. for every sentence $\psi$, either $T_{ext} \models \psi$ or $T_{ext} \models \neg \psi$).\\
(b) If $\lim_{n\to\infty}\mbbP_n(\psi) = 1$ for every $\psi \in T_{ext}$ then 
$(\mbbP_n : n \in \mbbN^+)$ satisfies a 0-1 law for first-order logic, and
for every sentence $\varphi \in FO(\sigma)$, 
$\lim_{n\to\infty}\mbbP_n(\varphi) = 1$ if and only if $T_{ext} \models \varphi$.
\end{fact}

\begin{lem}\label{0-1 laws in cases 1 and 2}
Suppose that either $\Eff^{\max} = [0, 1]$, or that
$\Eff^{\max} = \{\alpha^\star\}$ where $\alpha^\star \in (0, 1)$.
Then $(\mbbP_n : n \in \mbbN^+)$ satisfies a 0-1 law for first-order logic.
\end{lem}

\begin{proof}
Adopt the assumptions of the lemma.
It follows from parts~(1) and~(2) of
Theorem~\ref{the colour distribution, exact formulation} 
that there is $\varepsilon \in (0, 1/2)$ such that if
\[
\mbX_n^\varepsilon = \big\{\mcA \in \mbW_n : \varepsilon n \leq |P^\mcA| \leq (1 - \varepsilon)n \big\} \ 
\text{ then } \ \lim_{n\to\infty}\mbbP_n\big(\mbX_n^\varepsilon\big) = 1.
\]
This together with Fact~\ref{the theory of extension axioms is complete}~(b)
implies that it is sufficient to show that for every $\psi \in T_{ext}$, 
$\lim_{n\to\infty}\mbbP_n\big(\psi \ \big| \ \mbX_n^\varepsilon\big) = 1$.

So let 
$\psi := \forall x_1, \ldots, x_k \exists y \big( \theta_1(x_1, \ldots, x_k) \rightarrow \theta_2(x_1, \ldots, x_k, y)\big)$
be an extension axiom.
Suppose that $P(y)$ is a conjunct of $\theta_2(x_1, \ldots, x_k, y)$. The case when $\neg P(y)$ is a
conjunct of the same formula is treated in a similar way so we leave it to the reader.

Fix any $B \subseteq [n]$ such that $\varepsilon n \leq |B| \leq (1 - \varepsilon) n$ and let
\[
\mbY_n^B = \big\{\mcA \in \mbX_n^\varepsilon : P^\mcA = B \big\}.
\]
Suppose that $a_1, \ldots, a_k \in [n]$, $\mcA \in \mbY_n^B$, and 
$\mcA \models \theta_1(a_1, \ldots, a_k)$.
Take any $b \in B \setminus \{a_1, \ldots, a_k\}$.
Due to Proposition~\ref{lem:edge-factor},
the probability
that $\mcA \models \theta_2(a_1, \ldots, a_k, b)$
is a constant $p \in (0, 1)$ which is a product of numbers of the form $p_{01}, p_{10}$, $p_{11}$,
$1 - p_{01}$,  $1 - p_{10}$, and $1 - p_{11}$
from Notation~\ref{not:typed}, and $p$ depends only on $\theta_2$.
Using the independence result
in Proposition~\ref{lem:edge-factor}, the probability (conditioned on $\mcA \in \mbY_n^B$) that 
$\mcA \not\models \theta_2(a_1, \ldots, a_k, b)$ for all 
$b \in B \setminus \{a_1, \ldots, a_k\}$ is at most $(1 - p)^{\varepsilon n - k}$.
Since the number of choices of $a_1, \ldots, a_k$ is bounded by $n^k$ it follows that
$\mbbP_n\big(\neg\psi \ \big| \ \mbY_n^B \big) \leq n^k (1 - p)^{\varepsilon n - k}$.
Since we have the same bound for all $B \subseteq [n]$ such that 
$\varepsilon n \leq |B| \leq (1 - \varepsilon) n$
it follows from
Fact~\ref{dividing a conditional probability}
that $\mbbP_n\big(\neg\psi \ \big| \ \mbX_n^\varepsilon\big) \leq n^k (1 - p)^{\varepsilon n - k}$,
so this (conditional) probability tends to 0 as $n\to\infty$.
Hence $\lim_{n\to\infty}\mbbP_n\big(\psi \ \big| \ \mbX_n^\varepsilon\big) = 1$.
\end{proof}

\begin{defin}\label{definition of variants of extension theories}{\rm
Let $T_{ext}^-$ (respectively $T_{ext}^+$) contain $\neg \exists x P(x)$ 
(respectively $\forall x P(x)$) and all extension axioms 
$\forall \bar{x} \exists y \big(\theta_1(\bar{x}) \rightarrow \theta_2(\bar{x}, y)\big)$
such that $\neg P(y)$ (respectively $P(y)$) 
is a conjunct of $\theta_2(\bar{x}, y)$ and $\bar{x}$ denotes a finite sequence
of distinct variables.
}\end{defin}

\begin{lem}\label{the theory of modified extension axioms is complete}
(a) $T_{ext}^-$ and $T_{ext}^+$ have only infinite models, are $\omega$-categorical, and 
are therefore complete.\\
(b) If $\lim_{n\to\infty}\mbbP_n(\psi) = 1$ for every $\psi \in T_{ext}^-$ (respectively $\psi \in T_{ext}^+)$
then $(\mbbP_n : n \in \mbbN^+)$ satisfies a 0-1 law for first-order logic.
\end{lem}

\begin{proof}
The proof of (a) is a straightforward modification of the (back-and-forth) argument used to prove
the classical statement of Fact~\ref{the theory of extension axioms is complete}.
This is explained to some extent in \cite[Section~4.2]{EF} since the class
of finite $\sigma$-structures $\mcA$ such that $\mcA \models \neg \exists x P(x)$
(respectively such that $\mcA \models \forall x P(x)$) is a {\em parametric class} in the sense of \cite{EF}.
Part~(b) follows from part~(a) in exactly the same way as (b) follows from (a) in 
Fact~\ref{the theory of extension axioms is complete}.
\end{proof}

\begin{lem}\label{0-1 laws in cases 5 and 6}
Suppose that either $\Eff^{\max} = \{0\}$ and $\eff'(0) < 0$, 
or that $\Eff^{\max} = \{1\}$ and $\eff'(1) > 0$.
Then $(\mbbP_n : n \in \mbbN^+)$ satisfies a 0-1 law for first-order logic.
\end{lem}

\begin{proof}
First suppose that $\Eff^{\max} = \{0\}$ and $\eff'(0) < 0$.
By part~(5) of Theorem~\ref{the colour distribution, exact formulation},
$\lim_{n\to\infty}\mbbP_n\big(\mbW_n(0)\big) = 1$.
This together with 
Lemma~\ref{the theory of modified extension axioms is complete}
implies that it suffices to show that, for every $\psi \in T_{ext}^-$,
$\lim_{n\to\infty}$ $\mbbP_n\big(\psi \ \big| \ \mbW_n(0)\big) = 1$.
By the definition of $\mbW_n(0)$ it follows that $\mcA \models \neg\exists x P(x)$ for every $\mcA \in \mbW_n(0)$.
It remains to show that, for every extension axiom $\psi \in T_{ext}^-$, 
$\lim_{n\to\infty} \mbbP_n\big(\psi \ \big| \ \mbW_n(0)\big) = 1$.
This can be done by minor (and simplifying) modifications of the argument in the
proof of Lemma~\ref{0-1 laws in cases 1 and 2}, so we leave the details to the reader.

Suppose that $\Eff^{\max} = \{1\}$ and $\eff'(1) > 0$. 
By part~(6) of Theorem~\ref{the colour distribution, exact formulation},
$\lim_{n\to\infty}\mbbP_n\big(\mbW_n(n)\big) = 1$.
Hence it suffices to prove that for every $\psi \in T_{ext}^+$, 
$\lim_{n\to\infty} \mbbP_n\big(\psi \ \big| \ \mbW_n(n)\big) = 1$.
Since $\mbW_n(n)$ contains all $\mcA \in \mbW_n$ such that $\mcA \models \forall x P(x)$
we can argue similarly (and in a simpler way) as in the proof of
Lemma~\ref{0-1 laws in cases 1 and 2} and we leave the details to the reader.
\end{proof}

\begin{lem}\label{convergence law in case 7}
If $\Eff^{\max} = \{0, 1\}$ then 
$(\mbbP_n : n \in \mbbN^+)$ satisfies a convergence law, but not 0-1 law, for first-order logic, and for 
every first-order sentence $\varphi$, $\lim_{n\to\infty}\mbbP_n(\varphi) \in \{0, \frac{1}{2}, 1\}$.
\end{lem}

\begin{proof}
Suppose that $\Eff^{\max} = \{0, 1\}$.
By Theorem~\ref{the colour distribution, exact formulation},
\[
\lim_{n\to\infty} \mbbP_n\big(\mbW_n(0)\big) = \lim_{n\to\infty} \mbbP_n\big(\mbW_n(n)\big) = 1/2.
\]
It follows that $\lim_{n\to\infty}\mbbP_n\big(\neg\exists x P(x)\big) = 1/2$, so we do not have a 0-1 law.
In the proof of Lemma~\ref{0-1 laws in cases 5 and 6}
we explained that for every $\psi \in T_{ext}^-$ and every $\psi' \in T_{ext}^+$,
\[
\lim_{n\to\infty} \mbbP_n\big(\psi \ \big| \ \mbW_n(0)\big) = 
\lim_{n\to\infty} \mbbP_n\big(\psi' \ \big| \ \mbW_n(n)\big) = 1.
\]
Since both $T_{ext}^-$ and $T_{ext}^+$ are complete it follows that, for every sentence $\psi$,
\[
\lim_{n\to\infty} \mbbP_n\big(\psi \ \big| \ \mbW_n(0)\big) \in \{0, 1\} \ \text{ and } \
\lim_{n\to\infty} \mbbP_n\big(\psi \ \big| \ \mbW_n(n)\big) \in \{0, 1\}.
\]
From the above conclusions it follows that, for every sentence $\psi$,
\begin{align*}
\mbbP_n\big(\psi\big) 
&= \mbbP_n\big(\psi \ \big| \ \mbW_n(0)\big) \cdot \mbbP_n\big(\mbW_n(0)\big) \\
&+
\mbbP_n\Bigg( \psi \ \bigg| \ \bigcup_{m=1}^{n-1} \mbW_n(m)\Bigg) \cdot 
\mbbP_n\Bigg(\bigcup_{m=1}^{n-1} \mbW_n(m)\Bigg) \\
&+ 
 \mbbP_n\big(\psi \ \big| \ \mbW_n(n)\big) \cdot \mbbP_n\big(\mbW_n(n)\big) \\
&= \mbbP_n\big(\psi \ \big| \ \mbW_n(0)\big) \cdot \Big( \tfrac{1}{2} + o(1)\Big) +
\mbbP_n\Bigg( \psi \ \bigg| \ \bigcup_{m=1}^{n-1} \mbW_n(m)\Bigg) \cdot o(1) \\
&+ \mbbP_n\big(\psi \ \big| \ \mbW_n(n)\big) \cdot \Big( \tfrac{1}{2} + o(1)\Big).
\end{align*}
Hence $\lim_{n\to\infty} \mbbP_n(\psi) \in \{0, 1/2, 1\}$.
\end{proof}

\section{Conclusions}

\noindent
We have studied sequences of probabilities $(\mbbP_n^\mbbM(\psi) : n \in \mbbN^+)$
where $\psi$ is a sentence (i.e. formula without free variables) 
of first-order logic and $\mbbP_n^\mbbM(\psi)$ is the probability that $\psi$
is true in a random structure (possible world) with domain $[n] := \{1, \ldots, n\}$.
Here $\mbbM$ denotes a Markov logic network (MLN) and $\mbbP_n^\mbbM$ the probability distribution
that $\mbbM$ determines on the set $\mbW_n$ of structures with domain $[n]$.
An MLN is a finite set of pairs of the form $(\varphi(x_1, \ldots, x_k), w)$ where 
$\varphi(x_1, \ldots, x_k)$ is a first-order formula called a {\em soft constraint} and $w \geq 0$ a real number called its {\em weight}.
The intuition is that the weight of a soft constraint gives a measure of how likely it is that the soft constraint is satisfied
by a random tuple from a structure.
We have proved (Theorem~\ref{all three cases can be realized})
that even for a simple signature $\sigma := \{P, R\}$ where $P$ and $R$ are relation symbols of arities 1 and 2,
respectively, there are, for $j = 1, 2, 3$, MLNs $\mbbM_j = \{(\varphi_i(x, y) : w_{j, i}) : i = 1, \ldots, 8 \}$ 
(all three of which use the same 8 soft constraints) such that
\begin{enumerate}
\item for every first-order sentence $\psi$, $\mbbP_n^{\mbbM_1}(\psi)$ converges to 0 or 1 as $n\to\infty$,
\item  for every first-order sentence $\psi$, $\mbbP_n^{\mbbM_2}(\psi)$ converges as $n\to\infty$ but not necessarily to
0 or 1, and 
\item there is a first-order sentence $\psi$ such that $\mbbP_n^{\mbbM_3}(\psi)$ does not converge as $n\to\infty$.
\end{enumerate}
Hence, all three phenomena can arise in a fairly simple context.
Moreover, for every $\varepsilon > 0$, one can choose the weights $w_{j, i}$ so that the above holds and, in addition,
$|w_{j, i} - w_{k, i}| < \varepsilon$ for all $j = 1, 2, 3$, $k = 1, 2, 3$, and $i = 1, \ldots, 8$.
Thus, an arbitrarily small change of the weights of an MLN $\mbbM$ can change the behaviour 
of the sequence $(\mbbP_n^\mbbM(\psi) : n \in \mbbN^+)$.
Our results also characterize the asymptotic  behaviour of $\mbbP_n^\mbbM(\psi)$ for all MLNs $\mbbM$
with soft contraints being boolean combinations of $P(x)$, $P(y)$, and $R(x, y)$ in which both $x$ and $y$ occur, 
and for all first-order sentences $\psi$
(Theorem~\ref{limit of sentence depends only on the case}, 
Lemma~\ref{getting a canonical MLN}, and Theorem~\ref{convergence laws}).
Since the situation will hardly become simpler for more complex signatures or more complex soft constraints than considered
here this suggests that the asymptotic behaviour of MLNs is in general complex and not captured by a simple general theory.

Besides determining in which case of (1) -- (3) we end up,
the soft constraints and weights of the MLNs considered here determine the asymptotically almost sure
(as $n \to \infty$) distribution of coloured elements (those that satisfy $P(x)$), and they
determine (also in the limit) the probability that $R(x, y)$ holds conditioned on knowing whether $x$ and $y$ are coloured or not
(Theorem~\ref{the colour distribution, exact formulation}  and Proposition~\ref{lem:edge-factor}).

We have also seen that if we scale the weights of $\mbbM$
by $1/n$ where $n$ is the domain size, getting a distribution that
we call $\widetilde{\mbbP}_n^\mbbM$, then the weights loose their
influence on the limit of some natural sequences of probabilities. 
More precisely, no matter how the weights of $\mbbM$ are chosen,
$\widetilde{\mbbP}_n^{\mbbM}(\psi)$ converges to 0 or 1 as $n\to\infty$ for every first-order sentence $\psi$
(Theorem~\ref{0-1 law in scaled case}).
Moreover, the probability that $R(x, y)$ holds tends to 1/2 as the domain size grows,
no matter if $x$, respectively $y$, is coloured or not
(Lemma~\ref{lem:scaled-edge-factor}).
This contrasts the findings of 
Poole et al \cite{Poole} and Mittal et al \cite{Mittal}
that motivated the introduction of scaled MLNs \cite{JBB, Mittal, Wei25}, and it shows,
as previously found by Weitkämper \cite{Wei25},
that scaling the weights need not give the weights more influence in the limit. 
In fact, as we have shown, the opposite can happen.

Although we do not anticipate a simple general theory of asymptotic properties of sequences of distributions
determined by MLNs we are not entirely pessimistic.
We believe that there may be ``dividing lines'' between different types of MLNs (where ``type'' refers to what kind
of soft constraints the MLN uses). There may be a theory for understanding MLNs in one region determined by the dividing lines
and another theory for understanding MLNs in another region. There may also be regions such that the all MLNs in them
are very hard to analyse and understand. 
The task is to find these dividing lines, but in order to find them we may need information about more specific cases
of MLNs. The investigation \cite{TK2} is a step in this direction.

\subsection*{Acknowledgements}

Parts (3) and (4) of Theorem~\ref{convergence laws}
where found by generative AI but the proof has
been thoroughly reformulated by the authors to be easier to understand. 

Yasmin Tousinejad and Vera Koponen were partially supported by the Wallenberg AI, Autonomous Systems and Software
Program (WASP) funded by the Knut and Alice Wallenberg Foundation.
Vera Koponen was partially supported by the Swedish Research Council, grant 2023-05238.

\appendix

\section{Exact and approximate expressions for $Z_n(m)$}\label{sec:asymptotics}

\noindent
Recall that $Z_n(m) = \mu_n\big(\mbW_n(m)\big)$ where $\mbW_n(m) = \{\mcA \in \mbW_n : |P^\mcA| = m\}$.
Suppose that $0 \leq m \leq n$ and $\mcA \in \mbW_n(m)$.
Then $|P^\mcA| = m$, so intuitively speaking, $\mcA$ has exactly $m$ coloured vertices.
Then there are $m^2$ ordered pairs of two coloured
vertices, $(n-m)^2$ ordered pairs of two uncoloured vertices (these counts include the loops), and 
$2m(n-m)$ ordered pairs where exactly one of the vertices is coloured (no loops occur in the mixed coloured
pairs). 
Let $M \subseteq [n]$ and suppose that $|M| = m$.
Due to the specification of the formulas $\varphi_i(x, y)$ and weights $w_i$ for $i = 1, \ldots, 8$,
it follows that 
\[
\mu_n\big(\big\{ \mcA \in \mbW_n(m) : P^\mcA = M \big\}\big) = 
C_{12}^{m^2}C_{34}^{(n-m)^2}C_{56}^{m(n-m)}C_{78}^{m(n-m)}.
\]
As there are $\binom{n}{m}$ ways to choosing exactly $m$
vertices from $[n]$, it follows that 
\begin{equation}\label{eq:Znm-exact}
Z_n(m) = \mu_n(\mbW_n(m)) 
=\binom{n}{m}\,C_{12}^{\,m^2}C_{34}^{\,(n-m)^2}C_{56}^{m(n-m)}C_{78}^{m(n-m)}.
\end{equation}
It follows from the definition of $Z_n$ that 
\[
Z_n=\sum_{m=0}^n Z_n(m).
\]

\noindent
From~(\ref{eq:Znm-exact}) and~(\ref{eq:def-psi}) we get 

\begin{align*}
\ln Z_n(m) & = \ln\binom{n}{m} + m^2\ln C_{12} + (n-m)^2\ln C_{34} + m(n-m)\ln (C_{56}C_{78}) \\
&= \ln\binom{n}{m} + n^2 \bigg[\Big(\frac{m}{n}\Big)^2 \ln C_{12}
+ \Big(1 - \frac{m}{n}\Big)^2 \ln C_{34} \\
&+ \Big(\frac{m}{n}\Big)\Big(1 - \frac{m}{n}\Big) \ln (C_{56}C_{78}) \bigg] 
=  \ln\binom{n}{m} + n^2\eff\Big(\frac{m}{n}\Big).
\end{align*}

\noindent
Hence
\begin{equation}\label{ln of Z-n(m)}
\ln Z_n(m) = \ln\binom{n}{m} + n^2\eff\Big(\frac{m}{n}\Big).
\end{equation}

\begin{defin}\label{definition of rounding map}{\rm
For all $r \in \mbbR$ define
$
\lfloor r \rceil=\lfloor r +\dfrac{1}{2}\rfloor.
$
}\end{defin}

\begin{lem}\label{when the rounding is close to the real value}
If $0 \leq m \leq n$, $\alpha \in [0, 1]$ and $|m - \lfloor \alpha n \rceil |\leq 1$, then
\[
\left|\dfrac{1}{n^2}\ln Z_n(m)-\dfrac{1}{n^2}\ln Z_n(\lfloor \alpha n \rceil)\right|=O\left(\dfrac{1}{n}\right)
\ \text{  as } \ n\to\infty
\]
with the constant associated to $O(1/n)$ depending only on the fixed weights $w_1,\ldots,w_8$.
\end{lem}

\begin{proof}
Suppose that $\alpha \in [0, 1]$ and $|m - \lfloor \alpha n \rceil |\leq 1$.
By~(\ref{ln of Z-n(m)}) we have
\[
\ln Z_n(m) = \ln\binom{n}{m} + n^2\eff\Big(\frac{m}{n}\Big).
\]
We have $\binom{n}{m} \leq 2^n \leq e^n$ so $\ln \binom{n}{m} \leq n$ and hence
$\frac{\ln\binom{n}{m}}{n^2} \leq \frac{1}{n}$.
In the same way we get $\frac{\ln\binom{n}{\lfloor \alpha n \rceil}}{n^2} \leq \frac{1}{n}$.
From $|m - \lfloor \alpha n \rceil |\leq 1$ we get
$\big|\frac{m}{n} - \frac{\lfloor \alpha n \rceil}{n} \big| \leq \frac{1}{n}$.

Since $\eff'(\alpha)$ is bounded on $[0, 1]$ it follows that 
$\big|\eff\big(\frac{m}{n}\big) - \eff\big(\frac{\lfloor \alpha n \rceil}{n}\big)\big| = O(1/n)$
where the constant associated to $O(1/n)$ depends only on $\eff$ which in turn depends only on
the weights $w_1, \ldots, w_8$.
The lemma follows from these observations.
\end{proof}

\begin{defin}\label{not:H}{\rm
The {\em entropy function} $H$ is defined as follows:
\begin{align*}
&\text{For $\alpha\in(0,1)$, $H(\alpha)=-\alpha\ln\alpha-(1-\alpha)\ln(1-\alpha)$, and } \\
&H(0)=H(1)=0.
\end{align*}
}\end{defin}

\noindent
Then $H$ is continuous on $[0, 1]$.

\begin{lem}\label{lem:entropy-range}
We have
\[
0\ \le\ H(\alpha)\ \le\ \ln 2\qquad\text{for all }\alpha\in[0,1].
\]
Consequently, for all $\alpha,\beta\in[0,1]$,
\[
\bigl|H(\alpha)-H(\beta)\bigr|\ \le\ \ln 2.
\]
\end{lem}

\begin{proof}
Since
$H'(\alpha)=\ln \bigl(\frac{1-\alpha}{\alpha}\bigr)$ and
$H''(\alpha)=-\frac{1}{\alpha}-\frac{1}{1-\alpha}<0$ for $\alpha\in(0,1)$,
it follows that $H$ is strictly concave with unique maximum at $\alpha= 1/2$, 
and $H(1/2)=\ln 2$.
Since $H([0,1])\subseteq[0,\ln 2]$, the consequence
$\bigl|H(\alpha)-H(\beta)\bigr|\le \ln 2$ follows for all $\alpha,\beta\in[0,1]$.
\end{proof}

\begin{lem}\label{lem:stirling} {\rm (Stirling's approximation \cite{Robbins1955})}
For every $k \in \mbbN^+$ we have
\[
k! =
\sqrt{2\pi}\,k^{\,k+\frac{1}{2}}e^{-k}\,e^{r_k} 
\;\;\;\text{ where }\;\;\;
\dfrac{1}{12k+1}\ <\ r_k\ <\ \dfrac{1}{12k}.
\]
\end{lem}

\begin{lem}\label{lem:uniform}
For all $n\ge2$ and $1\le m\le n-1$, with $\alpha :=\dfrac{m}{n}$,
\[
\ln Z_n(m)=n^2\eff(\alpha)+nH(\alpha)+O(\ln n).
\]
At the endpoints, $\ln Z_n(0)=n^2\ln C_{34}$ and $\ln Z_n(n)=n^2\ln C_{12}$.
\end{lem}

\begin{proof}
By~(\ref{ln of Z-n(m)}),
\[
\ln Z_n(m)=\ln\binom{n}{m} + n^2\eff(\alpha).
\]
Direct calculations give $\ln Z_n(0)=n^2\ln C_{34}$ and $\ln Z_n(n)=n^2\ln C_{12}$.
So now it suffices to show, uniformly in $1\le m\le n-1$, that
\[
\ln\binom{n}{m}=nH(\alpha)+O(\ln n).
\]
By Lemma~\ref{lem:stirling},
for every $k\in\mbbN^+$ we have 
\[
k! =
\sqrt{2\pi}\,k^{\,k+\frac{1}{2}}e^{-k}\,e^{r_k} 
\;\;\;\text{where}\;\;\;
\dfrac{1}{12k+1}\ <\ r_k\ <\ \dfrac{1}{12k}. 
\]
Taking logarithms gives
\[
\ln k! = k\ln k-k+\frac{1}{2}\ln(2\pi k)+r_k.
\]
By applying this to $k=n$, $k=m$, and $k=n-m$, and subtracting we get:
\[
\begin{aligned}
\ln\binom{n}{m}&=\ln\frac{n!}{m!(n-m)!}
=\ln n!-\ln m!-\ln(n-m)!\\
&=\big[n\ln n-m\ln m-(n-m)\ln(n-m)\big]\\
&\quad +\frac12\big[\ln(2\pi n)-\ln(2\pi m)-\ln(2\pi(n-m))\big] \\
&\quad +\big(r_n-r_m-r_{\,n-m}\big).
\end{aligned}
\]
Writing $\alpha:=\dfrac{m}{n}$ so that
$m=\alpha n$ and $n-m=(1-\alpha)n$, we get
\begin{align*}
&n\ln n-m\ln m-(n-m)\ln(n-m) \\
= &n\ln n - \alpha n \ln (\alpha n) - (1-\alpha)n\ln ((1-\alpha)n) \\
= &n\Big(-\alpha\ln\alpha-(1-\alpha)\ln(1-\alpha)\Big)
= n\,H(\alpha),
\end{align*}
and
\[
\frac12\big[\ln(2\pi n)-\ln(2\pi m)-\ln(2\pi(n-m))\big]
= -\frac12\ln \big(2\pi n\,\alpha(1-\alpha)\big).
\]
Therefore, for $1\le m\le n-1$,
\[
\ln\binom{n}{m}
= nH(\alpha)-\dfrac{1}{2}\ln\big(2\pi n\,\alpha(1-\alpha)\big)
+\big(r_n-r_m-r_{n-m}\big),
\]
where $r_n-r_m-r_{n-m} = O(1)$.
\end{proof}

\begin{lem}\label{lem:ratio}
Let $m_n^\star\in\arg\max_{0\le m\le n} Z_n(m)$.
Then for all $m\in\{0,\dots,n\}$,
\[
\dfrac{Z_n(m)}{Z_n(m_n^\star)}
\ \le\ \exp\Big(-n^2\big[\eff\Big(\dfrac{m_n^\star}{n}\Big)-\eff\Big(\dfrac{m}{n}\Big)\big]+n\ln 2\Big).
\]

\end{lem}

\begin{proof}
From \eqref{ln of Z-n(m)},
\[
\ln\dfrac{Z_n(m)}{Z_n(m_n^\star)}
=\Big[\ln\binom{n}{m}-\ln\binom{n}{m_n^\star}\Big]
+n^2\Big[\eff\Big(\dfrac{m}{n}\Big)-\eff\Big(\dfrac{m_n^\star}{n}\Big)\Big].
\]
Since $\ln\binom{n}{m}\le n\ln 2$ and $\ln\binom{n}{m_n^\star}\ge 0$, we obtain
\[
\ln\dfrac{Z_n(m)}{Z_n(m_n^\star)}
\ \le\ n\ln 2+n^2\Big[\eff\Big(\dfrac{m}{n}\Big)-\eff\Big(\dfrac{m_n^\star}{n}\Big)\Big].
\]
\end{proof}

\begin{lem}\label{lem:max-gap}
Suppose that $\eff$ is nonconstant. For all large enough $n \in \mbbN$,
if  $m_n^\star\in\arg\max_{0\le m\le n}Z_n(m)$ then
\[
\eff\Big(\frac{m_n^\star}{n}\Big) \ge \max_{\alpha\in[0,1]} \eff(\alpha) - O \Big(\dfrac{1}{n}\Big).
\]
\end{lem}

\begin{proof}
Recall from~\eqref{ln of Z-n(m)} that for every $n\in\mbbN^+$ and every $0\le m\le n$,
\[
\ln Z_n(m)=\ln\binom{n}{m}+n^2\,\eff\Big(\frac{m}{n}\Big).
\]
Since $1\le \binom{n}{m}\le 2^n$, we have the uniform bounds
\[
0 \le \ln\binom{n}{m} \le n\ln 2.
\]
Let
\[
M := \max_{\alpha\in[0,1]}\eff(\alpha)
\qquad\text{and}\qquad
M_n := \max_{0\le m\le n}\eff\Big(\frac{m}{n}\Big).
\]
Then $M_n\le M$. Fix $\alpha^\star\in\Eff^{\max}$ and let $m^\star:=\lfloor \alpha^\star n\rceil$.
By Definition~\ref{definition of rounding map}, we have $m^\star=\lfloor\alpha^\star n+\frac{1}{2}\rfloor$, therefore $\lfloor\alpha^\star n+\frac{1}{2}\rfloor \leq \alpha^\star n+\frac{1}{2}<\lfloor\alpha^\star n+\frac{1}{2}\rfloor +1$, so $m^\star\leq\alpha^\star n+\frac{1}{2}<m^\star+1$, thus  $|m^\star -\alpha^\star n|\leq \frac{1}{2}$.
Now dividing by $n>0$ gives $\big|\frac{m^\star}{n}-\alpha^\star\big|\le \frac{1}{2n}$.
Since $\eff$ is a polynomial, $\eff'$ is bounded on $[0,1]$.
Set
\[
K:=\sup_{\alpha\in[0,1]}|\eff'(\alpha)|<\infty.
\]
By the mean value theorem,
\[
\Big|\eff\Big(\frac{m^\star}{n}\Big)-\eff(\alpha^\star)\Big|
\le
K\Big|\frac{m^\star}{n}-\alpha^\star\Big|
\le
\frac{K}{2n}.
\]
Because $\eff(\alpha^\star)=M$, this implies
\begin{equation}\label{f bounded below by M minus something}
\eff\Big(\frac{m^\star}{n}\Big)\ge M-\frac{K}{2n}.
\end{equation}

Let $m_n^\star\in\arg\max_{0\le m\le n}Z_n(m)$ and set $\alpha_n^\star:=\frac{m_n^\star}{n}$.
Since $m_n^\star$ maximizes $Z_n(m)$, we have $Z_n(m_n^\star)\ge Z_n(m^\star)$, hence
\[
\ln Z_n(m_n^\star)\ge \ln Z_n(m^\star)
\]
which gives
\[
\ln\binom{n}{m_n^\star}+n^2\eff(\alpha_n^\star)
\ge
\ln\binom{n}{m^\star}+n^2\eff\Big(\frac{m^\star}{n}\Big).
\]
So,
\[
n^2\eff(\alpha_n^\star)
\ge
n^2\eff\Big(\frac{m^\star}{n}\Big)
+\ln\binom{n}{m^\star}-\ln\binom{n}{m_n^\star}.
\]
Using $\ln\binom{n}{m^\star}\ge 0$ and $\ln\binom{n}{m_n^\star}\le n\ln 2$, we get
\[
n^2\eff(\alpha_n^\star)
\ge
n^2\eff\Big(\frac{m^\star}{n}\Big)-n\ln 2,
\]
so
\[
\eff(\alpha_n^\star)
\ge
\eff\Big(\frac{m^\star}{n}\Big)-\frac{\ln 2}{n}.
\]
This together with \eqref{f bounded below by M minus something} gives
\begin{align*}
\eff(\alpha_n^\star)
&\ge
\eff\Big(\frac{m^\star}{n}\Big)-\frac{\ln 2}{n}
\ge
M-\frac{\ln 2+K/2}{n} \\
&=
\max_{\alpha\in[0,1]}\eff(\alpha)-O\Big(\frac{1}{n}\Big).
\end{align*}

\end{proof}

\section{Limit behaviour of the colour proportion}\label{sec:limit}

\noindent
We prove Theorem~\ref{the colour distribution, exact formulation} in this section.
Recall that 
\[
\Eff^{\max}=\underset{\alpha\in[0,1]}{\arg\max} \, \eff(\alpha).
\]
Each case of Theorem~\ref{the colour distribution, exact formulation} 
follows from one of the lemmas below.
First we consider the simplest case.

\begin{lem}\label{colour proportion when psi is constant}
Suppose that $\Eff^{\max} = [0, 1]$, which is equivalent to $\eff$ being constant.
Then, for every $\varepsilon > 0$,
\[
\lim_{n\to\infty} \mbbP_n \Big(
  \big\{\mcA \in \mb W_n :
    (1/2 - \varepsilon)n
      \le \big|P^\mcA\big|
      \le (1/2 + \varepsilon)n
  \big\}
\Big) = 1.
\]
\end{lem}

\begin{proof}
By Lemma~\ref{prop:maximiser},
$C_{12}C_{34} = C_{56}C_{78}$ and $C_{12} = C_{34}$.
Let $C = C_{12}$, so $C^2 = C_{56}C_{78}$.
It follows from~\eqref{eq:Znm-exact}
that, for all $0 \leq m \leq n$,
\begin{align*}
Z_n(m) 
&= \binom{n}{m} C_{12}^{m^2}C_{34}^{(n-m)^2}C_{56}^{m(n-m)}C_{78}^{m(n-m)} \\
&= \binom{n}{m} C^{m^2}C^{(n-m)^2}C^{2m(n-m)}
= \binom{n}{m} C^{n^2}
\end{align*}
and hence
\[
Z_n = \sum_{m=0}^n Z_n(m) = \sum_{m=0}^n \binom{n}{m} C^{n^2} = 2^n  C^{n^2}.
\]
Therefore 
\[
\mbbP_n\big(\mbW_n(m)\big) = \frac{Z_n(m)}{Z_n} = \frac{\binom{n}{m}}{2^n}.
\]
Hence  $\mbbP_n\big(\mbW_n(m)\big)$ is the probability of having exactly $m$ successes in a series of
$n$ independent trials with success probability $1/2$. Thus the lemma follows from the 
law of large numbers.
\end{proof}

\noindent
The next result tells that if $\eff$ is not constant, then, with probability tending to 1 (as $n\to\infty$),
the colour proportion will be close to the maximal value(s) that $\eff$ takes on $[0, 1]$.

\begin{prop}\label{thm:concentration}
Assume $\eff$ is nonconstant. For all sufficiently small $\delta>0$ the number 
\[
c_\delta := 
\inf\Big\{\big(\max_{\beta \in [0,1]}\eff(\beta)\big) - \eff(\gamma):\ 
\gamma \in [0, 1] \text{ and } \inf_{\beta\in\Eff^{\max}}|\gamma-\beta|\ge\delta \Big\}
\]
is well-defined and positive and we have
\[
\mbb P_n \left(\Big\{\mcA \in \mb W_n : 
\inf_{\beta\in\Eff^{\max}} \Big| \frac{|P^\mcA|}{n} - \beta \Big| \ge \delta\Big\}\right) \ 
\le \ (n+1)\,\exp \big(-c_\delta n^2 + O(n)\big).
\]
\end{prop}

\begin{proof}
Fix $\delta>0$ and let
\[
S_\delta  :=  \Big\{\alpha\in[0,1]:\ \inf_{\beta\in\Eff^{\max}}|\alpha-\beta|\ge\delta\Big\}.
\]
Since $\eff$ is a polynomial of degree at most 2 and (by assumption) nonconstant it is clear that 
$S_\delta \neq \varnothing$ if $\delta$ is small enough.
Let $m_n^\star\in\arg\max_{0\le m\le n} Z_n(m)$ and write $\alpha_n^\star= \dfrac{m_n^\star}{n}$. 
By Lemma~\ref{lem:max-gap}, there exists a constant $C>0$ (depending only on the fixed weights) such that
\begin{equation}\label{eq:pf:nearmax}
\eff(\alpha_n^\star)\ \ge\ \max_{\alpha \in [0,1]}\eff(\alpha) \ - \ \dfrac{C}{n}.
\end{equation}
Suppose that
$m\in\{0,1,\dots,n\}$ and that $\alpha := \dfrac{m}{n}\in S_\delta$. 
Lemma~\ref{lem:ratio} gives
\begin{equation}\label{Z-n-m divided by Z-n-m-star}
\dfrac{Z_n(m)}{Z_n(m_n^\star)}
\ \le\ \exp \Big(-n^2\big[\eff(\alpha_n^\star)-\eff(\alpha)\big]+n\ln 2\Big).
\end{equation}
Define
\[
c_\delta\ =\ \inf\Big\{\big(\max_{\beta \in [0,1]}\eff(\beta)\big) - \eff(\gamma):\ \gamma\in S_\delta\Big\}.
\]
As $\eff$ is a nonconstant polynomial of degree at most 2 and we have chosen $\delta > 0$
small enough that $S_\delta \neq \varnothing$ it follows that $c_\delta > 0$.
Combining the definition of $c_\delta$ with \eqref{eq:pf:nearmax} gives
\[
\eff(\alpha_n^\star)-\eff(\alpha) \geq  \max_{\beta \in [0, 1]} \eff(\beta) - \eff(\alpha) - \frac{C}{n}
\geq c_\delta - \frac{C}{n}.
\]
Combining this with \eqref{Z-n-m divided by Z-n-m-star} gives
\begin{equation}\label{eq:pf:ratio-final}
\dfrac{Z_n(m)}{Z_n(m_n^\star)}
\ \le\ \exp \Big(-n^2\big[c_\delta-\dfrac{C}{n}\big]+n\ln 2\Big)
\ =\ \exp \big(-c_\delta n^2 + O(n)\big).
\end{equation}
Since $Z_n\ge Z_n(m_n^\star)$, we have
\[
\mbb P_n \big(\big\{\mc A\in \mb W_n: \big|P^\mcA\big| = m \big\}\big)
=\dfrac{Z_n(m)}{Z_n}
\ \le\ \dfrac{Z_n(m)}{Z_n(m_n^\star)}.
\]
Summing \eqref{eq:pf:ratio-final} over at most $n+1$ values of $m \in \{0, \ldots, n\}$ 
with $\alpha:=\dfrac{m}{n}\in S_\delta$ yields
\[
\mbb P_n \Big(\{\mc A \in \mb W_n:\inf_{\beta\in\Eff^{\max}}\big||P^\mcA|/n -\beta\big|\ge\delta\}\Big)
\ \le\ (n+1)\,\exp \big(-c_\delta n^2 + O(n)\big).
\]
\end{proof}

\noindent
The remaining results of this section give more precise information about the colour proportion
in certain cases. This information will be used to prove logical convergence laws in those cases.

\begin{lem}\label{lem:pinning} $\text{  }$ \\
(a) If $\Eff^{\max} = \{0\}$ and $\eff'(0) < 0$, then 
$\lim_{n\to\infty} \mbbP_n\big(\mbW_n(0)\big) = 1$. \\
(b) If $\Eff^{\max} = \{1\}$ and $\eff'(1) > 0$, then 
$\lim_{n\to\infty} \mbbP_n\big(\mbW_n(n)\big) = 1$. 
\end{lem}

\begin{proof}
(a) Suppose that $\Eff^{\max} = \{0\}$ and that $\eff'(0) < 0$.
Recall that $\eff(\alpha) = c_2\alpha^2 + c_1 \alpha + c_0$. 
As $\eff'(0) = c_1$ it follows that $c_1 < 0$.
We have $\eff(1/n) - \eff(0) = \frac{c_2}{n^2} + \frac{c_1}{n}$ and hence
$n^2\eff(1/n) - n^2\eff(0) = c_2 + c_1 n$.
By Lemma~\ref{lem:uniform}
we have 
\[
\ln Z_n(m) = n^2\eff\Big(\frac{m}{n}\Big) + nH\Big(\frac{m}{n}\Big) + O(\ln n).
\]
Using this we get 
\begin{align}\label{comparing Z-n(m) and Z-n(0)}
\ln\frac{Z_n(m)}{Z_n(0)} &= \ln Z_n(m) - \ln Z_n(0) \\
&= 
n^2\Big[\eff\Big(\frac{m}{n}\Big) - \eff(0)\Big] + nH\Big(\frac{m}{n}\Big) \pm O(\ln n). \nonumber
\end{align}
Since $\eff'(0) < 0$ it follows that $\eff$ is decreasing in an open interval containing $0$.
Then we can choose $\delta > 0$ so that $\delta < |c_1|$ and $\eff$ is decreasing on $[-\delta, \delta]$.
As $H$ is continuous and $H(0) = 0$ we can also let $\delta$ be small enough so that 
$H(\alpha) < |c_1| - \delta$ for all $\alpha \in [0, \delta]$.
Suppose that  $1 \leq m \leq \delta n$.
Then $1/n \leq m/n \leq \delta$ so $\eff(m/n) \leq \eff(1/n)$ and using
\eqref{comparing Z-n(m) and Z-n(0)}
we get 
\begin{align*}
\ln\frac{Z_n(m)}{Z_n(0)} &\leq n^2\Big[\eff\Big(\frac{1}{n}\Big) - \eff(0)\Big] + 
n(|c_1| - \delta) \pm O(\ln n) \\
&= c_2 + c_1n + n(|c_1| - \delta) \pm O(\ln n)  = -\delta n \pm O(\ln n).
\end{align*}
It follows that 
\begin{align*}
&\mbbP_n\big(\big\{\mcA \in \mbW_n : 1 \leq |P^\mcA| \leq \lfloor \delta n \rfloor\big\}\big) =
\frac{\sum_{m=1}^{\lfloor \delta n \rfloor} Z_n(m)}{Z_n} 
\leq \frac{\sum_{m=1}^{\lfloor \delta n \rfloor} Z_n(m)}{Z_n(0)} \\
&\leq n e^{-\delta n \pm O(\ln n)} \to 0 \quad \text{ as } n \to \infty.
\end{align*}
It now suffices to show that
\[
\lim_{n\to\infty} \mbbP_n\big(\big\{\mcA \in \mbW_n : \lfloor \delta n \rfloor \leq |P^\mcA| \big\}\big) = 0.
\]
But this follows from Theorem~\ref{thm:concentration}.

(b) Suppose that $\Eff^{\max} = \{1\}$ and $\eff'(1) > 0$.
This case is ``symmetric'' to the previous and we can argue in essentially the same way 
(letting $n-1$, respectively $n$, have the roles of $1$, respectively $0$, in case (a)) and show that, 
for sufficiently small $\delta > 0$,
\[
\lim_{n\to\infty} \mbbP_n\big(\big\{\mcA \in \mbW_n : (1-\delta)n \leq |P^\mcA| \leq n-1 \big\}\big) = 0.
\]
We leave the details to the reader.
\end{proof}

\begin{lem}\label{distribution of P when 0 and 1 are maxpoints}
If $\Eff^{\max} = \{0, 1\}$
then 
\[
\lim_{n\to\infty} \mbbP_n\big(\mbW_n(0)\big) = \lim_{n\to\infty} \mbbP_n\big(\mbW_n(n)\big) = 1/2.
\]
\end{lem}

\begin{proof}
Suppose that $\Eff^{\max} = \{0, 1\}$.
Then $\eff(0) = \eff(1)$, $\eff'(0) < 0$ and $\eff'(1) > 0$ and $\eff$ has its global minimum in $\alpha = 1/2$.
Since $\eff$ is a polynomial of degree 2 it follows that, for all $\alpha \in [0, 1]$, $\eff(\alpha) = \eff(1-\alpha)$.
Recall that from \eqref{ln of Z-n(m)} we have
$Z_n(m) = \exp\big(\ln\binom{n}{m} + n^2\eff(m/n)\big)$.
Hence, for all $0 \leq m \leq n$,
\begin{align*}
Z_n(n-m) &= \exp\bigg(\ln\binom{n}{n-m} + n^2\eff((n-m)/n)\bigg)  \\
&= \exp\bigg(\ln \binom{n}{m} + n^2\eff(m/n)\bigg) = Z_n(m).
\end{align*}
Therefore it suffices to prove that
\[
\lim_{n\to\infty} \mbbP_n\bigg(\bigcup_{m=1}^{\lceil n/2 \rceil} \mbW_n(m)\bigg) = 0.
\]
It follows from Theorem~\ref{thm:concentration}
that for all $\delta > 0$, 
\[
\lim_{n\to\infty} \mbbP_n\bigg(\bigcup_{\delta n \leq m \leq (1-\delta)n} \mbW_n(m)\bigg) = 0.
\]
Hence it suffices to prove that for sufficiently small $\delta > 0$,
\[
\lim_{n\to\infty} \mbbP_n\bigg(\bigcup_{1 \leq m \leq \delta n} \mbW_n(m)\bigg) = 0.
\]
Since $\eff'(0) < 0$, so $\eff$ is decreasing on $(-\delta, \delta)$ if $\delta > 0$ is small enough,
we can reason just as in the proof of 
Lemma~\ref{lem:pinning}~(a) 
to get this conclusion.
\end{proof}

\begin{lem}\label{lem:caseC-window-rigorous}
(a) If $\Eff^{\max} = \{0\}$ and $\eff'(0) = 0$ then, for all $\varepsilon \in (0, \tfrac{1}{2})$,
\begin{equation*}
\begin{aligned}
\lim_{n\to\infty} \mbbP_n \Big(
  \big\{\mc A \in \mb W_n :
    \big\lfloor\tfrac{1 - \varepsilon}{2|c_2|}\ln n \big\rfloor
      \le \big|P^\mcA\big|
      \le \big\lceil \tfrac{1 + \varepsilon}{2|c_2|}\ln n \big\rceil
  \big\}
\Big)
\ = \ 1.
\end{aligned}
\end{equation*}
(b) If  $\Eff^{\max} = \{1\}$ and $\eff'(1) = 0$ then, for all $\varepsilon \in (0, \tfrac{1}{2})$,
\begin{equation*}
\begin{aligned}
\lim_{n\to\infty} \mbbP_n \Big(
  \big\{\mc A \in \mbW_n :
    n - \big\lceil \tfrac{1 + \varepsilon}{2|c_2|}\ln n \big\rceil
      \le \big|P^\mcA\big|
      \le n - \big\lfloor \tfrac{1 - \varepsilon}{2|c_2|}\ln n \big\rfloor
  \big\}
\Big)
\ = \ 1.
\end{aligned}
\end{equation*}
\end{lem}

\begin{proof}
(a) Suppose that $\Eff^{\max} = \{0\}$ and $\eff'(0) = 0$.
Then $c_2<0$ and $c_1=0$ so
\[
\eff(\alpha)=c_2\alpha^2+c_0.
\]
For $0\le m\le n-1$ define
\begin{equation}\label{definition of h-n(m)}
h_n(m) := \ln\dfrac{Z_n(m+1)}{Z_n(m)}.
\end{equation}
So the sign of $h_n(m)$ tells if $Z_n(m+1) > Z_n(m)$, $Z_n(m+1) = Z_n(m)$,
or $Z_n(m+1) < Z_n(m)$.
It follows that (if $0\le m\le n-1$)
\begin{equation}\label{eq:block-hn}
\begin{aligned}
h_n(m) =& \ln(Z_n(m+1)) - \ln(Z_n(m)) \\
=& \ln\binom{n}{m+1} + n^2 \eff\Big(\frac{m+1}{n}\Big) - \ln\binom{n}{m} - n^2\eff\Big(\frac{m}{n}\Big) \\
=& \ln \Bigg(
   \frac{n!}{(m+1)!(n-m-1)!}
   \cdot
   \frac{m!(n-m)!}{n!}
   \Bigg) \\
   &+ n^{2} \Bigg(
   \eff \Big(\dfrac{m+1}{n}\Big)
   - \eff \Big(\dfrac{m}{n}\Big)
   \Bigg) \\
=& \ln\dfrac{n-m}{m+1}
   + n^{2} \Bigg(
   \eff \Big(\dfrac{m+1}{n}\Big)
   - \eff \Big(\dfrac{m}{n}\Big)
   \Bigg) \\
 =& \ln\dfrac{n-m}{m+1} + c_2 (2m + 1).
\end{aligned}
\end{equation}
Recalling that $c_2 < 0$, it follows that for $0\le m\le n-2$,
\[
h_n(m + 1)-h_n(m)=\ln\frac{n-m-1}{n-m}+\ln\frac{m+1}{m+2}+2c_2 \ < \ 0.
\]
Hence $m\mapsto h_n(m)$ is strictly decreasing on $\{0,1,\dots,n-1\}$. 
A strictly decreasing sequence can pass from being positive to being negative at most once.
Therefore there exists $m_\star\in\{0,\dots,n-1\}$ such that $h_n(m)>0$ for $m<m_\star$, $h_n(m_\star)\ge 0\ge h_n(m_\star{+}1)$, and $h_n(m)<0$ for $m>m_\star$. Equivalently,
\[
Z_n(0)<Z_n(1)<\cdots<Z_n(m_\star) \ (\le)\ \ Z_n(m_\star{+}1)>\cdots>Z_n(n),
\]
with two maxima $m_\star$ and $m_\star + 1$ only when $h_n(m_\star)=0$.

Fix $k\in\mbb N$. For $0\le j\le n-1$ we have
\[
h_n(j)=\ln\frac{n-j}{j+1}+2c_2\,j+c_2.
\]
So for each fixed $j\le k$, $\lim_{n\to\infty} h_n(j) = \infty$.
Hence there exists $N(k)$ such that for all $n\ge N(k)$ we have $h_n(j)>0$ for every $0\le j\le k$. 
It follows that if  $n \geq N(k)$ then
\[
Z_n(0)<Z_n(1)<\cdots<Z_n(k)<Z_n(k{+}1).
\]
Moreover, since $h_n(m)$ is strictly decreasing in $m$ for all $0\le m\le k$, we have
\begin{equation}\label{Z-n(m) over Z-n(k+1)}
\begin{aligned}
&\frac{Z_n(m)}{Z_n(k + 1)} =
\frac{Z_n(m)}{Z_n(m + 1)} \cdot \frac{Z_n(m+1)}{Z_n(m + 2)} \cdot \\ 
&\ldots \cdot 
\frac{Z_n(k - 1)}{Z_n(k)} \cdot \frac{Z_n(k)}{Z_n(k + 1)} = \\
&=\prod_{r=m}^k e^{-h_n(r)} = 
\exp \Big(-\sum_{r=m}^{k} h_n(r)\Big)\ \\
&\le\ \exp \big(-(k{+}1-m)\,h_n(k)\big)\ \le\ e^{-h_n(k)}.
\end{aligned}
\end{equation} 

Let $\varepsilon \in (0, 1/2)$ and define 
\[
m_- := m_-(n) := \bigg\lfloor \frac{1 - \varepsilon}{2|c_2|}\ln n \bigg\rfloor \ \text{ and } \
m_+ := m_+(n) := \bigg\lceil \frac{1 + \varepsilon}{2|c_2|}\ln n \bigg\rceil.
\]
Part~(a) of the lemma follows directly from claims~1 and~2 below.

\medskip
\noindent
{\bf Claim 1.} $\lim_{n\to\infty} \mbbP_n\bigg(\bigcup_{0 \leq m \leq m_-(n)} \mbW_n(m)\bigg) = 0$.
\smallskip

\noindent
We now prove Claim 1.
By the definition of $m_-(n)$ we have
\begin{equation}\label{properties of m-minus}
\begin{aligned}
&m_-(n)=\frac{1-\varepsilon}{2|c_2|} \ln n+O(1), \quad
\lim_{n\to\infty}m_-(n) = \infty, \text{ and } \\
&\lim_{n\to\infty} \frac{m_-(n)}{n} = 0.
\end{aligned}
\end{equation}
We have
\[
\ln\frac{n-m_-}{m_-+1}
=\ln(n-m_-)-\ln(m_-+1)
\]
and since $m_-(n) =O(\ln n)$ and $\dfrac{m_-(n)}{n}\to 0$ we get
(with the abbreviation $m_- := m_-(n)$)
\[
\ln(n-m_-)=\ln n+\ln\Bigl(1-\frac{m_-}{n}\Bigr)
=\ln n+O\Bigl(\frac{m_-}{n}\Bigr)
=\ln n+o(1),
\]
and since $\lim_{n\to\infty} m_-(n) = \infty$,
\[
\ln(m_-+1)=\ln m_-+\ln\Bigl(1+\frac{1}{m_-}\Bigr)
=\ln m_-+O\Bigl(\frac{1}{m_-}\Bigr)
=\ln m_-+o(1).
\]
Thus
\[
\ln\frac{n-m_-}{m_-+1}
=\ln n-\ln m_-+o(1).
\]
Moreover, using \eqref{properties of m-minus}, we get
\[
2c_2 m_-
=-2|c_2|\,m_-
=-(1-\varepsilon)\ln n+O(1).
\]
Substituting into \eqref{eq:block-hn} at $m=m_-$ gives
\[
\begin{aligned}
h_n(m_-)
&=\ln\frac{n-m_-}{m_-+1}+2c_2 m_-+c_2\\[3pt]
&=\bigl(\ln n-\ln m_-+o(1)\bigr)
  - (1-\varepsilon)\ln n+O(1)\\[3pt]
&=\varepsilon\ln n-\ln m_-+o(\ln n).
\end{aligned}
\]
Since $m_-=\Theta(\ln n)$, we have 
\[
\ln m_-=\ln\ln n+O(1)=o(\ln n),
\]
so there exists $n_1$ such that, for all $n\ge n_1$,
\begin{equation}\label{eq:caseC-hn-mminus-lb-new}
h_n(m_-)\ \ge\ \frac{\varepsilon}{2}\,\ln n.
\end{equation}

Because $h_n$ is strictly decreasing on $\{0, \ldots, n-1\}$
it follows that for all $0\le r\le m_-$ and $n\ge n_1$,
\begin{equation}\label{eq:caseC-hn-all-small-new}
h_n(r)\ \ge\ h_n(m_-)\ \ge\ \frac{\varepsilon}{2}\,\ln n.
\end{equation}
Using \eqref{Z-n(m) over Z-n(k+1)}, it follows that for all $m\le m_-$,
\[
\frac{Z_n(m)}{Z_n(m_- + 1)}
=\exp\biggl(-\sum_{r=m}^{m_-}h_n(r)\biggr),
\]
and by \eqref{eq:caseC-hn-all-small-new},
\[
\sum_{r=m}^{m_-}h_n(r)
\ \ge\ (m_- + 1 -m)\,\frac{\varepsilon}{2}\,\ln n,
\]
so
\begin{equation*}\label{eq:caseC-lower-tail-ratio-new}
\frac{Z_n(m)}{Z_n(m_- + 1)}
\ \le\ n^{-(\frac{\varepsilon}{2})(m_- + 1 - m)}.
\end{equation*}
Since $Z_n\ge Z_n(m_- + 1)$, we obtain
\begin{align*}
&\mbbP_n\Bigg(\bigcup_{0 \leq m \leq m_-(n)} \mbW_n(m)\Bigg) \ 
= \ \sum_{m=0}^{m_-(n)}\frac{Z_n(m)}{Z_n}
 \ \le\ \sum_{m=0}^{m_-(n)}\frac{Z_n(m)}{Z_n(m_-(n)  + 1)}\\
& \leq \ \sum_{m=0}^{m_-(n)} n^{-(\frac{\varepsilon}{2})(m_- + 1 - m)}
 \ \le \ \sum_{k = 1}^\infty n^{-(\frac{\varepsilon}{2})k}
 = \frac{n^{-\frac{\varepsilon}{2}}}{1-n^{-\frac{\varepsilon}{2}}}.
\end{align*}
Since $n^{-\frac{\varepsilon}{2}}/(1-n^{-\frac{\varepsilon}{2}}) \to 0$ as $n\to\infty$ the claim is proved.

\medskip
\noindent
{\bf Claim 2.} $\lim_{n\to\infty} \mbbP_n\bigg(\bigcup_{m_+(n) \leq m \leq n} \mbW_n(m)\bigg) = 0$.
\smallskip

\noindent
We argue similarly as when we proved Claim~1.
By the definition of $m_+(n)$ we have
\begin{equation}\label{properties of m-plus}
\begin{aligned}
&m_+(n)=\frac{1+\varepsilon}{2|c_2|} \ln n+O(1),
\quad
\lim_{n\to\infty}m_+(n) = \infty,\ \ \text{ and } \\ 
&\lim_{n\to\infty} \frac{m_+(n)}{n} = 0.
\end{aligned}
\end{equation}
It follows (with the abbreviation $m_+ = m_+(n)$) that
\[
2c_2 m_+
=-2|c_2|\,m_+
=-(1+\varepsilon)\ln n+O(1).
\]
Similarly as in the proof of Claim~1, we get
\[
\ln\frac{n-m_+}{m_++1}
=\ln n-\ln m_+ + o(1).
\]
Therefore, by using \eqref{eq:block-hn}, we get
\[
\begin{aligned}
h_n(m_+)
&=\ln\frac{n-m_+}{m_++1}+2c_2 m_+ + c_2\\[3pt]
&=\bigl(\ln n-\ln m_++o(1)\bigr)
  -(1+\varepsilon)\ln n+O(1)\\[3pt]
&=-\varepsilon\ln n-\ln m_+ + o(\ln n).
\end{aligned}
\]
Since $m_+=\Theta(\ln n)$, we have $\ln m_+=o(\ln n)$, so there exists $n_2$ such that, for all $n\ge n_2$,
\begin{equation}\label{eq:caseC-hn-mplus-better}
h_n(m_+)\ \le\ -\frac{3\varepsilon}{4}\,\ln n.
\end{equation}

For every integer $k\ge 2$ we have $\dfrac{k-1}{k}\in[\frac{1}{2},1)$, so
\[
\bigl|\ln\dfrac{k-1}{k}\bigr|\le \ln 2.
\]
Applying this with $k=n-m$ and $k=m+2$ gives
\[
\Bigl|\ln\frac{n-m-1}{n-m}\Bigr|\le \ln 2,
\qquad
\Bigl|\ln\frac{m+1}{m+2}\Bigr|\le \ln 2,
\]
and therefore, for  all $n$ and all $0\le m\le n-2$,
\begin{equation*}
\bigl|h_n(m+1)-h_n(m)\bigr|
\le 2\ln 2 + 2|c_2|.
\end{equation*}
Let
$L := 2\ln 2 + 2|c_2|$ so $\bigl|h_n(m+1)-h_n(m)\bigr| \leq L$ for all $0\le m\le n-2$.
Then
\begin{align*}
h_n(m_+-1)
&= h_n(m_+) + \bigl(h_n(m_+-1)-h_n(m_+)\bigr) \\
&\le\ h_n(m_+) + L \ \le \ -\frac{3\varepsilon}{4}\ln n + L.
\end{align*}
Combining this with \eqref{eq:caseC-hn-mplus-better}, there is $n_3$ such that, for all $n\ge n_3$,
\[
L\ \le\ \frac{\varepsilon}{4}\,\ln n,
\]
and hence
\begin{equation}\label{eq:caseC-hn-mplusminus1-ub-new}
h_n(m_+-1)
\ \le\ -\frac{3\varepsilon}{4}\,\ln n + \frac{\varepsilon}{4}\,\ln n
\ =\ -\frac{\varepsilon}{2}\,\ln n.
\end{equation}

Now fix $n\ge \max\{n_2,n_3\}$. 
From \eqref{Z-n(m) over Z-n(k+1)} and the (above proved) fact that $h_n(m)$ is strictly decreasing it follows that
for all $m\ge m_++1$,
\[
\ln\frac{Z_n(m)}{Z_n(m_+)}
= \sum_{r=m_+}^{m-1} h_n(r)
\ \le\ (m-m_+)\,h_n(m_+)
\ \le\ - (m-m_+)\,\frac{\varepsilon}{2}\,\ln n
\]
Thus, for $m_+ +1 \leq m \leq n$, 
\begin{equation*}\label{eq:caseC-upper-tail-ratio-new}
\frac{Z_n(m)}{Z_n(m_+)}
\ \le\ n^{-(\frac{\varepsilon}{2})(m-m_+)}.
\end{equation*}
By \eqref{eq:caseC-hn-mplusminus1-ub-new} we also have
\[
\frac{Z_n(m_+)}{Z_n}
\ \le\ \frac{Z_n(m_+)}{Z_n(m_+-1)}
= \exp\bigl(h_n(m_+-1)\bigr)
\ \le\ n^{-\frac{\varepsilon}{2}}.
\]
Since $Z_n\ge Z_n(m_+(n))$, we get
\begin{align*}
&\mbb P_n\Bigg(\bigcup_{m_+(n) \leq m \leq n} \mbW_n(m)\Bigg)
= \sum_{m=m_+}^{n}\frac{Z_n(m)}{Z_n} \\
&\le \frac{Z_n(m_+)}{Z_n}
    +\sum_{m=m_++1}^{n}\frac{Z_n(m)}{Z_n(m_+)} \nonumber \\
&\le n^{-\frac{\varepsilon}{2}}
    +\sum_{k = 1}^\infty n^{-(\frac{\varepsilon}{2})k}
\ \le \ n^{-\frac{\varepsilon}{2}} + \frac{n^{-\frac{\varepsilon}{2}}}{1-n^{-\frac{\varepsilon}{2}}}
\nonumber
\end{align*}
Since the last expression above tends to 0 as $n \to \infty$, Claim~2 is proved.
Part~(a) of Lemma~\ref{lem:caseC-window-rigorous} 
follows immediately from claims~1 and~2.

\medskip

(b) Now suppose that $\Eff^{\max} = \{1\}$ and $\eff'(1) = 0$.
Then there are constants $d_2 < 0$ and $d_0 > 0$ such that
\[
\eff(\alpha) = d_2(\alpha - 1)^2 + d_0.
\]
This case is ``symmetric'' to the previous case if we consider ``uncoloured'' vertices instead of ``coloured'',
that is, if we consider $n - m$ (where $m$ is the number of coloured vertices) instead of $m$ in part~(a) of
Lemma~\ref{lem:caseC-window-rigorous}.
We have 
\begin{align*}
\ln Z_n(m) &= \ln\binom{n}{m} + n^2 \eff\Big(\frac{m}{n}\Big) 
=  \ln\binom{n}{m} + n^2 \Bigg(d_2\bigg(\frac{m}{n} - 1\bigg)^2 + d_0 \Bigg) \\
&= \ln\binom{n}{n-m} + n^2 \Bigg(d_2\bigg(\frac{n-m}{n}\bigg)^2 + d_0 \Bigg)
\end{align*}
so we can argue as in part~(a) with $n-m$ in place of $m$.
Therefore, with the same notations $m_-(n)$ and $m_+(n)$ as above, the arguments proving part~(a) show that 
\[
\lim_{n\to\infty} \mbbP_n\Big( \big\{\mcA \in \mbW_n : m_-(n) \leq n - |P^\mcA| \leq m_+(n) \big\}\big)
= 1
\]
and it follows that
\[
\lim_{n\to\infty} \mbbP_n\Big( \big\{\mcA \in \mbW_n : n - m_+(n) \leq |P^\mcA| \leq n - m_-(n) \big\}\big)
= 1.
\]
This concludes the proof of Lemma~\ref{lem:caseC-window-rigorous}.
\end{proof}

\section{Proofs of Theorem~\ref{the colour proportion in scaled case}
and~\ref{0-1 law in scaled case}}\label{sec:one-over-n-scaling}

\noindent
In this section we first prove
Theorem~\ref{the colour proportion in scaled case} which is the following statement:

\medskip
\noindent
{\em For every choice of the weights $w_1, \ldots, w_8 \geq 0$, there is $\varepsilon > 0$ such that
\[
\lim_{n\to\infty} \widetilde{\mbbP}_n \Big(
  \big\{\mcA \in \mb W_n :
    \varepsilon n
      \le \big|P^\mcA\big|
      \le (1 - \varepsilon)n
  \big\}
\Big) = 1.
\]
}

\medskip

 \noindent
The proof will be completed by 
Corollary~\ref{corollary for colour concentration}
below.
Then we (easily) prove Theorem~\ref{0-1 law in scaled case} restated as 
Corollary~\ref{cor:logical-0-1 law in scaled case}
below.

Just as we have defined the weight $\mu_n(\mcA)$ of $\mcA \in \mbW_n$
(and $\mu_n(\mbX)$ of $\mbX \subseteq \mbW_n$) we can
define the scaled weight $\widetilde{\mu}_n$ of $\mcA \in \mbW_n$, or of $\mbX \subseteq \mbW_n$
and then define $\widetilde{\mbbP}_n$ with the help of $\widetilde{\mu}_n$:

More precisely, we define, for $\mcA\in\mbW_n$, 
\[
\widetilde{\mu}_n(\mcA)
:=
\exp\Big(\frac1n\sum_{i=1}^8 w_i\,|\varphi_i(\mcA)|\Big).
\]
For $\mbX\subseteq\mbW_n$, set
\[
\widetilde{\mu}_n(\mbX):=\sum_{\mcA\in\mbX}\widetilde{\mu}_n(\mcA).
\]
For $0\le m\le n$, let
\[
\widetilde{Z}_n(m):=\widetilde{\mu}_n\big(\mbW_n(m)\big),
\qquad
\widetilde{Z}_n:=\widetilde{\mu}_n(\mbW_n)=\sum_{m=0}^n\widetilde{Z}_n(m),
\]
and define
\[
\widetilde{\mbbP}_n(\mbX):=\frac{\widetilde{\mu}_n(\mbX)}{\widetilde{Z}_n}
\qquad(\mbX\subseteq\mbW_n).
\]
If $\widetilde{\mu}_n(\mbY)>0$, then
\[
\widetilde{\mbbP}_n(\mbX \ \big| \ \mbY)
:=
\frac{\widetilde{\mu}_n(\mbX\cap\mbY)}{\widetilde{\mu}_n(\mbY)}.
\]
If $\varphi(\bar{x})$ is a $\sigma$-formula and $\bar{a}\in[n]^{|\bar{x}|}$, then
\[
\widetilde{\mbbP}_n(\varphi(\bar{a}))
:=
\widetilde{\mbbP}_n\big(\{\mcA\in\mbW_n:\mcA\models\varphi(\bar{a})\}\big).
\]
 In particular, if $\varphi$ is a $\sigma$-sentence, then
\[
\widetilde{\mbbP}_n(\varphi)
:=
\widetilde{\mbbP}_n\big(\{\mcA\in\mbW_n:\mcA\models\varphi\}\big).
\]

\medskip

\noindent
Analogously as in Notation~\ref{notation C-ij}, define
\[
\begin{array}{ll}
\widetilde C_{12}(n):=e^{w_1/n}+e^{w_2/n},
&\widetilde C_{34}(n):=e^{w_3/n}+e^{w_4/n},\\[2mm]
\widetilde C_{56}(n):=e^{w_5/n}+e^{w_6/n},
&\widetilde C_{78}(n):=e^{w_7/n}+e^{w_8/n}.
\end{array}
\]
Let
\[
\begin{aligned}
S_{12}&:=w_1+w_2,
&
S_{34}&:=w_3+w_4,\\
S_{56}&:=w_5+w_6,
&
S_{78}&:=w_7+w_8,
\end{aligned}
\]
and
\begin{equation}\label{delta}
S_\times:=S_{56}+S_{78}=w_5+w_6+w_7+w_8,
\qquad
\Delta:=S_{12}+S_{34}-S_\times.
\end{equation}
For $\alpha\in[0,1]$, define
\begin{equation}\label{eq:scaled-energy-function}
E(\alpha)
:=
\frac12\Big(
\alpha^2S_{12}
+(1-\alpha)^2S_{34}
+\alpha(1-\alpha)S_\times
\Big),
\end{equation}
\begin{equation}\label{eq:scaled-bounded-correction}
\begin{aligned}
Q(\alpha)
:= \frac18\Big(&
\alpha^2(w_1-w_2)^2 +(1-\alpha)^2(w_3-w_4)^2 \\
&+ \alpha(1-\alpha)\big((w_5-w_6)^2+(w_7-w_8)^2\big)
\Big),
\end{aligned}
\end{equation}
Recall from Definition~\ref{not:H} that the entropy function is
\[
H(\alpha)=
\begin{cases}
-\alpha\ln\alpha-(1-\alpha)\ln(1-\alpha), & \text{if }0<\alpha<1,\\[1mm]
0, & \text{if }\alpha=0\text{ or }\alpha=1.
\end{cases}
\]
This is the same entropy function as in the unscaled part of the paper. The reason is that
the entropy term comes only from choosing the coloured set. More precisely, if
$|P^\mcA|=m$, then there are exactly $\binom{n}{m}$ possible choices for $P^\mcA$,
both in the unscaled model and in the scaled model. Scaling the weights does not change the binomial factor
$\binom{n}{m}$. Hence the Stirling estimate
\[
\ln\binom{n}{m}=nH(m/n)+O(\ln n)
\]
uses exactly the same function $H$ as before.

We define the scaled phase function by
\begin{equation}\label{eq:scaled-phase-function}
\widetilde{\eff}(\alpha):=H(\alpha)+E(\alpha).
\end{equation}
The term $H(\alpha)$ is the colour-set entropy contribution, and the term $E(\alpha)$
is the leading contribution of the scaled ordered-pair weights. 

\begin{lem}\label{lem:scaled-two-term-log-expansion}
Let $u,v\in\mbbR$ be fixed. Then, as $n\to\infty$,
\begin{equation}\label{eq:scaled-log-two-term}
\ln\bigl(e^{u/n}+e^{v/n}\bigr)
=
\ln 2+\frac{u+v}{2n}+\frac{(u-v)^2}{8n^2}+O(n^{-4}).
\end{equation}
More generally, for every fixed $s>0$,
\begin{equation}\label{eq:scaled-log-two-term-s}
\ln\bigl(e^{u/n^s}+e^{v/n^s}\bigr)
=
\ln 2+\frac{u+v}{2n^s}+\frac{(u-v)^2}{8n^{2s}}+O(n^{-4s}).
\end{equation}

\end{lem}

\begin{proof}
It is enough to prove the more general estimate~\eqref{eq:scaled-log-two-term-s}. The special
case~\eqref{eq:scaled-log-two-term} is obtained by taking $s=1$.

Fix $s>0$. We have
\[
\frac{u}{n^s}
=
\frac{u+v}{2n^s}+\frac{u-v}{2n^s},
\qquad
\frac{v}{n^s}
=
\frac{u+v}{2n^s}-\frac{u-v}{2n^s},
\]
thus
\[
e^{\frac{u}{n^s}}+e^{\frac{v}{n^s}}=\exp (\frac{u+v}{2n^s}+\frac{u-v}{2n^s})+\exp (\frac{u+v}{2n^s}-\frac{u-v}{2n^s}).
\]
Therefore
\[
\begin{aligned}
e^{u/n^s}+e^{v/n^s}
&=
\exp\left(\frac{u+v}{2n^s}\right)
\left[
\exp\left(\frac{u-v}{2n^s}\right)
+
\exp\left(-\frac{u-v}{2n^s}\right)
\right].
\end{aligned}
\]
Using
\[
e^t+e^{-t}=2\cosh t,
\]
with
\[
t=\frac{u-v}{2n^s},
\]
we get 
\[
e^{u/n^s}+e^{v/n^s}
=
2\exp\left(\frac{u+v}{2n^s}\right)
\cosh\left(\frac{u-v}{2n^s}\right).
\]
Taking logarithms gives
\begin{equation}\label{eq:log-form}
\ln\bigl(e^{u/n^s}+e^{v/n^s}\bigr)
=
\ln 2+\frac{u+v}{2n^s}
+
\ln\cosh\left(\frac{u-v}{2n^s}\right).
\end{equation}

It remains to expand the last term. By the Taylor expansion of $\cosh t$ at $t=0$, we have
\[
\cosh t
=
1+\frac{t^2}{2}+O(t^4)
\qquad \text{as } t\to 0.
\]
Define
\[
u(t):= \cosh(t) - 1 \quad \text{ so } \quad u(t) = \frac{t^2}{2}+O(t^4).
\]
Then $u(t)\to 0$ as $t\to0$, and $u(t)^2=O(t^4)$. By the Taylor expansion of $\ln(1+u)$ at $u=0$, we have
\[
\ln(1+u)
=
u+O(u^2)
\qquad \text{as } u\to0.
\]
Applying this with $u=u(t)$, we obtain
\[
\ln\cosh t
=
\ln(1+u(t))
=
u(t)+O(u(t)^2)
=
\frac{t^2}{2}+O(t^4)
\qquad \text{as } t\to0.
\]

Now substitute
\[
t=\frac{u-v}{2n^s}.
\]
Then
\[
\frac{t^2}{2}
=
\frac12\left(\frac{u-v}{2n^s}\right)^2
=
\frac{(u-v)^2}{8n^{2s}},
\]
and, since $u$ and $v$ are fixed,
\[
O(t^4)
=
O\left(\left(\frac{u-v}{2n^s}\right)^4\right)
=
O(n^{-4s}).
\]
Therefore
\[
\ln\cosh\left(\frac{u-v}{2n^s}\right)
=
\frac{(u-v)^2}{8n^{2s}}
+
O(n^{-4s}).
\]
Substituting this back into ~\eqref{eq:log-form}, we obtain
\[
\ln\bigl(e^{u/n^s}+e^{v/n^s}\bigr)
=
\ln 2+\frac{u+v}{2n^s}
+\frac{(u-v)^2}{8n^{2s}}
+O(n^{-4s}).
\]
This proves~\eqref{eq:scaled-log-two-term-s}. 
\end{proof}

\medskip

\noindent
We now explain why the power $1/n$ is the natural scaling when the leading
ordered-pair contribution is nonconstant.

First recall the unscaled calculation from Section~\ref{sec:asymptotics}. In the unscaled model,
Equation~\eqref{ln of Z-n(m)} says that, for $0\le m\le n$,
\[
\ln Z_n(m)
=
\ln\binom{n}{m}
+
n^2\eff(m/n).
\]
 If we write $\alpha=m/n$, then the
ordered-pair part is
\[
n^2\eff(\alpha)
=
n^2\left[
\alpha^2\ln C_{12}
+
(1-\alpha)^2\ln C_{34}
+
\alpha(1-\alpha)\ln(C_{56}C_{78})
\right].
\]
This term has order $n^2$. Indeed, when $|P^\mcA|=m$, the ordered pairs split as follows:
\[
m^2=n^2\alpha^2,
\qquad
(n-m)^2=n^2(1-\alpha)^2,
\qquad
m(n-m)=n^2\alpha(1-\alpha).
\]
By contrast, the binomial factor contributes
\[
\ln\binom{n}{m}
=
nH(\alpha)+O(\ln n).
\]
This is a contribution to $\ln Z_n(m)$ of order $n$. Thus, in the
unscaled model, the ordered-pair weight contribution dominates the colour-set entropy whenever
$\eff$ is nonconstant.

Now consider a general scaling $w_i/n^s$, where $s>0$ is fixed. Temporarily write
$Z_n^{(s)}(m)$ for the analogue of $Z_n(m)$ under this scaling, and set
\[
C_{ij}^{(s)}(n):=e^{w_i/n^s}+e^{w_j/n^s}.
\]
The same counting as in the formula for $Z_n(m)$ gives
\[
Z_n^{(s)}(m)
=
\binom{n}{m}
C_{12}^{(s)}(n)^{m^2}
C_{34}^{(s)}(n)^{(n-m)^2}
C_{56}^{(s)}(n)^{m(n-m)}
C_{78}^{(s)}(n)^{m(n-m)}.
\]
Taking logarithms gives
\[
\begin{aligned}
\ln Z_n^{(s)}(m)
&=
\ln\binom{n}{m}
+
m^2\ln C_{12}^{(s)}(n)
+
(n-m)^2\ln C_{34}^{(s)}(n) \\
&\quad
+
m(n-m)\ln C_{56}^{(s)}(n)
+
m(n-m)\ln C_{78}^{(s)}(n).
\end{aligned}
\]

By Lemma~\ref{lem:scaled-two-term-log-expansion},
\[
\begin{aligned}
\ln C_{12}^{(s)}(n)
&=
\ln2+\frac{S_{12}}{2n^s}
+\frac{(w_1-w_2)^2}{8n^{2s}}
+O(n^{-4s}),\\
\ln C_{34}^{(s)}(n)
&=
\ln2+\frac{S_{34}}{2n^s}
+\frac{(w_3-w_4)^2}{8n^{2s}}
+O(n^{-4s}),\\
\ln C_{56}^{(s)}(n)
&=
\ln2+\frac{S_{56}}{2n^s}
+\frac{(w_5-w_6)^2}{8n^{2s}}
+O(n^{-4s}),\\
\ln C_{78}^{(s)}(n)
&=
\ln2+\frac{S_{78}}{2n^s}
+\frac{(w_7-w_8)^2}{8n^{2s}}
+O(n^{-4s}).
\end{aligned}
\]
Put again $\alpha=m/n$. The order $n^2$ terms are
\[
\begin{aligned}
& m^2\ln2+(n-m)^2\ln2+m(n-m)\ln2+m(n-m)\ln2 \\
&\qquad =
\bigl(m^2+(n-m)^2+2m(n-m)\bigr)\ln2
=
n^2\ln2.
\end{aligned}
\]
This term is independent of $m$. It is therefore present in $\ln Z_n^{(s)}(m)$, but it
does not affect comparisons between different colour counts.

The next terms are of order $n^{2-s}$. They are
\[
\begin{aligned}
& m^2\frac{S_{12}}{2n^s}
+
(n-m)^2\frac{S_{34}}{2n^s}
+
m(n-m)\frac{S_{56}}{2n^s}
+
m(n-m)\frac{S_{78}}{2n^s} \\
&\qquad =
n^{2-s}\frac12
\left[
\alpha^2S_{12}
+
(1-\alpha)^2S_{34}
+
\alpha(1-\alpha)(S_{56}+S_{78})
\right] \\
&\qquad =
n^{2-s}E(\alpha).
\end{aligned}
\]
The terms of order $n^{2-2s}$ are
\[
\begin{aligned}
& m^2\frac{(w_1-w_2)^2}{8n^{2s}}
+
(n-m)^2\frac{(w_3-w_4)^2}{8n^{2s}} \\
&\quad
+
m(n-m)\frac{(w_5-w_6)^2}{8n^{2s}}
+
m(n-m)\frac{(w_7-w_8)^2}{8n^{2s}} \\
&\qquad =
n^{2-2s}Q(\alpha).
\end{aligned}
\]
The remaining errors are bounded by a constant times
\[
n^2\cdot n^{-4s}=n^{2-4s}.
\]
 Therefore
\[
\ln Z_n^{(s)}(m)
=
n^2\ln2
+
\ln\binom{n}{m}
+
n^{2-s}E(\alpha)
+
n^{2-2s}Q(\alpha)
+
O(n^{2-4s}).
\]
Using the same binomial estimate as before,
\[
\ln\binom{n}{m}=nH(\alpha)+O(\ln n),
\]
we obtain
\[
\ln Z_n^{(s)}(m)
=
n^2\ln2
+
nH(\alpha)
+
n^{2-s}E(\alpha)
+
n^{2-2s}Q(\alpha)
+
O(\ln n)
+
O(n^{2-4s}).
\]

Now assume that $E$ is nonconstant. Then the first nonconstant contribution coming from
the ordered-pair weights has size $n^{2-s}$, while the colour-set entropy contribution
has size $n$. Hence:
\[
\begin{array}{|c|c|}
\hline 
s<1
&
n^{2-s}E(\alpha)\text{ is larger than the entropy term }nH(\alpha).
\\[1mm]
s=1
&
n^{2-s}E(\alpha)=nE(\alpha)\text{ has the same order as }nH(\alpha).
\\[1mm]
s>1
&
n^{2-s}E(\alpha)\text{ is smaller than the entropy term }nH(\alpha). \\
\hline
\end{array}
\]
Thus, when $E$ is nonconstant, $s=1$ is the unique positive scaling exponent for which
the leading ordered-pair weight contribution and the colour-set entropy contribution compete
on the same scale.

In particular, for the $1/n$-scaled model, we have
\[
\ln Z_n^{(1)}(m)
=
n^2\ln2
+
nH(\alpha)
+
nE(\alpha)
+
Q(\alpha)
+
O(\ln n).
\]
Equivalently,
\[
\ln Z_n^{(1)}(m)
=
n^2\ln2
+
n\bigl(H(\alpha)+E(\alpha)\bigr)
+
O(\ln n).
\]
This is why the scaled phase function is
\[
\widetilde{\eff}(\alpha)=H(\alpha)+E(\alpha).
\]

\begin{rem}\label{rem:why-1-over-n-is-the-right-power}
{\rm
In the argument above we assume that $E$ is not constant. This assumption is important. If
$E$ is constant, then the term
\[
n^{2-s}E(\alpha)
\]
has the same value for every colour proportion $\alpha$. Hence, although this term may be
large, it does not help compare different colour counts. In that case, the first term coming
from the ordered-pair weights which can distinguish different values of $\alpha$ may occur
later in the expansion, starting with the term
\[
n^{2-2s}Q(\alpha),
\]
where $Q$ is defined in~\eqref{eq:scaled-bounded-correction}.

To see when $E(\alpha)$ is constant, recall that
\[
E(\alpha)
=
\frac12\Big(
\alpha^2S_{12}
+
(1-\alpha)^2S_{34}
+
\alpha(1-\alpha)S_\times
\Big),
\]
where
\[
S_\times=S_{56}+S_{78}.
\]
We expand the two products involving $1-\alpha$:
\[
(1-\alpha)^2=1-2\alpha+\alpha^2,
\qquad
\alpha(1-\alpha)=\alpha-\alpha^2.
\]
Substituting these into the formula for $E(\alpha)$, we get
\[
\begin{aligned}
E(\alpha)
&=
\frac12\Big(
\alpha^2S_{12}
+
(1-2\alpha+\alpha^2)S_{34}
+
(\alpha-\alpha^2)S_\times
\Big)\\
&=
\frac12\Big(
S_{34}
+
(-2S_{34}+S_\times)\alpha
+
(S_{12}+S_{34}-S_\times)\alpha^2
\Big).
\end{aligned}
\]
Thus $E(\alpha)$ is constant when the
coefficients of $\alpha$ and $\alpha^2$ are both zero. Therefore $E$ is constant only
when
\[
-2S_{34}+S_\times=0
\qquad\text{and}\qquad
S_{12}+S_{34}-S_\times=0.
\]
Equivalently,
\[
S_\times=2S_{34}
\qquad\text{and}\qquad
S_{12}+S_{34}=S_\times,
\]
which means
\[
S_{12}+S_{34}=2S_{34},
\]
and hence
\[
S_{12}=S_{34}.
\]
So the constant case is when
\[
S_{12}=S_{34}=\frac{S_\times}{2}.
\]

}
\end{rem}

\begin{prop}\label{prop:scaled-partition-expansion}
For every $n\in\mbbN^+$ and every $m\in\{0,\ldots,n\}$,
\begin{equation}\label{eq:scaled-Znm-exact}
\widetilde{Z}_n(m)
=
\binom{n}{m}
\widetilde C_{12}(n)^{m^2}
\widetilde C_{34}(n)^{(n-m)^2}
\widetilde C_{56}(n)^{m(n-m)}
\widetilde C_{78}(n)^{m(n-m)}.
\end{equation}
If $\alpha:=m/n$, then, as $n\to\infty$, for $0\le m\le n$,
\begin{equation}\label{eq:scaled-logZ-full}
\ln \widetilde{Z}_n(m)
=
n^2\ln 2+\ln\binom{n}{m}+nE(\alpha)+Q(\alpha)+O(n^{-2}).
\end{equation}
Consequently, as $n\to\infty$, for $0\le m\le n$,
\begin{equation}\label{eq:scaled-logZ-leading}
\ln \widetilde{Z}_n(m)
=
n^2\ln 2+n\widetilde{\eff}(m/n)+O(\ln n).
\end{equation}

\end{prop}

\begin{proof}
Fix $n\in\mbbN^+$ and $m\in\{0,\ldots,n\}$. Choose $M\subseteq[n]$ with $|M|=m$, and write
$M^c:=[n]\setminus M$. We first compute the weight of the event that the coloured set is exactly
$M$.

If $a,b\in M$, then the ordered pair $(a,b)$ contributes either $e^{w_1/n}$, when $R(a,b)$ holds,
or $e^{w_2/n}$, when $R(a,b)$ does not hold. Summing over the two possible values of $R(a,b)$ gives
one factor $e^{w_1/n}+e^{w_2/n}=\widetilde C_{12}(n)$. There are $m^2$ such ordered pairs, so the
contribution of all pairs inside $M$ is $\widetilde C_{12}(n)^{m^2}$.

The same reasoning applies to the other colour-types. There are $(n-m)^2$ ordered pairs with
$a,b\in M^c$, and each gives the factor $\widetilde C_{34}(n)$. There are $m(n-m)$ ordered pairs
with $a\in M$ and $b\in M^c$, and each gives $\widetilde C_{56}(n)$. Finally, there are $m(n-m)$
ordered pairs with $a\in M^c$ and $b\in M$, and each gives $\widetilde C_{78}(n)$. Using independence of the summation over the different ordered pairs,
\begin{equation}\label{eq:scaled-fixed-colouring-factor}
\begin{aligned}
\widetilde{\mu}_n\big(\{\mcA\in\mbW_n:P^\mcA=M\}\big)
&=
\prod_{(a,b)\in M^2}\big(e^{w_1/n}+e^{w_2/n}\big)\\
&\quad\cdot
\prod_{(a,b)\in (M^c)^2}\big(e^{w_3/n}+e^{w_4/n}\big)\\
&\quad\cdot
\prod_{a\in M}\prod_{b\in M^c}\big(e^{w_5/n}+e^{w_6/n}\big)\\
&\quad\cdot
\prod_{a\in M^c}\prod_{b\in M}\big(e^{w_7/n}+e^{w_8/n}\big)\\
&=
\widetilde C_{12}(n)^{m^2}
\widetilde C_{34}(n)^{(n-m)^2}\\
&\quad\cdot
\widetilde C_{56}(n)^{m(n-m)}
\widetilde C_{78}(n)^{m(n-m)}.
\end{aligned}
\end{equation}
There are $\binom{n}{m}$ possible choices of $M$, and the value in
\eqref{eq:scaled-fixed-colouring-factor} depends only on $m$. This proves
\eqref{eq:scaled-Znm-exact}.

Applying Lemma~\ref{lem:scaled-two-term-log-expansion} to $(w_1,w_2)$, $(w_3,w_4)$, $(w_5,w_6)$,
and $(w_7,w_8)$ gives
\begin{equation}\label{eq:scaled-four-log-estimates}
\begin{aligned}
\ln\widetilde C_{12}(n)
&=
\ln 2+\frac{S_{12}}{2n}+\frac{(w_1-w_2)^2}{8n^2}+O(n^{-4}),\\
\ln\widetilde C_{34}(n)
&=
\ln 2+\frac{S_{34}}{2n}+\frac{(w_3-w_4)^2}{8n^2}+O(n^{-4}),\\
\ln\widetilde C_{56}(n)
&=
\ln 2+\frac{S_{56}}{2n}+\frac{(w_5-w_6)^2}{8n^2}+O(n^{-4}),\\
\ln\widetilde C_{78}(n)
&=
\ln 2+\frac{S_{78}}{2n}+\frac{(w_7-w_8)^2}{8n^2}+O(n^{-4}).
\end{aligned}
\end{equation}
Writing $\alpha=m/n$, we have
\[
m^2=n^2\alpha^2,
\qquad
(n-m)^2=n^2(1-\alpha)^2,
\qquad
m(n-m)=n^2\alpha(1-\alpha).
\]
Multiplying the estimates in~\eqref{eq:scaled-four-log-estimates} by the corresponding numbers of
ordered pairs gives, uniformly for $0\le m\le n$,
\begin{align*}
m^2\ln\widetilde C_{12}(n)
&= n^2\alpha^2\ln 2
   + n\frac{S_{12}}2\alpha^2 \\
&\qquad
   + \frac{(w_1-w_2)^2}{8}\alpha^2
   + O(n^{-2}),\\[0.3em]
(n-m)^2\ln\widetilde C_{34}(n)
&= n^2(1-\alpha)^2\ln 2
   + n\frac{S_{34}}2(1-\alpha)^2 \\
&\qquad
   + \frac{(w_3-w_4)^2}{8}(1-\alpha)^2
   + O(n^{-2}),\\[0.3em]
m(n-m)\ln\widetilde C_{56}(n)
&= n^2\alpha(1-\alpha)\ln 2
   + n\frac{S_{56}}2\alpha(1-\alpha) \\
&\qquad
   + \frac{(w_5-w_6)^2}{8}\alpha(1-\alpha)
   + O(n^{-2}),\\[0.3em]
m(n-m)\ln\widetilde C_{78}(n)
&= n^2\alpha(1-\alpha)\ln 2
   + n\frac{S_{78}}2\alpha(1-\alpha) \\
&\qquad
   + \frac{(w_7-w_8)^2}{8}\alpha(1-\alpha)
   + O(n^{-2}).
\end{align*}
 Adding the four estimates and using
\[
\alpha^2+(1-\alpha)^2+2\alpha(1-\alpha)=1
\]
gives
\[
\begin{aligned}
&m^2\ln\widetilde C_{12}(n)
+(n-m)^2\ln\widetilde C_{34}(n)
+m(n-m)\ln\widetilde C_{56}(n)\\
&\qquad
+m(n-m)\ln\widetilde C_{78}(n)
=
n^2\ln 2+nE(\alpha)+Q(\alpha)+O(n^{-2}),
\end{aligned}
\]
for $0\le m\le n$. Taking logarithms in~\eqref{eq:scaled-Znm-exact} proves
\eqref{eq:scaled-logZ-full}.

It remains to estimate the binomial term. For $1\le m\le n-1$, the Stirling calculation
in the proof of Lemma~\ref{lem:uniform} gives
\[
\ln\binom{n}{m}
=
nH(\alpha)-\frac12\ln\bigl(2\pi n\alpha(1-\alpha)\bigr)+O(1).
\]
Thus
\[
\ln\binom{n}{m}=nH(\alpha)+O(\ln n).
\]
 At the endpoints $m=0$ and $m=n$, both
$\ln\binom{n}{m}$ and $H(m/n)$ are equal to $0$. Therefore the same estimate holds for
all $0\le m\le n$.

The function $Q$ is continuous on $[0,1]$, and hence bounded there. Substituting the binomial
estimate into~\eqref{eq:scaled-logZ-full} and absorbing the bounded term $Q(\alpha)$ and the
$O(n^{-2})$ term into $O(\ln n)$ gives~\eqref{eq:scaled-logZ-leading}.
\end{proof}

Recall that \[ \widetilde{\eff}(\alpha) =H(\alpha)+E(\alpha)= H(\alpha)+\frac12\Big( \alpha^2S_{12}+(1-\alpha)^2S_{34}+\alpha(1-\alpha)S_\times \Big). \] Let \[ \widetilde{\Eff}^{\max} := \underset{\alpha\in[0,1]}{\arg\max}\,\widetilde{\eff}(\alpha). \] Since $\widetilde{\eff}$ is continuous on the compact interval $[0,1]$, it attains its maximum there. Hence $\widetilde{\Eff}^{\max}$ is nonempty. Moreover, \[ \widetilde{\Eff}^{\max} = \left\{ \gamma\in[0,1]: \widetilde{\eff}(\gamma) = \max_{\alpha\in[0,1]}\widetilde{\eff}(\alpha) \right\} = \widetilde{\eff}^{-1} \left( \left\{ \max_{\alpha\in[0,1]}\widetilde{\eff}(\alpha) \right\} \right). \] The singleton \[ \left\{ \max_{\alpha\in[0,1]}\widetilde{\eff}(\alpha) \right\} \] is closed in $\mathbb{R}$, and therefore its inverse image under the continuous function $\widetilde{\eff}$ is closed in $[0,1]$. Since $[0,1]$ is compact, it follows that $\widetilde{\Eff}^{\max}$ is compact.

\begin{lem}\label{lem:scaled-variational-properties}
There exists $\eta_0>0$ such that
\begin{equation}\label{eq:scaled-maximisers-away-from-boundary}
\widetilde{\Eff}^{\max}\subseteq[\eta_0,1-\eta_0].
\end{equation}
For $0<\alpha<1$, and $\Delta=S_{12}+S_{34}-S_\times$ ,
\begin{align}
\widetilde{\eff}'(\alpha)
&=
\ln\frac{1-\alpha}{\alpha}
+\Delta\alpha+\frac12S_\times-S_{34},
\label{eq:scaled-first-derivative}\\
\widetilde{\eff}''(\alpha)
&=
-\frac1{\alpha(1-\alpha)}+\Delta.
\label{eq:scaled-second-derivative}
\end{align}
In particular,
\[
\widetilde{\eff}'(\alpha)\to+\infty\quad\text{as }\alpha\to0,
\qquad
\widetilde{\eff}'(\alpha)\to-\infty\quad\text{as }\alpha\to1.
\]
If $\Delta\le4$, then $\widetilde{\eff}$ is strictly concave on $[0,1]$. Hence, for
$\Delta\le4$, the set $\widetilde{\Eff}^{\max}$ is a singleton.
\end{lem}

\begin{proof}
The function $\widetilde{\eff}$ is continuous on the compact interval $[0,1]$, so
$\widetilde{\Eff}^{\max}$ is nonempty and compact.
\\
We first show that neither endpoint is a maximiser.
\\
We first show that the endpoint $0$ is not a maximiser. Recall that
\[
H(\alpha)=-\alpha \ln \alpha-(1-\alpha)\ln(1-\alpha),
\]
in particular, $H(0)=0.$
\\
For $0<\alpha<1$, we have
\[
-\alpha\ln\alpha=\alpha\ln\left(\frac{1}{\alpha}\right).
\]
Moreover, using the Taylor expansion of $\ln(1-\alpha)$ around $\alpha=0$, we get
\[
\ln(1-\alpha)=-\alpha+O(\alpha^2)
\qquad \text{as } \alpha\to 0.
\]
Therefore
\[
-(1-\alpha)\ln(1-\alpha)
=
-(1-\alpha)\bigl(-\alpha+O(\alpha^2)\bigr)
=
\alpha+O(\alpha^2).
\]
Hence, as $\alpha\to0$,
\[
H(\alpha)
=
\alpha\ln\left(\frac{1}{\alpha}\right)+\alpha+O(\alpha^2)
=
\alpha\ln\left(\frac{1}{\alpha}\right)+O(\alpha).
\]

Next, $E$ is the polynomial
\[
E(\alpha)
=
\frac{1}{2}S_{34}
+
\frac{1}{2}(S_\times-2S_{34})\alpha
+
\frac{1}{2}(S_{12}+S_{34}-S_\times)\alpha^2.
\]
Thus
\[
E(0)=\frac{1}{2}S_{34}.
\]
Consequently,
\[
E(\alpha)-E(0)
=
\frac{1}{2}(S_\times-2S_{34})\alpha
+
\frac{1}{2}(S_{12}+S_{34}-S_\times)\alpha^2.
\]
Equivalently,
\[
E(\alpha)-E(0)
=
\alpha\left[
\frac{1}{2}(S_\times-2S_{34})
+
\frac{1}{2}(S_{12}+S_{34}-S_\times)\alpha
\right].
\]
The expression inside the square brackets is bounded as $\alpha\to0$. Hence
\[
E(\alpha)-E(0)=O(\alpha)
\qquad \text{as } \alpha\to0.
\]

Combining the estimates for $H$ and $E$, we obtain
\[
\widetilde{\eff}(\alpha)-\widetilde{\eff}(0)
=
H(\alpha)-H(0)+E(\alpha)-E(0)
=
\alpha\ln\left(\frac{1}{\alpha}\right)+O(\alpha).
\]

Since $\ln(1/\alpha)\to\infty$ as $\alpha\to0$, the right hand side is positive for all sufficiently small $\alpha>0$. Therefore
\[
\widetilde{\eff}(\alpha)>\widetilde{\eff}(0)
\]
for all sufficiently small $\alpha>0$. Hence $0$ is not a maximiser, which means,
$
0\notin\widetilde{\Eff}^{\max}.
$

Similarly, if $\alpha=1-\varepsilon$ and $\varepsilon\to0$, then
\[
H(1-\varepsilon)=\varepsilon\ln(1/\varepsilon)+O(\varepsilon),
\qquad
E(1-\varepsilon)-E(1)=O(\varepsilon),
\]
so
\[
\widetilde{\eff}(1-\varepsilon)-\widetilde{\eff}(1)
=
\varepsilon\ln(1/\varepsilon)+O(\varepsilon)>0
\]
for all sufficiently small $\varepsilon>0$. Thus $1\notin\widetilde{\Eff}^{\max}$.
Since $\widetilde{\Eff}^{\max}$ is compact and contained in $(0,1)$, its distance from
$\{0,1\}$ is positive. This proves~\eqref{eq:scaled-maximisers-away-from-boundary}.

For $0<\alpha<1$,
\[
H'(\alpha)=\ln\frac{1-\alpha}{\alpha},
\qquad
H''(\alpha)=-\frac1{\alpha(1-\alpha)}.
\]
Also,
\[
E(\alpha)
=
\frac12\Big(\Delta\alpha^2+(S_\times-2S_{34})\alpha+S_{34}\Big),
\]
and hence
\[
E'(\alpha)=\Delta\alpha+\frac12S_\times-S_{34},
\qquad
E''(\alpha)=\Delta.
\]
Adding the derivatives of $H$ and $E$ gives~\eqref{eq:scaled-first-derivative} and
\eqref{eq:scaled-second-derivative}. 

The number $4$ is important because
\[
\alpha(1-\alpha)\le\frac14
\qquad\text{for }0<\alpha<1,
\]
or equivalently
\[
\frac1{\alpha(1-\alpha)}\ge4,
\]
with equality only at $\alpha=1/2$. Thus the entropy curvature
$-1/(\alpha(1-\alpha))$ is always at most $-4$.

Assume first that $\Delta<4$. For $0<\alpha<1$,
\[
\widetilde{\eff}''(\alpha)
=
-\frac1{\alpha(1-\alpha)}+\Delta
\le -4+\Delta<0.
\]
Thus $\widetilde{\eff}'$ is strictly decreasing on $(0,1)$.

If $\Delta=4$, then
\[
\widetilde{\eff}''(\alpha)=4-\frac1{\alpha(1-\alpha)}\le0
\]
for all $\alpha\in(0,1)$, and equality holds only at $\alpha=1/2$. Hence, if
$0<\alpha<\beta<1$, then
\[
\widetilde{\eff}'(\beta)-\widetilde{\eff}'(\alpha)
=
\int_\alpha^\beta \widetilde{\eff}''(t)\,dt<0,
\]
because every nontrivial interval $[\alpha,\beta]$ contains points different from $1/2$, and on
those points the integrand is strictly negative. Thus $\widetilde{\eff}'$ is again strictly decreasing
on $(0,1)$.

In both cases $\Delta\le4$, the derivative $\widetilde{\eff}'$ is strictly decreasing on $(0,1)$.
Therefore $\widetilde{\eff}$ is strictly concave on $(0,1)$, and by continuity it is strictly concave
on $[0,1]$. A strictly concave function on a compact interval has at most one maximiser, and therefore
$\widetilde{\Eff}^{\max}$ is a singleton when $\Delta\le4$.
\end{proof}

\noindent
For the concentration estimate below, set
\[
\widetilde{\beta}:=\max_{\alpha\in[0,1]}\widetilde{\eff}(\alpha),
\qquad
\widetilde{\Eff}^{\max}
=
\{\rho\in[0,1]:\widetilde{\eff}(\rho)=\widetilde{\beta}\}.
\]
For $\delta>0$, let
\[
\widetilde B_\delta
:=
\left\{\gamma\in[0,1]:
\inf_{\rho\in\widetilde{\Eff}^{\max}}|\gamma-\rho|\ge\delta
\right\}.
\]
Thus $\widetilde B_\delta$ consists of the colour proportions lying at least $\delta$
away from the maximising set. We denote the corresponding event by
\[
\widetilde A_{n,\delta}
:=
\left\{\mcA\in\mbW_n:
\frac{|P^\mcA|}{n}\in\widetilde B_\delta
\right\}.
\]

\begin{prop}
\label{prop:scaled-colour-concentration}

Fix $\delta>0$. If $\widetilde B_\delta=\varnothing$, then
$\widetilde{\mbbP}_n(\widetilde A_{n,\delta})=0$
for every $n$. If $\widetilde B_\delta\neq\varnothing$, set
\[
\widetilde c_\delta
:=
\widetilde\beta
-
\max_{\gamma\in\widetilde B_\delta}
\widetilde{\eff}(\gamma).
\]
Then $\widetilde c_\delta>0$, and
\[
\widetilde{\mbbP}_n(\widetilde A_{n,\delta})
\le
\exp\bigl(-\widetilde c_\delta n+O(\ln n)\bigr)
\qquad
(n\to\infty).
\]

Consequently, for every open neighbourhood
$U$ of $\widetilde{\Eff}^{\max}$ in $[0,1]$,
\[
\lim_{n\to\infty}
\widetilde{\mbbP}_n
\left(
\left\{
\mcA\in\mbW_n:
\frac{|P^\mcA|}{n}\in U
\right\}
\right)
=
1.
\]
\end{prop}

\begin{proof}
Define the distance from a colour proportion $\gamma\in[0,1]$ to the whole set of maximisers by
\[
d(\gamma)
:=
\inf_{\rho\in\widetilde{\Eff}^{\max}}|\gamma-\rho|.
\]
Then
\[
\widetilde B_\delta
=
\{\gamma\in[0,1]:d(\gamma)\ge\delta\},
\]
and
\[
\widetilde A_{n,\delta}
=
\Big\{\mcA\in\mbW_n:
d\Big(\frac{|P^\mcA|}{n}\Big)\ge\delta
\Big\}.
\]

Suppose first that $\widetilde B_\delta=\varnothing$. For every $\mcA\in\mbW_n$, the number
$|P^\mcA|/n$ belongs to $[0,1]$. Since no point of $[0,1]$ belongs to $\widetilde B_\delta$, we have
\[
d\Big(\frac{|P^\mcA|}{n}\Big)<\delta.
\]
Therefore $\mcA\notin\widetilde A_{n,\delta}$. Hence $\widetilde A_{n,\delta}=\varnothing$, and so
\[
\widetilde{\mbbP}_n(\widetilde A_{n,\delta})=0
\]
for every $n$.

Now suppose that $\widetilde B_\delta\neq\varnothing$. 

The function $d$ is continuous. Indeed, for any
$\gamma,\gamma'\in[0,1]$ and any
$\alpha\in\widetilde{\Eff}^{\max}$, the triangle inequality gives
\[
|\gamma-\alpha|
\le
|\gamma-\gamma'|+|\gamma'-\alpha|.
\]
Taking the infimum over
$\alpha\in\widetilde{\Eff}^{\max}$, we obtain
\[
d(\gamma)
=
\inf_{\alpha\in\widetilde{\Eff}^{\max}}|\gamma-\alpha|
\le
|\gamma-\gamma'|
+
\inf_{\alpha\in\widetilde{\Eff}^{\max}}|\gamma'-\alpha|
=
|\gamma-\gamma'|+d(\gamma').
\]
Similarly, we also get
\[
d(\gamma')
\le
|\gamma-\gamma'|+d(\gamma).
\]
Therefore
\[
|d(\gamma)-d(\gamma')|
\le
|\gamma-\gamma'|.
\]
Thus $d$ is Lipschitz continuous, and in particular continuous, on
$[0,1]$.

Since
\[
\widetilde B_\delta
=
d^{-1}([\delta,\infty)),
\]
the set $\widetilde B_\delta$ is closed in the compact interval $[0,1]$, therefore
$\widetilde B_\delta$ is compact. Since we are now assuming
$\widetilde B_\delta\neq\varnothing$, and since $\widetilde{\eff}$ is continuous, there exists
$\gamma_\delta\in\widetilde B_\delta$ such that
\[
\widetilde{\eff}(\gamma_\delta)
=
\max_{\gamma\in\widetilde B_\delta}\widetilde{\eff}(\gamma).
\]

The set $\widetilde B_\delta$ contains no maximising colour proportion. Indeed, if
$\rho\in\widetilde{\Eff}^{\max}$, then
\[
d(\rho)=0,
\]
whereas every $\gamma\in\widetilde B_\delta$ satisfies
\[
d(\gamma)\ge\delta>0.
\]
Hence
\[
\widetilde B_\delta\cap\widetilde{\Eff}^{\max}=\varnothing.
\]

If
\[
\max_{\gamma\in\widetilde B_\delta}\widetilde{\eff}(\gamma)
=
\widetilde{\beta},
\]
then
\[
\widetilde{\eff}(\gamma_\delta)=\widetilde{\beta}.
\]
By the definition of $\widetilde{\Eff}^{\max}$, this would imply
\[
\gamma_\delta\in\widetilde{\Eff}^{\max}.
\]
But $\gamma_\delta\in\widetilde B_\delta$, contradicting
\[
\widetilde B_\delta\cap\widetilde{\Eff}^{\max}=\varnothing.
\]
Therefore
\[
\max_{\gamma\in\widetilde B_\delta}\widetilde{\eff}(\gamma)
<
\widetilde{\beta}.
\]
Hence
\[
\widetilde c_\delta
=
\widetilde{\beta}
-
\max_{\gamma\in\widetilde B_\delta}\widetilde{\eff}(\gamma)
>
0.
\]

We now prove the probability estimate. The estimate
\eqref{eq:scaled-logZ-leading} is uniform in $m\in\{0,\ldots,n\}$. Hence there are constants
$C>0$ and $N_1\in\mbbN^+$ such that, for every $n\ge N_1$ and every
$m\in\{0,\ldots,n\}$,
\begin{equation}\label{eq:scaled-logZ-two-sided}
\left|
\ln\widetilde Z_n(m)
-
\Big(n^2\ln2+n\widetilde{\eff}(m/n)\Big)
\right|
\le
C\ln(n+1).
\end{equation}
Equivalently, for all such $n$ and $m$,
\[
\ln\widetilde Z_n(m)
\le
n^2\ln2+n\widetilde{\eff}(m/n)+C\ln(n+1),
\]
and
\[
\ln\widetilde Z_n(m)
\ge
n^2\ln2+n\widetilde{\eff}(m/n)-C\ln(n+1).
\]

Let
\[
\alpha^\star\in\widetilde{\Eff}^{\max}.
\]

By Lemma~\ref{lem:scaled-variational-properties}, there is $\eta_0>0$ such that
\[
\widetilde{\Eff}^{\max}\subseteq[\eta_0,1-\eta_0].
\]
In particular,
\[
\alpha^\star\in[\eta_0,1-\eta_0].
\]
For every $n$, choose an integer $\widehat m_n\in\{0,\ldots,n\}$ such that
\[
|\widehat m_n-\alpha^\star n|\le \frac12.
\]
Then
\[
\left|\frac{\widehat m_n}{n}-\alpha^\star\right|
\le
\frac1{2n}.
\]

Set
\[
L
:=
\sup_{\alpha\in[\eta_0/2,\,1-\eta_0/2]}
|\widetilde{\eff}'(\alpha)|.
\]
This number is finite because $\widetilde{\eff}'$ is continuous on the compact interval
$[\eta_0/2,1-\eta_0/2]$.

For all sufficiently large $n$, both $\alpha^\star$ and $\widehat m_n/n$ lie in
$[\eta_0/2,1-\eta_0/2]$. Indeed, $\alpha^\star\in[\eta_0,1-\eta_0]$, and
$\widehat m_n/n$ is within $1/(2n)$ of $\alpha^\star$.

By the mean value theorem, there exists a point $\xi_n$ between $\alpha^\star$ and
$\widehat m_n/n$ such that
\[
\widetilde{\eff}\Big(\frac{\widehat m_n}{n}\Big)
-
\widetilde{\eff}(\alpha^\star)
=
\widetilde{\eff}'(\xi_n)
\left(
\frac{\widehat m_n}{n}-\alpha^\star
\right).
\]
Taking absolute values gives
\[
\left|
\widetilde{\eff}\Big(\frac{\widehat m_n}{n}\Big)
-
\widetilde{\eff}(\alpha^\star)
\right|
=
|\widetilde{\eff}'(\xi_n)|
\left|
\frac{\widehat m_n}{n}-\alpha^\star
\right|.
\]
Since $\xi_n\in[\eta_0/2,1-\eta_0/2]$, we have
\[
|\widetilde{\eff}'(\xi_n)|\le L.
\]
Therefore
\[
\left|
\widetilde{\eff}\Big(\frac{\widehat m_n}{n}\Big)
-
\widetilde{\eff}(\alpha^\star)
\right|
\le
L
\left|
\frac{\widehat m_n}{n}-\alpha^\star
\right|
\le
\frac{L}{2n}.
\]
Since $\alpha^\star$ is a maximiser,
\[
\widetilde{\eff}(\alpha^\star)=\widetilde{\beta}.
\]
Hence
\begin{equation}\label{eq:scaled-rounded-maximiser}
\widetilde{\eff}\Big(\frac{\widehat m_n}{n}\Big)
\ge
\widetilde{\beta}-\frac{L}{2n}.
\end{equation}

Now let $m\in\{0,\ldots,n\}$ be such that
\[
\frac{m}{n}\in\widetilde B_\delta.
\]
By the definition of $\widetilde c_\delta$,
\[
\max_{\gamma\in\widetilde B_\delta}\widetilde{\eff}(\gamma)
=
\widetilde{\beta}-\widetilde c_\delta.
\]
Since $m/n\in\widetilde B_\delta$, it follows that
\[
\widetilde{\eff}\Big(\frac{m}{n}\Big)
\le
\widetilde{\beta}-\widetilde c_\delta.
\]
Combining this with \eqref{eq:scaled-rounded-maximiser}, we get
\[
\begin{aligned}
\widetilde{\eff}\Big(\frac{m}{n}\Big)
-
\widetilde{\eff}\Big(\frac{\widehat m_n}{n}\Big)
&\le
(\widetilde{\beta}-\widetilde c_\delta)
-
\left(\widetilde{\beta}-\frac{L}{2n}\right)\\
&=
-\widetilde c_\delta+\frac{L}{2n}.
\end{aligned}
\]
Multiplying by $n$ gives
\begin{equation}\label{eq:scaled-phase-gap-after-rounding}
n\left[
\widetilde{\eff}\Big(\frac{m}{n}\Big)
-
\widetilde{\eff}\Big(\frac{\widehat m_n}{n}\Big)
\right]
\le
-\widetilde c_\delta n+\frac{L}{2}.
\end{equation}

Since
\[
\widetilde Z_n
=
\sum_{r=0}^n\widetilde Z_n(r)
\ge
\widetilde Z_n(\widehat m_n),
\]
we have
\[
\widetilde{\mbbP}_n(\mbW_n(m))
=
\frac{\widetilde Z_n(m)}{\widetilde Z_n}
\le
\frac{\widetilde Z_n(m)}{\widetilde Z_n(\widehat m_n)}.
\]
Taking logarithms and using \eqref{eq:scaled-logZ-two-sided}, first with the upper bound for $m$
and then with the lower bound for $\widehat m_n$, gives
\[
\begin{aligned}
\ln
\frac{\widetilde Z_n(m)}{\widetilde Z_n(\widehat m_n)}
&=
\ln\widetilde Z_n(m)
-
\ln\widetilde Z_n(\widehat m_n)\\
&\le
\Big[
n^2\ln2
+
n\widetilde{\eff}\Big(\frac{m}{n}\Big)
+
C\ln(n+1)
\Big]\\
&\quad -
\Big[
n^2\ln2
+
n\widetilde{\eff}\Big(\frac{\widehat m_n}{n}\Big)
-
C\ln(n+1)
\Big]\\
&=
n\left[
\widetilde{\eff}\Big(\frac{m}{n}\Big)
-
\widetilde{\eff}\Big(\frac{\widehat m_n}{n}\Big)
\right]
+
2C\ln(n+1).
\end{aligned}
\]
Using \eqref{eq:scaled-phase-gap-after-rounding}, we obtain
\[
\ln
\frac{\widetilde Z_n(m)}{\widetilde Z_n(\widehat m_n)}
\le
-\widetilde c_\delta n+\frac{L}{2}+2C\ln(n+1).
\]
Therefore, for every $m$ with $m/n\in\widetilde B_\delta$,
\[
\widetilde{\mbbP}_n(\mbW_n(m))
\le
\exp\left(
-\widetilde c_\delta n+\frac{L}{2}+2C\ln(n+1)
\right).
\]
Since $L/2$ is constant and $2C\ln(n+1)=O(\ln n)$, this can be written as
\[
\widetilde{\mbbP}_n(\mbW_n(m))
\le
\exp\big(-\widetilde c_\delta n+O(\ln n)\big),
\]
for all $m$ satisfying $m/n\in\widetilde B_\delta$.

Finally,
\[
\widetilde A_{n,\delta}
=
\bigcup_{\substack{0\le m\le n\\ m/n\in\widetilde B_\delta}}
\mbW_n(m).
\]
The union is disjoint because two different values of $m$ give two different colour counts.
There are at most $n+1$ possible values of $m$. Hence
\[
\begin{aligned}
\widetilde{\mbbP}_n(\widetilde A_{n,\delta})
&=
\sum_{\substack{0\le m\le n\\ m/n\in\widetilde B_\delta}}
\widetilde{\mbbP}_n(\mbW_n(m))\\
&\le
(n+1)\exp\big(-\widetilde c_\delta n+O(\ln n)\big).
\end{aligned}
\]
This proves the concentration estimate.

Now let $U$ be a neighbourhood of
$\widetilde{\Eff}^{\max}$ in $[0,1]$. Since $\widetilde{\Eff}^{\max}$ is compact and
$d$ is continuous, there exists $\delta>0$ such that
\[
\{\gamma\in[0,1]:d(\gamma)<\delta\}
\subseteq
U.
\]
Therefore, if $\gamma\notin U$, then
\[
d(\gamma)\ge\delta.
\]
Equivalently,
\[
[0,1]\setminus U
\subseteq
\widetilde B_\delta.
\]
Thus
\[
\begin{aligned}
\widetilde{\mbbP}_n\Big(
\Big\{\mcA\in\mbW_n:\frac{|P^\mcA|}{n}\notin U\Big\}
\Big)
\le
\widetilde{\mbbP}_n(\widetilde A_{n,\delta}).
\end{aligned}
\]
If $\widetilde B_\delta=\varnothing$, then the right-hand side is always $0$. If
$\widetilde B_\delta\neq\varnothing$, then the estimate already proved gives
\[
\widetilde{\mbbP}_n(\widetilde A_{n,\delta})
\le
(n+1)\exp\big(-\widetilde c_\delta n+O(\ln n)\big)
\longrightarrow
0,
\]
because $\widetilde c_\delta>0$. Hence
\[
\widetilde{\mbbP}_n\Big(
\Big\{\mcA\in\mbW_n:\frac{|P^\mcA|}{n}\notin U\Big\}
\Big)
\longrightarrow
0.
\]
Taking complements gives
\[
\widetilde{\mbbP}_n\Big(
\Big\{\mcA\in\mbW_n:\frac{|P^\mcA|}{n}\in U\Big\}
\Big)
\longrightarrow
1.
\]
This proves the proposition.
\end{proof}

\medskip

\noindent
From~(\ref{eq:scaled-maximisers-away-from-boundary}) and 
Proposition~\ref{prop:scaled-colour-concentration} we now immediately get:

\begin{cor}\label{corollary for colour concentration}
For every choice of nonnegative real weights $w_1, \ldots, w_8$, there is $\varepsilon > 0$ such that
\[
\lim_{n\to\infty} \widetilde{\mbbP}_n \Big(
  \big\{\mcA \in \mb W_n :
    \varepsilon n
      \le \big|P^\mcA\big|
      \le (1 - \varepsilon)n
  \big\}
\Big) = 1.
\]
\end{cor}

\begin{notation}\label{not:scaled-typed}{\rm
For $n\in\mbbN^+$ and $B\subseteq[n]$, let
\[
\widetilde{\mbY}_n^B
:=
\bigl\{\mcA\in\mbW_n:P^\mcA=B\bigr\}.
\]
Thus $\widetilde{\mbY}_n^B$ is the event that the set of coloured vertices is $B$.
For the corresponding one-edge probabilities in the scaled model, set
\begin{align*}
\widetilde p_{11}^{(n)}
&:=
\frac{e^{w_1/n}}{e^{w_1/n}+e^{w_2/n}},
&
\widetilde p_{00}^{(n)}
&:=
\frac{e^{w_3/n}}{e^{w_3/n}+e^{w_4/n}},\\[1mm]
\widetilde p_{10}^{(n)}
&:=
\frac{e^{w_5/n}}{e^{w_5/n}+e^{w_6/n}},
&
\widetilde p_{01}^{(n)}
&:=
\frac{e^{w_7/n}}{e^{w_7/n}+e^{w_8/n}}.
\end{align*}
In the subscript, $1$ denotes membership in $B$ and $0$ denotes non-membership in $B$,
for the first and second coordinates respectively.
}\end{notation}

\begin{lem}\label{lem:scaled-edge-factor}

Let $n\in\mbbN^+$ and let $B\subseteq[n]$. Conditioned on $\widetilde{\mbY}_n^B$, the edge variables
\[
\big(R_{(a,b)}:(a,b)\in[n]^2\big)
\]
are mutually independent. Moreover, for every $(a,b)\in[n]^2$,
\[
\widetilde{\mbbP}_n\big(R_{(a,b)}=1 \ \big| \ \widetilde{\mbY}_n^B\big)
=
\begin{cases}
\widetilde p_{11}^{(n)} & \text{if } a,b\in B,\\
\widetilde p_{00}^{(n)} & \text{if } a,b\notin B,\\
\widetilde p_{10}^{(n)} & \text{if } a\in B,\ b\notin B,\\
\widetilde p_{01}^{(n)} & \text{if } a\notin B,\ b\in B.
\end{cases}
\]
Furthermore,
\begin{align*}
\widetilde p_{11}^{(n)}
&=
\frac12+\frac{w_1-w_2}{4n}+O(n^{-3}),\\
\widetilde p_{00}^{(n)}
&=
\frac12+\frac{w_3-w_4}{4n}+O(n^{-3}),\\
\widetilde p_{10}^{(n)}
&=
\frac12+\frac{w_5-w_6}{4n}+O(n^{-3}),\\
\widetilde p_{01}^{(n)}
&=
\frac12+\frac{w_7-w_8}{4n}+O(n^{-3}).
\end{align*}
In particular, there exists $N\in\mbbN^+$ such that, for every $n\ge N$, each of the eight numbers
\[
\widetilde p_{11}^{(n)},\ 1-\widetilde p_{11}^{(n)},\
\widetilde p_{00}^{(n)},\ 1-\widetilde p_{00}^{(n)},\
\widetilde p_{10}^{(n)},\ 1-\widetilde p_{10}^{(n)},\
\widetilde p_{01}^{(n)},\ 1-\widetilde p_{01}^{(n)}
\]
is between $1/4$ and $3/4$.
\end{lem}

\begin{proof}
Fix $n\in\mbbN^+$ and $B\subseteq[n]$. After conditioning on
$\widetilde{\mbY}_n^B$, the set of vertices satisfying $P$ is fixed and equal to $B$. Thus the only
remaining choices are the truth values of the binary relation $R(a,b)$ for the ordered pairs
$(a,b)\in[n]^2$.

For each ordered pair $(a,b)$ and each $r\in\{0,1\}$, define a factor
$\lambda_{a,b}^{(n)}(r)$ as follows. If $a,b\in B$, then
\[
\lambda_{a,b}^{(n)}(1)=e^{w_1/n},
\qquad
\lambda_{a,b}^{(n)}(0)=e^{w_2/n}.
\]
If $a,b\notin B$, then
\[
\lambda_{a,b}^{(n)}(1)=e^{w_3/n},
\qquad
\lambda_{a,b}^{(n)}(0)=e^{w_4/n}.
\]
If $a\in B$ and $b\notin B$, then
\[
\lambda_{a,b}^{(n)}(1)=e^{w_5/n},
\qquad
\lambda_{a,b}^{(n)}(0)=e^{w_6/n}.
\]
If $a\notin B$ and $b\in B$, then
\[
\lambda_{a,b}^{(n)}(1)=e^{w_7/n},
\qquad
\lambda_{a,b}^{(n)}(0)=e^{w_8/n}.
\]
This records which one of the formulas
$\varphi_1(a,b),\ldots,\varphi_8(a,b)$ is satisfied by the ordered pair $(a,b)$. 

Let $\mcA\in\widetilde{\mbY}_n^B$. Since the colouring is fixed, each ordered pair contributes
exactly one factor, according to whether $R(a,b)$ is true or false in $\mcA$. Therefore
\[
\widetilde{\mu}_n(\mcA)
=
\exp\Big(\frac1n\sum_{i=1}^8 w_i\,|\varphi_i(\mcA)|\Big)
=
\prod_{(a,b)\in[n]^2}
\lambda_{a,b}^{(n)}\big(R_{(a,b)}(\mcA)\big),
\]
where $R_{(a,b)}(\mcA)$ equals $1$ if $\mcA\models R(a,b)$ and equals $0$
otherwise.

Now sum this product over all possible choices of the edge values. A choice of all edge values is a
function
\[
\xi:[n]^2\to\{0,1\}.
\]
Hence
\[
\begin{aligned}
\widetilde{\mu}_n(\widetilde{\mbY}_n^B)
&=
\sum_{\xi:[n]^2\to\{0,1\}}
\prod_{(a,b)\in[n]^2}
\lambda_{a,b}^{(n)}\big(\xi(a,b)\big)\\
&=
\prod_{(a,b)\in[n]^2}
\Big(
\lambda_{a,b}^{(n)}(0)+\lambda_{a,b}^{(n)}(1)
\Big).
\end{aligned}
\]

Let $F\subseteq[n]^2$ be finite, and for each $(a,b)\in F$ fix a value
$r_{a,b}\in\{0,1\}$. We compute the conditional probability that all these prescribed edge values
occur. The weight of the event
\[
\Big\{\mcA\in\widetilde{\mbY}_n^B:
R_{(a,b)}(\mcA)=r_{a,b}\text{ for all }(a,b)\in F
\Big\}
\]
is obtained by fixing the factors for the pairs in $F$ and summing over the remaining pairs:
\[
\begin{aligned}
&\widetilde{\mu}_n\Big(
\Big\{\mcA\in\widetilde{\mbY}_n^B:
R_{(a,b)}(\mcA)=r_{a,b}\text{ for all }(a,b)\in F
\Big\}
\Big)\\
&\qquad =
\prod_{(a,b)\in F}
\lambda_{a,b}^{(n)}(r_{a,b})
\prod_{(a,b)\in[n]^2\setminus F}
\Big(
\lambda_{a,b}^{(n)}(0)+\lambda_{a,b}^{(n)}(1)
\Big).
\end{aligned}
\]
Dividing this expression by
\[
\widetilde{\mu}_n(\widetilde{\mbY}_n^B)
=
\prod_{(a,b)\in[n]^2}
\Big(
\lambda_{a,b}^{(n)}(0)+\lambda_{a,b}^{(n)}(1)
\Big)
\]
cancels all factors coming from $[n]^2\setminus F$. Therefore
\[
\begin{aligned}
&\widetilde{\mbbP}_n \left(
\bigwedge_{(a,b)\in F}\{R_{(a,b)}=r_{a,b}\}
\ \middle|\ \widetilde{\mbY}_n^B
\right)\\
&\qquad =
\prod_{(a,b)\in F}
\widetilde{\mbbP}_n \left(
R_{(a,b)}=r_{a,b}
\ \middle|\ \widetilde{\mbY}_n^B
\right)\\
&\qquad =
\prod_{(a,b)\in F}
\frac{
\lambda_{a,b}^{(n)}(r_{a,b})
}{
\lambda_{a,b}^{(n)}(0)+\lambda_{a,b}^{(n)}(1)
}.
\end{aligned}
\]
The right-hand side is a product of one-edge conditional probabilities. Hence, conditioned on
$\widetilde{\mbY}_n^B$, the variables
\[
\big(R_{(a,b)}:(a,b)\in[n]^2\big)
\]
are mutually independent.

Taking $F=\{(a,b)\}$ gives the one-edge probabilities. If $a,b\in B$, then
\[
\widetilde{\mbbP}_n\big(R_{(a,b)}=1 \ \big| \ \widetilde{\mbY}_n^B\big)
=
\frac{e^{w_1/n}}{e^{w_1/n}+e^{w_2/n}}
=
\widetilde p_{11}^{(n)}.
\]
If $a,b\notin B$, then
\[
\widetilde{\mbbP}_n\big(R_{(a,b)}=1 \ \big| \ \widetilde{\mbY}_n^B\big)
=
\frac{e^{w_3/n}}{e^{w_3/n}+e^{w_4/n}}
=
\widetilde p_{00}^{(n)}.
\]
If $a\in B$ and $b\notin B$, then
\[
\widetilde{\mbbP}_n\big(R_{(a,b)}=1 \ \big| \ \widetilde{\mbY}_n^B\big)
=
\frac{e^{w_5/n}}{e^{w_5/n}+e^{w_6/n}}
=
\widetilde p_{10}^{(n)}.
\]
If $a\notin B$ and $b\in B$, then
\[
\widetilde{\mbbP}_n\big(R_{(a,b)}=1 \ \big| \ \widetilde{\mbY}_n^B\big)
=
\frac{e^{w_7/n}}{e^{w_7/n}+e^{w_8/n}}
=
\widetilde p_{01}^{(n)}.
\]

Now define
\[
g(t)
:=
\frac{1}{1+e^{-t}}.
\]
Then
\[
g(0)=\frac12.
\]
Differentiating gives
\[
g'(t)
=
\frac{e^{-t}}{(1+e^{-t})^2}.
\]
Hence
\[
g'(0)
=
\frac{1}{(1+1)^2}
=
\frac14.
\]
Differentiating once more gives
\[
g''(t)
=
\frac{e^{-t}(e^{-t}-1)}{(1+e^{-t})^3}.
\]
Therefore
\[
g''(0)
=
\frac{1\cdot(1-1)}{(1+1)^3}
=
0.
\]
The Taylor expansion of $g$ at $0$ is therefore
\[
\begin{aligned}
g(t)
&=
g(0)+g'(0)t+\frac12g''(0)t^2+O(t^3)\\
&=
\frac12+\frac{t}{4}+O(t^3)
\qquad(t\to0).
\end{aligned}
\]

For $\widetilde p_{11}^{(n)}$, we have:
\[
\begin{aligned}
\widetilde p_{11}^{(n)}
&=
\frac{e^{w_1/n}}{e^{w_1/n}+e^{w_2/n}}\\
&=
\frac{1}{1+e^{w_2/n}e^{-w_1/n}}\\
&=
\frac{1}{1+e^{(w_2-w_1)/n}}\\
&=
\frac{1}{1+e^{-(w_1-w_2)/n}}\\
&=
g\Big(\frac{w_1-w_2}{n}\Big).
\end{aligned}
\]
Using the Taylor expansion with
\[
t=\frac{w_1-w_2}{n}
\]
gives
\[
\widetilde p_{11}^{(n)}
=
\frac12
+
\frac{w_1-w_2}{4n}
+
O(n^{-3}).
\]

The same calculation, written out for the other three colour-types, gives
\[
\begin{aligned}
\widetilde p_{00}^{(n)}
&=
\frac{e^{w_3/n}}{e^{w_3/n}+e^{w_4/n}}\\
&=
\frac{1}{1+e^{(w_4-w_3)/n}}\\
&=
\frac{1}{1+e^{-(w_3-w_4)/n}}\\
&=
g\Big(\frac{w_3-w_4}{n}\Big)\\
&=
\frac12+\frac{w_3-w_4}{4n}+O(n^{-3}),
\end{aligned}
\]
\[
\begin{aligned}
\widetilde p_{10}^{(n)}
&=
\frac{e^{w_5/n}}{e^{w_5/n}+e^{w_6/n}}\\
&=
\frac{1}{1+e^{(w_6-w_5)/n}}\\
&=
\frac{1}{1+e^{-(w_5-w_6)/n}}\\
&=
g\Big(\frac{w_5-w_6}{n}\Big)\\
&=
\frac12+\frac{w_5-w_6}{4n}+O(n^{-3}),
\end{aligned}
\]
and
\[
\begin{aligned}
\widetilde p_{01}^{(n)}
&=
\frac{e^{w_7/n}}{e^{w_7/n}+e^{w_8/n}} \\
&=
\frac{1}{1+e^{(w_8-w_7)/n}} \\
&=
\frac{1}{1+e^{-(w_7-w_8)/n}}\\
&=
g\Big(\frac{w_7-w_8}{n}\Big) \\
&=
\frac12+\frac{w_7-w_8}{4n}+O(n^{-3}).
\end{aligned}
\]
The error term is $O(n^{-3})$ because the weights $w_1,\ldots,w_8$ are fixed, so each substituted
value of $t$ is of order $1/n$.

In particular,
\[
\widetilde p_{11}^{(n)},\quad
\widetilde p_{00}^{(n)},\quad
\widetilde p_{10}^{(n)},\quad
\widetilde p_{01}^{(n)}
\longrightarrow
\frac12 \quad \text{as}\quad n\to \infty
\]
Their complements also converge to $1/2$ as $n \to \infty$:
\[
1-\widetilde p_{11}^{(n)},\quad
1-\widetilde p_{00}^{(n)},\quad
1-\widetilde p_{10}^{(n)},\quad
1-\widetilde p_{01}^{(n)}
\longrightarrow
\frac12.
\]

Thus each of the eight listed quantities converges to $1/2$.

Let these eight quantities be denoted by $q_1^{(n)},\ldots,q_8^{(n)}$. For each
$j\in\{1,\ldots,8\}$, convergence gives an integer $N_j$ such that, whenever
$n\ge N_j$,
\[
\left|q_j^{(n)}-\frac12\right|<\frac14.
\]
Set
\[
N:=\max\{N_1,\ldots,N_8\}.
\]
Then, for every $n\ge N$ and every $j\in\{1,\ldots,8\}$,
\[
\left|q_j^{(n)}-\frac12\right|<\frac14.
\]
Hence
\[
q_j^{(n)}
>
\frac12-\frac14
=
\frac14 \quad \text{and}\quad q_j^{(n)}<\frac12+\frac14=\frac34.
\]
Therefore all eight quantities are at least $1/4$ for every $n\ge N$.
\end{proof}

\begin{cor}\label{cor:logical-0-1 law in scaled case}
For every sentence $\varphi \in FO(\sigma)$,
$\lim_{n\to\infty} \widetilde{\mbbP}_n(\varphi)$ exists, is either 0 or 1, and the limit does not depend on
the weights $w_1, \ldots, w_8$.
In particular, for every choice of non-negative real weights $w_1,\ldots,w_8$, the scaled sequence
$(\widetilde{\mbbP}_n:n\in\mbbN^+)$ satisfies a $0$--$1$ law for first-order logic.
\end{cor}

\begin{proof}
By Corollary~\ref{corollary for colour concentration} there is $\varepsilon > 0$ such that if 
\[
\mbX_n^\varepsilon := \big\{\mcA \in \mbW_n : \varepsilon n \leq |P^\mcA| \leq (1 - \varepsilon) n \big\}
\]
then $\lim_{n\to\infty} \widetilde{\mbbP}_n\big(\mbX_n^\varepsilon\big) = 1$.
By
Lemma~\ref{lem:scaled-edge-factor}
we have $\lim_{n\to\infty} \widetilde{p}_{ij}^{(n)} = 1/2$ for all $i, j \in \{0, 1\}$.
Therefore
we can argue like in the proof of 
Lemma~\ref{0-1 laws in cases 1 and 2} 
to conclude that for every extension axiom $\psi$, $\lim_{n\to\infty} \widetilde{\mbbP}_n(\psi) = 1$.
Now the claim of the corollary follows directly from Fact~\ref{the theory of extension axioms is complete}
with $\widetilde{\mbbP}_n$ in place of $\mbbP_n$ (as part~(b) of the fact is valid for any sequence
of distributions).
\end{proof}

\section{Proof of cases (3) and (4) of Theorem~\ref{convergence laws}}\label{proof of nonconvergence}

\noindent
In this section we prove parts (3) and (4) of Theorem~\ref{convergence laws}, which means that we prove
the following statement:

\medskip
\noindent
{\em Suppose that $\Eff^{\max} = \{0\}$ and $\eff'(0) = 0$, or that
$\Eff^{\max} = \{1\}$ and $\eff'(1) = 0$.
Then there is a first-order sentence $\chi$ such that $\mbbP_n(\chi)$ does not converge as $n \to \infty$.
}

\medskip

\noindent
As usual we use the abbreviations $\mu_n := \mu_n^\mbbM$ and $\mbbP_n := \mbbP_n^\mbbM$.
Suppose that $\Eff^{\max} = \{0\}$ and $\eff'(0) = 0$.
We will prove, in detail, that there is a first-order sentence $\chi$ such that $\mbbP_n(\chi)$ does not converge as $n \to \infty$,
and then tell how the proof can easily be adjusted to a proof in the case
when $\Eff^{\max} = \{1\}$ and $\eff'(1) = 0$ by just changing a few parameters.
By case~(3) of Theorem~\ref{the colour distribution, exact formulation} (which we have proved)
it follows that there are constants $D > d > 0$ such that
\begin{equation}\label{the colour bounds in nonconvergence}
\lim_{n\to\infty} \mbbP_n\big(\big\{\mcA \in \mbW_n :  d\ln n \leq |P^\mcA| \leq D\ln n \big\}\big) = 1.
\end{equation}

\subsection{The undirected graph structure on the coloured vertices}\label{undirected graph structure}

Let $E$ be a relation symbol of arity $2$.
By an {\em undirected graph without loops} we mean a finite $\{E\}$-structure $\mcG$
such that 
\[
\mcG \models \forall x \neg E(x, x) \wedge \forall x \forall y \big(E(x, y) \rightarrow E(y, x)\big).
\]
If $\mcG$ is an undirected graph without loops then let $E(\mcG)$ denote the set of ordered
pairs $(a, b)$ of elements from the domain (vertex set) of $\mcG$ such that $\mcG \models E(a, b)$.
Thus, the number of unordered edges of $\mcG$ is $|E(\mcG)|/2$.
For every $n \in \mbbN^+$ let $\mbG_n$ be the set of all undirected graphs without loops with
domain (vertex set) $[n]$.
For any $p \in (0, 1)$ and $\mcG \in \mbG_n$ define
\[
\mbbP_n^p(\mcG) = p^{|E(\mcG)|/2} (1-p)^{\binom{n}{2} - |E(\mcG)|/2},
\]
so $\mbbP_n^p$ is a probability distribution on $\mbG_n$ which, for every unordered pair $\{a, b\}$
of different vertices,
assigns probability $p$ to the existence of an edge between $a$ and $b$,
independently of the existence of edges between other pairs.

For every $n$ and $\mcA \in \mbW_n$ let $G_P(\mcA)$ denote the undirected graph without loops
with vertex set $P^\mcA$ and where, for all distinct $a, b \in P^\mcA$, 
\[
G_P(\mcA) \models E(a, b) \qquad \text{ if and only if } \qquad 
\mcA \models R(a, b) \wedge R(b, a).
\]
Recall the number $p_{11} \in (0, 1)$ from
Notation~\ref{not:typed}.
By Proposition~\ref{lem:edge-factor},
the probability, using $\mbbP_n := \mbbP_n^\mbbM$ determined by the MLN $\mbbM$,
that a random $\mcA \in \mbW_n$ satisfies 
$\mcA \models R(a, b) \wedge R(b, a)$ for distinct $a, b \in P^\mcA$ is $p_{11}^2$,
independently of whether $R(a', b')$ holds or not for $(a', b') \neq (a, b), (b, a)$.
For the rest of the proof let
\[
p := p_{11}^2.
\]
It follows that if $B \subseteq [n]$, $m := |B|$ and $\mcG$ is an undirected graph without loops with
vertex set $B$, then
\begin{align*}
&\mbbP_n\big(\big\{\mcA \in W_n : G_P(\mcA) = \mcG \big\} \ \big| \ 
\big\{\mcA \in \mbW_n : P^\mcA = B\big\}\big) = \\
&p^{|E(\mcG)|/2} (1-p)^{\binom{m}{2} - |E(\mcG)|/2}.
\end{align*}
The above together with the definition of $\mbbP_n^p$ above implies that for every first-order, or second-order, sentence $\varphi$ over $\{E\}$,
and every $1 \leq m \leq n$,
\begin{equation}\label{P-n and P-n-p coincide}
\begin{aligned}
&\mbbP_n\big(\big\{\mcA \in \mbW_n : G_P(\mcA) \models \varphi\big\} \ \big| \ 
\big\{\mcA \in \mbW_n : |P^\mcA| = m\big\}\big) = \\
&\mbbP_m^p\big(\big\{\mcG \in \mbG_m : \mcG \models \varphi \big\}\big).
\end{aligned}
\end{equation}

\subsection{Second-order logic}

A {\em second-order variable of arity $k$} is a symbol that denotes
an arbitrary $k$-ary relation.
{\em Second-order logic} over some signature, say $\tau$, is the extension of first-order logic over $\tau$
obtained by (in addition to the usual first-order syntax), 
for any $k$, letting $X(x_1, \ldots, x_k)$ and $X = Y$ be atomic formulas
if $X$ and $Y$ are second-order variables of arity $k$ and $x_1, \ldots, x_k$ are first-order variables
which thus denote arbitrary elements from some domain; the meaning of  $X(x_1, \ldots, x_k)$
is that the tuple $(x_1, \ldots, x_k)$ satisfies $X$.
Moreover, in second-order logic we allow the constructions 
\[
\exists X \varphi \qquad \text{ and } \qquad \forall X \varphi
\]
with the meaning, if $X$ is $k$-ary, that ``some $k$-ary relation $X$ satisfies $\varphi$'', and
```every $k$-ary relation $X$ satisfies $\varphi$''.
A second-order formula without free variables is called a {\em second-order sentence}.
The next example of a second-order sentence will play a crucial role later in the proof.

\begin{exam}\label{defining evenness in SO}{\rm
Let $U$ denote a second order variable of arity $2$.
Let $\varphi_{ev}$ denote the following second-order sentence over $\es$ (the empty signature)
which expresses that the domain has even cardinality by saying that there is a perfect matching
between elements in the universe:
\[
\exists U \big( \forall x \neg U(x, x) \  \wedge \
\forall x \forall y \big[U(x, y) \leftrightarrow U(y, x)\big]\ \wedge \ 
\forall x \exists y \big[ U(x, y) \wedge \forall z \big[ U(x, z) \rightarrow z = y\big]\big] \big).
\]
Then, for every signature $\tau$ and every finite $\tau$-structure $\mcA$ with domain $A$,
$\mcA \models \varphi_{ev}$ if and only if $|A|$ is even.
}\end{exam}
 
\noindent
By a {\em monadic second-order formula} over $\tau$ we mean a second-order formula
in which all second-order variables (if any) have arity 1;
so intuitively the formula may only quantify over elements of the domain and over 
subsets (1-ary relations) of the domain.

Let $\widetilde{+}$ and $\widetilde{\cdot}$ denote relation symbols of arity 3.
For every $n \in \mbbN^+$, let $\big([n]; \tilde{+}, \widetilde{\cdot}\big)$ denote
the $\{\widetilde{+}, \widetilde{\cdot}\}$-structure with domain $[n]$ such that, for all $i, j, k \in [n]$,
\begin{align*}
&\big([n]; \widetilde{+}, \widetilde{\cdot}\big) \models \widetilde{+}(i, j, k) \ 
\Longleftrightarrow \ i + j = k, \text{ and}\\
&\big([n]; \widetilde{+}, \widetilde{\cdot}\big) \models \widetilde{\cdot}(i, j, k) \ 
\Longleftrightarrow \ i \cdot j = k.
\end{align*}

\noindent
The following result follows from a result of Kaufmann and Shelah in \cite[Theorem~1]{KS} and the closing remark in the same article.
It is also stated in \cite[Theorem~9]{HHMS}.
 
\begin{theor}\label{SO to MSO} \cite{KS}
Let $\varphi$ be a second-order sentence over $\{\widetilde{+}, \widetilde{\cdot}\}$ such that all second-order
variables in $\varphi$ are of arity 1 or 2.
Then there is an monadic second-order sentence $\psi$ over $\{E\}$ such that 
\[
\lim_{n\to\infty} \mbbP_n^p\big(\big\{\mcG \in \mbG_n : 
\mcG \models \psi \text{ if and only if } 
\big([\lfloor \sqrt{n} \rfloor]; \widetilde{+}, \widetilde{\cdot}\big) \models \varphi \big\} \big)
= 1.
\]
\end{theor}

\noindent
By applying Theorem~\ref{SO to MSO} to $\varphi_{ev}$ 
from Example~\ref{defining evenness in SO}
we find a monadic second-order formula
$\psi$ over $\{E\}$ such that 
\begin{equation}\label{psi expresses evenness}
\lim_{n\to\infty} \mbbP_n^p\big( \big\{ \mcG \in \mbG_n :
\mcG \models \psi \text{ if and only if } \lfloor\sqrt{n}\rfloor \text{ is even} \big\} \big) = 1.
\end{equation}

\subsection{Nonconvergence of a first-order sentence}

Define, for all $k \in \mbbN^+$,
\[
n_k := \big\lfloor (k^2 + k) \exp\big(2|c_2|(k^2 + k)\big) \big\rfloor,
\]
so $\lim_{k\to\infty} n_k = \infty$.

\begin{lem}\label{square root of n-k is k}
$\lim_{k\to\infty} \mbbP_{n_k}\big( \big\{ \mcA \in \mbW_{n_k} : 
\big\lfloor \sqrt{|P^\mcA|} \big\rfloor = k \big\} \big) = 1.$
\end{lem}

\begin{proof}
Recall that $\eff(\alpha) = c_2\alpha^2 + c_1\alpha + c_0$.
The assumptions that $\Eff^{\max} = $ and $\eff'(0) = 0$ 
imply that $c_1 = 0$, $c_2 < 0$ and $\eff(0) = c_0$.
From~\eqref{ln of Z-n(m)} we get, for all $0 \leq m \leq n$,
\begin{align*}
Z_n(m) &= \binom{n}{m} \exp\big(n^2\eff(m/n)\big) = \binom{n}{m} \exp\big(n^2\big(c_2(m/n)^2 + c_0\big)\big)\\
&= \exp\big(c_0 n^2\big) \binom{n}{m} \exp\big(c_2m^2\big)
\end{align*}
where the factor $\exp\big(c_0 n^2\big)$ is independent from $m$.
Setting
\begin{equation}\label{definition of Z'}
Z'_n(m) := \binom{n}{m} \exp\big(c_2m^2\big)
\end{equation}
we get
\begin{equation}\label{probability of W-n-m in terms of Z'}
\mbbP_n\big(\mbW_n(m)\big) = \frac{Z_n(m)}{\sum_{l=0}^n Z_n(l)} = 
\frac{Z'_n(m)}{\sum_{l=0}^n Z'_n(l)}
\end{equation}
For all $0 \leq m < n$ we have 
\begin{equation}\label{expression of Z' over Z'}
\frac{Z'_{n}(m+1)}{Z'_{n}(m)} = 
\frac{\binom{n}{m+1} \exp\big(c_2(m + 1)^2\big)}{\binom{n}{m} \exp\big(c_2 m^2\big)} = 
\frac{n-m}{m+1}\exp\big(c_2(2m + 1)\big).
\end{equation}
Using~\eqref{expression of Z' over Z'} we get, for $0 \leq m \leq n-2$,
\begin{align*}
\frac{Z'_{n}(m+2)}{Z'_{n}(m+1)} \bigg/ \frac{Z'_{n}(m+1)}{Z'_{n}(m)} = 
\frac{n-m-1}{n-m} \cdot \frac{m+1}{m+2} \cdot e^{2c_2} < 1 \ \ \text{ as } c_2 < 0.
\end{align*}
Hence the sequence $\frac{Z'_n(m+1)}{Z'_n(m)}$, $m = 0, 1, \ldots, n-1$, is decreasing.

Let $t_k := k^2 + k$ so $n_k = \lfloor t_k e^{2|c_2|t_k} \rfloor$.
Using~\eqref{expression of Z' over Z'} and that $n_k \leq t_k e^{2|c_2|t_k}$ we get
\begin{align*}
&\frac{Z'_{n_k}(t_k + 1)}{Z'_{n_k}(t_k)} = 
\frac{n_k - t_k}{t_k + 1} \exp\big(c_2(2t_k + 1)\big) \\
&\leq \frac{t_k \exp\big(2|c_2|t_k\big)}{t_k + 1} \exp\big(c_2(2t_k + 1)\big) \leq
e^{c_2} < e^{c_2/2} \ \ \text{ as $c_2 < 0$.}
\end{align*}
Since $n_k \geq t_k e^{2|c_2|t_k} - 1$ we get $n_k - t_k + 1 \geq t_k e^{2|c_2|t_k} - t_k > 0$ if $k$ is large enough.
It follows that
\begin{align*}
&\frac{Z'_{n_k}(t_k - 1)}{Z'_{n_k}(t_k)} =
\frac{\binom{n_k}{t_k - 1} \exp\big(c_2(t_k - 1)^2\big)}{\binom{n_k}{t_k} \exp\big(c_2 t_k^2\big)} = 
\frac{t_k}{n_k - t_k + 1} \exp\big(c_2(1 - 2t_k)\big)  \\
&\leq \frac{t_k}{ t_k e^{2|c_2|t_k} - t_k} \exp\big(c_2(1 - 2t_k)\big) =
\frac{\exp\big(c_2(1 - 2t_k)\big)}{\exp\big(-2c_2 t_k\big) - 1} \\
&= \frac{e^{c_2}}{1 - e^{2c_2 t_k}} \to e^{c_2} \text{ as } k \to \infty, \text{ so }\\
&\frac{Z'_{n_k}(t_k - 1)}{Z'_{n_k}(t_k)} \leq e^{c_2/2}  \ \ \text{ if $k$ is large enough.}
\end{align*}
Recall the observation above that the
sequence $\frac{Z'_n(m+1)}{Z'_n(m)}$, $m = 0, 1, \ldots, n-1$, is decreasing.
For large enough $k$ and $j \in \mbbN^+$ such that $t_k + j \leq n_k$ we get
\begin{align}\label{bound of Z' and j}
&\frac{Z'_{n_k}(t_k + j)}{Z'_{n_k}(t_k)} = 
\prod_{i=0}^{j-1} \frac{Z'_{n_k}(t_k + i + 1)}{Z'_{n_k}(t_k + i)} \leq 
\bigg(\frac{Z'_{n_k}(t_k + 1)}{Z'_{n_k}(t_k)}\bigg)^j \leq e^{c_2 j /2}.
\end{align}
For large enouh $k$ and $1 \leq j \leq t_k$ we get
\begin{align}\label{bound of Z' and j again}
\frac{Z'_{n_k}(t_k - j)}{Z'_{n_k}(t_k)} = 
\prod_{i = t_k - j}^{t_k - 1} \frac{Z'_{n_k}(i)}{Z'_{n_k}(i + 1)} \leq
\bigg(\frac{Z'_{n_k}(t_k - 1)}{Z'_{n_k}(t_k)}\bigg)^j \leq
e^{c_2 j / 2}.
\end{align}
Let $\rho := e^{c_2/2}$.
Using \eqref{probability of W-n-m in terms of Z'} and \eqref{bound of Z' and j} we get, for large enough $k$,
\begin{align*}
&\mbbP_{n_k}\Bigg(\bigcup_{j=k+1}^{n_k - t_k} \mbW_{n_k}(t_k + j)\Bigg) = 
\sum_{j=k+1}^{n_k - t_k} \mbbP_{n_k}\big(\mbW_{n_k}(t_k + j)\big) =
\sum_{j=k+1}^{n_k - t_k} \frac{Z'_{n_k}(t_k + j)}{\sum_{l=0}^{n_k} Z'_{n_k}(l)} \\
&\leq \sum_{j=k+1}^{n_k - t_k} \frac{Z'_{n_k}(t_k + j)}{Z'_{n_k}(t_k)} 
\leq \sum_{j=k+1}^{n_k - t_k} \rho^j \leq \sum_{j=k+1}^\infty \rho^j = \frac{\rho^{k+1}}{1 - \rho}
\to 0 \ \ \text{ as } \ k \to \infty.
\end{align*}
Using \eqref{probability of W-n-m in terms of Z'} and \eqref{bound of Z' and j again} we also get, for large enough $k$,
\begin{align*}
&\mbbP_{n_k}\Bigg(\bigcup_{j=k+1}^{t_k} \mbW_{n_k}(t_k - j)\Bigg) = 
\sum_{j=k+1}^{t_k} \mbbP_{n_k}\big(\mbW_{n_k}(t_k - j)\big) =
\sum_{j=k+1}^{t_k} \frac{Z'_{n_k}(t_k - j)}{\sum_{l=0}^{n_k} Z'_{n_k}(l)} \\
&\leq \sum_{j=k+1}^{t_k} \frac{Z'_{n_k}(t_k - j)}{Z'_{n_k}(t_k)} 
\leq \sum_{j=k+1}^{t_k} \rho^j \leq \sum_{j=k+1}^\infty \rho^j = \frac{\rho^{k+1}}{1 - \rho}
\to 0 \ \ \text{ as } \ k \to \infty.
\end{align*}
Now we have proved that 
\[
\lim_{k\to\infty} \mbbP_{n_k}\big(\big\{\mcA \in \mbW_{n_k} : t_k - k \leq |P^\mcA| \leq t_k + k\big\}\big) = 1.
\]
But as $t_k = k^2 + k$ we have $t_k - k \leq |P^\mcA| \leq t_k + k$ if and only if 
$\lfloor \sqrt{|P^\mcA|} \rfloor = k$. Thus the lemma is proved.
\end{proof}

\noindent
Define, for $k \in \mbbN^+$,
\[
\mbX_{n_k} := \Big\{ \mcA \in \mbW_{n_k} : \Big\lfloor \sqrt{|P^\mcA|} \Big\rfloor = k \Big\}.
\]
Note that $k = \lfloor \sqrt{m} \rfloor$ if and only if $k^2 \leq m \leq (k+1)^2 - 1 = k^2 + 2k$.
Therefore
\[
\mbX_{n_k} = \big\{ \mcA \in \mbW_{n_k} : k^2 \leq |P^\mcA| \leq k^2 + 2k \big\}.
\]
Suppose that $k^2 \leq m \leq k^2 + 2k$. Let $\varepsilon \in (0, 1/2)$.
By~\eqref{P-n and P-n-p coincide} and~\eqref{psi expresses evenness} if follows that if $k$
is sufficiently large, then
\begin{align*}
&\mbbP_{n_k}\big(\big\{\mcA \in \mbW_{n_k} : G_P(\mcA) \models \psi \ \big| \ 
\big\{\mcA \in \mbW_{n_k} : |P^\mcA| = m \big\}\big) = \\
&\mbbP_m^p\big(\big\{\mcG \in \mbG_m : \mcG \models \psi \big\}\big) \in 
\begin{cases}
(1-\varepsilon, 1) \text{ if $k$ is even},\\
(0, \varepsilon) \text{ if $k$ is odd}.
\end{cases}
\end{align*}
If follows that if $k$ is sufficiently large then
\begin{equation}\label{psi conditioned on X}
\mbbP_{n_k}\big(\big\{\mcA \in \mbW_{n_k} : G_P(\mcA) \models \psi \ \big| \ \mbX_{n_k}\big) \in 
\begin{cases}
(1-\varepsilon, 1) \text{ if $k$ is even},\\
(0, \varepsilon) \text{ if $k$ is odd}.
\end{cases}
\end{equation}

\noindent
Suppose that we have found a first-order sentence $\chi$ over $\sigma = \{P, R\}$ such that
\begin{equation}\label{equivalence of psi and chi}
\lim_{k\to\infty} \mbbP_{n_k}\big(\big\{\mcA \in \mbW_{n_k} : 
G_P(\mcA) \models \psi \text{ if and only if } \mcA \models \chi \big\}\big) = 1.
\end{equation}
Then, as $\lim_{k\to\infty} \mbbP_{n_k}\big(\mbX_{n_k}\big) = 1$
by Lemma~\ref{square root of n-k is k},
it follows from~\eqref{psi conditioned on X} (where $\varepsilon \in (0, 1/2)$ is
arbitrary) and~\eqref{equivalence of psi and chi}
that for any $\varepsilon \in (0, 1/2)$ and all large enough $k$,
\[
\mbbP_{n_k}\big(\big\{\mcA \in \mbW_{n_k} : \mcA \models \chi \big\}\big) \in
\begin{cases}
(1-\varepsilon, 1) \text{ if $k$ is even},\\
(0, \varepsilon) \text{ if $k$ is odd}.
\end{cases}
\]
Consequently $\mbbP_n\big(\big\{\mcA \in \mbW_n : \mcA \models \chi \big\}\big)$
does not converge as $n\to\infty$.

{\em Thus it remains to prove that there is a first-order sentence $\chi$ over $\{P, R\}$ such 
that~\eqref{equivalence of psi and chi} holds.}
Let $D > d > 0$ be the constants from \eqref{the colour bounds in nonconvergence}.
Let
\begin{align*}
&q := p_{10} \quad \text{ where $p_{10}$ is like in Notation~\ref{not:typed}, so $0 < q < 1$, and let} \\
&\delta := \min\{q, 1-q\} \ \text{ and } \ \rho := \max\{q, 1-q\}.
\end{align*}
Then $0 < \delta \leq \rho < 1$ and we can choose an integer $t$ such that 
\[
2\rho^t D \ln(1/\delta) < 1.
\]
For each $\bar{s} = (s_1, \ldots, s_t) \in \{0, 1\}^t$ define a formula
\begin{align*}
Part_{\bar{s}}(y, x_1, \ldots, x_t) := 
P(y) \ \wedge \ \bigwedge_{\substack{1 \leq i \leq t \\ s_i = 1}}
R(y, x_i) \ \wedge \ 
\bigwedge_{\substack{1 \leq i \leq t \\ s_i = 0}}
\neg R(y, x_i).
\end{align*}
It follows that if $\mcA \in \mbW_n$, $a_1, \ldots, a_t \in [n] \setminus P^\mcA$, and $b \in P^\mcA$,
then there is a unique $\bar{s} = (s_1, \ldots, s_t) \in \{0, 1\}^t$ such that 
$\mcA \models Part_{\bar{s}}(b, a_1, \ldots, a_t)$.
Hence we can view every $t$-tupe $\bar{a} = (a_1, \ldots, a_t)$ of vertices outside of $P^\mcA$ as defining a 
partition of $P^\mcA$ via the parametrized formulas $Part_{\bar{s}}(y, \bar{a})$ for 
$\bar{s} \in \{0, 1\}^t$. That is, every $Part_{\bar{s}}(y, \bar{a})$ defines a part of $P^\mcA$
and the union of all disjoint parts equals $P^\mcA$.
Let $\bar{s}_1, \ldots, \bar{a}_{2^t}$ be an enumeration of $\{0, 1\}^t$.
Now define
\begin{align*}
&Code(z_1, \ldots, z_{2^t}, x_1, \ldots, x_t) :=
\bigwedge_{i=1}^{2^t} \bigg(\neg P(z_i) \wedge \bigwedge_{j=1}^t z_i \neq x_j \bigg) \ \text{ and} \\
&Mem(y, z_1, \ldots, z_{2^t}, x_1, \ldots, x_t) := 
\bigvee_{i=1}^{2^t} \big(Part_{\bar{s}_i}(y, x_1, \ldots, x_t) \ \wedge \ 
R(y, z_i)\big).
\end{align*}
The intuition is that if $a_1, \ldots, a_t \in [n] \setminus P^\mcA$ are distinct and
$c_1, \ldots, c_{2^t} \in [n] \setminus (P^\mcA  \cup  \{a_1, \ldots, a_t\})$,
then the sequence $c_1, \ldots, c_{2^t}$ is a code of the set 
\[
\big\{b \in P^\mcA :  \mcA \models Mem(b, c_1, \ldots, c_{2^t}, a_1, \ldots, a_t) \big\}.
\]
We now describe a formula $Good(x_1, \ldots, x_t)$ that expresses that every
subset of $P^\mcA$ has a code, and we explain why after its description.
Let $Good(x_1, \ldots, x_t)$ be the conjunction of three formulas expressing, respectively, that
\begin{itemize}
\item the empty set has a code,
\item every singleton subset of $P^\mcA$ has a code, and
\item if each of $X \subseteq P^\mcA$ and $Y \subseteq P^\mcA$ has a code then so
has $X \cup Y$.
\end{itemize}
Since every finite set $\{b_1, \ldots, b_l\}$ can be constructed inductively from the 
empty set (the base case) by adding one element at a time (the inductive step) it follows
that the conjunction of the three formulas expresses that every subset of $P^\mcA$ has a code.
In more detail, we can define the three formulas as follows, where
$\bar{x} = (x_1, \ldots, x_t)$ and $\bar{z} = (z_1, \ldots, z_{2^t})$:
\begin{align*}
&\exists \bar{z} \big(Code(\bar{z}, \bar{x}) \ \wedge \ 
\forall y \big[ P(y) \rightarrow \neg Mem(y, \bar{z}, \bar{x}) \big]\big), \\
&\forall y \big( P(y) \rightarrow \exists \bar{z} 
\big[ Code(\bar{z}, \bar{x}) \ \wedge \ 
\forall u \big( P(u) \rightarrow \big[ Mem(u, \bar{z}, \bar{x}) \leftrightarrow u = y\big]\big)\big]\big), \\
&\forall \bar{z} \forall \bar{w} \big( 
\big[ Code(\bar{z}, \bar{x}) \wedge Code(\bar{w}, \bar{x}) \big] \rightarrow \\
&\exists \bar{v} \big[ Code(\bar{v}, \bar{x}) \ \wedge \ 
\forall y \big( P(y) \rightarrow \big[ Mem(y, \bar{v}, \bar{x}) \leftrightarrow 
\big( Mem(y, \bar{z}, \bar{x}) \vee Mem(y, \bar{w}, \bar{x}) \big) \big] \big) \big] \big).
\end{align*}

\begin{lem}\label{a good tuple exists}
$\lim_{n\to\infty} \mbbP_n\big( \exists \bar{x} \, Good(\bar{x})\big) = 1$.
\end{lem}

\begin{proof}
Recall that by \eqref{the colour bounds in nonconvergence} we have
\[
\lim_{n\to\infty} \mbbP_n\big(\big\{\mcA \in \mbW_n :  d\ln n \leq |P^\mcA| \leq D\ln n \big\}\big) = 1
\]
where $0 < d < D$. Therefore it suffices to prove that if 
$\mbU_n$ is the set of all $\mcA \in \mbW_n$ such that $d\ln n \leq |P^\mcA| \leq D\ln n$,
then $\lim_{n\to\infty} \mbbP_n\big( \exists \bar{x} \, Good(\bar{x}) \ | \ \mbU_n \big) = 1$.

Let $\mcA \in \mbU_n$ and condition on $P^\mcA = B \subseteq [n]$ and let $m := |B|$, so 
$d\ln n \leq m \leq D\ln n$.
We may assume that $n$ is large enough such that $n - m \geq n - D\ln n > t$.
Then we can choose any distinct $a_1, \ldots, a_t \in [n] \setminus B$.
For every $\bar{s} = (s_1, \ldots, s_t) \in \{0, 1\}^t$ let
\[
B_{\bar{s}} := \big\{ b \in B : \mcA \models Part_{\bar{s}}(b, \bar{a}) \big\} \ \text{ where } \bar{a} = (a_1, \ldots, a_t).
\]
Our first aim is to find an asymptotically almost sure upper bound to each $|B_{\bar{s}}|$ (conditioned on $P^\mcA = B$).
Let $Y_{\bar{s}} := |B_{\bar{s}}|$.
By Proposition~\ref{lem:edge-factor},
conditioned on $P^\mcA = B$, all directed edges are independent and each edge from a vertex
in $B$ to a vertex outside of $B$ has probability $q$ (where $q := p_{10}$ as mentioned above).
We can view $Y_{\bar{s}}$ as a random variable on the set of $\mcA \in \mbW_n$ such that $P^\mcA = B$.
For each $b \in B$ the probability that $b \in B_{\bar{s}}$, or equivalently that 
$\mcA \models Part_{\bar{s}}(b, \bar{a})$, is
\[
\alpha_{\bar{s}} := q^{|\{i : s_i = 1\}|} (1 - q)^{|\{i : s_i = 0\}|} \leq \rho^t.
\]
This probability is independent of whether $b' \in B_{\bar{s}}$ for $b' \in B \setminus \{b\}$, so
$Y_{\bar{s}}$ is binomially distributed.
Therefore $\mbbE[Y_{\bar{s}} \ | \ P^\mcA = B] = m\alpha_{\bar{s}} \leq m \rho^t$ and
$Var[Y_{\bar{s}} \ | \ P^\mcA = B] = m\alpha_{\bar{s}}(1 - \alpha_{\bar{s}})$.
Without loss of generality we can assume that $t$ is large enough so that $\rho^t \leq 1/4$.
Then $m\alpha_{\bar{s}}(1 - \alpha_{\bar{s}}) \leq m \rho^t \leq  m/4$, so
$Var[Y_{\bar{s}} \ | \ P^\mcA = B] \leq m/4$.
If $Y_{\bar{s}} > 2m \rho^t$ then $Y_{\bar{s}} - \mbbE[Y_{\bar{s}} \ | \ P^\mcA = B] > m \rho^t$,
so Chebyshev's inequality gives that
\[
\mbbP_n\big( Y_{\bar{s}}  > 2m \rho^t \ | \  P^\mcA = B \big) \leq \frac{m/4}{(m \rho^t)^2} = \frac{1}{4m \rho^{2t}}.
\]
Let $\gamma := 2\rho^t D$. As there are $2^t$ choices of $\bar{s}$
it follows that, conditioned on $P^\mcA = B$, with probability at least
$1 - \frac{2^t}{4m \rho^{2t}}$ we have 
\[
|B_{\bar{s}}| \leq 2m \rho^t \leq 2 \rho^t D \ln n = \gamma \ln n \quad \text{ for all } \bar{s} \in \{0, 1\}^t.
\]

Now we will show that, conditioned on $P^\mcA = B$ and $|B_{\bar{s}}| \leq \gamma \ln n$ for all $\bar{s} \in \{0, 1\}^t$, 
the probability that every subset of $B$
has a code tends to one as $n \to \infty$. Moreover we will get a lower bound of this probability that tends to 1
and does not depend on the choice of $B$ (as long as $d \ln n \leq |B| \leq D \ln n$). So in the end we only condition on $\mbU_n$.
We begin by considering some $X \subseteq B_{\bar{s}}$ for an arbitrary $\bar{s} \in \{0, 1\}^t$.
Suppose that $\bar{s} = \bar{s}_i$ in the enumeration $\bar{s}_1, \ldots, \bar{s}_{2^t}$ of $\{0, 1\}^t$.
By definition, a sequence $c_1, \ldots, c_{2^t} \in [n] \setminus (B \cup \rng(\bar{a}))$ is a code of $X$ if
for all $b \in B_{\bar{s}}$, $b \in X$ if and only if $\mcA \models R(b, c_i)$.
Take any $c \in [n] \setminus (B \cup \rng(\bar{a}))$.
The probability (conditioned on $P^\mcA = B$ and $B_{\bar{s}} \leq \gamma \ln n$)
that, for all $b \in B_{\bar{s}}$, $b \in X$ if and only if $\mcA \models R(b, c)$ is
\[
q^{|X|}(1 - q)^{|B_{\bar{s}}| - |X|} \geq \delta^{|B_{\bar{s}}|}.
\]
We can choose such $c$ in $N := n - m - t$ ways. By using the well known inequality that $1 - u \leq e^{-u}$ if $0 \leq u \leq 1$,
it follows that the probability that no $c$ succeeds (i.e. that $X$ does not have a code) is at most
\[
\big(1 - \delta^{|B_{\bar{s}}|}\big)^N \leq \exp\Big(-\delta^{|B_{\bar{s}}|}\Big)^N
= \exp\Big(-N\delta^{|B_{\bar{s}}|}\Big) \leq \exp\Big(-N \delta^{\gamma \ln n}\Big).
\]
Let $S \subseteq B$. If $S$ does not have a code then there is $\bar{s} \in \{0, 1\}^t$ such that 
$X := S \cap B_{\bar{s}}$ does not have a code.
As there are $2^t$ choices of $\bar{s}$ and $2^{|B_{\bar{s}}|} \leq 2^{\gamma \ln n}$ choices of $X \subseteq B_{\bar{s}}$
it follows that, conditioned on $\mcA \in \mbU_n$ and that 
$|B_{\bar{s}}| \leq \gamma \ln n$ for all $\bar{s}$, 
the probability that some $S \subseteq B$ does not have a code is at most (for large enough $n$)
\begin{align*}
&2^t \cdot  2^{\gamma \ln n} \exp\big(-N \delta^{\gamma \ln n}\big) \\
&= 2^t n^{\gamma \ln 2} \exp\big(-N n^{\gamma \ln(\delta)}\big) \quad \text{(by changing bases of exponents)} \\
&= 2^t n^{\gamma \ln 2} \exp\big(-N n^{-\gamma \ln(1/\delta)}\big) 
\leq 2^t n^{\gamma \ln 2} \exp\big(-\tfrac{n}{2} n^{-\gamma \ln(1/\delta)}\big) \\
&= 2^t n^{\gamma \ln 2} \exp\big(- \tfrac{1}{2} n^{1 -\gamma \ln(1/\delta)}\big)
\end{align*}
The logarithm of the last expression is 
\[
\ln \big(2^t n^{\gamma \ln 2} \exp\big(- \tfrac{1}{2} n^{1 -\gamma \ln(1/\delta)}\big)\big) = 
t\ln 2 + \gamma \ln 2 \ln n - \tfrac{1}{2} n^{1 -\gamma \ln(1/\delta)}
\]
By the choice of $t$ and $\gamma := 2\rho^t D$ we have $1 > \gamma \ln(1/\delta)$, so
\[
t\ln 2 + \gamma \ln 2 \ln n - \tfrac{1}{2} n^{1 -\gamma \ln(1/\delta)} \to -\infty \ \ \text{ as } n \to \infty.
\]
Hence, conditioned on $\mcA \in \mbU_n$
and $|B_{\bar{s}}| \leq \gamma \ln n$ for all $\bar{s} \in \{0, 1\}^t$,
the probability that some $S \subseteq B$ does not have a code is at most
$2^t n^{\gamma \ln 2} \exp\big(- \tfrac{1}{2} n^{1 -\gamma \ln(1/\delta)}\big) \to 0$ as $n \to \infty$.

Now, conditioned on $\mcA \in \mbU_n$, the probability that $\mcA \not\models \exists \bar{x} \, Good(\bar{x})$
is less or equal to the probability  that $|B_{\bar{s}}| > \gamma \ln n$
for some $\bar{s}$ or that some $S \subseteq B$ does not have a code, and this probability is at most
\[
\frac{2^t}{4 \rho^{2t} d \ln n} + 2^t n^{\gamma \ln 2} \exp\big(- \tfrac{1}{2} n^{1 -\gamma \ln(1/\delta)}\big)
\to 0 \ \ \text{ as } \ n \to \infty.
\]
This concludes the proof of the lemma.
\end{proof}

\begin{lem}\label{good tuple gives translation}
If $\theta$ is a monadic second-order sentence over $\{E\}$
then there is a first-order formula $\theta'(\bar{x})$ over $\{P, R\}$ such that
if $n \in \mbbN^+$, $\mcA \in \mbW_n$, $\bar{a} \in [n]^{|\bar{x}|}$,
and $\mcA \models Good(\bar{a})$, then
\[
G_P(\mcA) \models \theta \quad \text{ if and only if } \quad \mcA \models \theta'(\bar{a}).
\]
\end{lem}

\begin{proof}
Let $\theta$ be a monadic second-order sentence over $\{E\}$.
Suppose that $\mcA \models Good(\bar{a})$.
It follows that every $S \subseteq P^\mcA$ has a code $\bar{c}_S \in (P^\mcA)^{2^t}$ in the sense
that for all $b \in P^\mcA$, $b \in S$ if and only if $\mcA \models Mem(b, \bar{c}_S, \bar{a})$.

First note that any occurence in $\theta$
of an atomic formula $X = Y$ where $X$ and $Y$ are 1-ary second-order
variables can be replaced by $\forall u (X(u) \leftrightarrow Y(u))$ without changing the meaning
of the formula (where $u$ is a fresh first-order variable not used before this replacement).
Then associate a $2^t$-tuple $\bar{z}_X$ of fresh first-order variables to each second-order variable $X$
in $\theta$. As $\theta$ is monadic each such $X$ has arity 1.

By induction on the complexity of subformulas we show that if $\xi(\bar{y}, \bar{X})$
is a subformula of $\theta$ 
(where not all the indicated first or second order variables need to be present)
then there is a first-order formula 
$\xi'(\bar{y}, \bar{w}, \bar{x})$ (with $|\bar{w}| = 2^t|\bar{X}|$) such that for any
$\bar{b} \in (P^\mcA)^{|\bar{y}|}$, $\bar{S} = (S_1, \ldots, S_l)$
where $l = |\bar{X}|$ and
 $S_i \subseteq P^\mcA$ for all $i = 1, \ldots, l$, and any choice of codes $\bar{c}_i$ of $S_i$
for $i = 1, \ldots, l$, if $\bar{c} = \bar{c}_1 \ldots \bar{c}_l$, then
\begin{equation}\label{translation of subformulas}
G_P(\mcA) \models \xi(\bar{b}, \bar{S}) \quad \text{ if and only if } 
\quad \mcA \models \xi'(\bar{b}, \bar{c}, \bar{a})
\end{equation}
The translation is as follows:
\begin{itemize}
\item If $\xi$ has the form $x=y$ then $\xi'$ is $x=y$.
\item If $\xi$ has the form $E(x, y)$ then $\xi'$ is $x \neq y \wedge R(x, y) \wedge R(y, x)$.
\item If $\xi$ has the form $X(u)$ then $\xi'$ is $Mem(u, \bar{z}_X, \bar{x})$.
\item If $\xi$ has the form $\neg \xi_1$ then $\xi'$ is $\neg \xi'_1$.
\item If $\xi$ has the form $\xi_1 \square \xi_2$ where $\square \in \{\wedge, \vee, \rightarrow, \leftrightarrow\}$
then $\xi'$ is $\xi'_1 \square \xi'_2$.
\item If $\xi$ has the form $\exists x \xi_1$ (respectively $\forall x \xi_1$) then $\xi'$ is
$\exists x (P(x) \wedge \xi'_1)$ (respectively $\forall x (P(x) \rightarrow \xi'_1)$.
\item If $\xi$ has the form $\exists X \xi_1$ (respectively $\forall X \xi_1$) then $\xi'$ is
$\exists \bar{z}_X (Code(\bar{z}_X, \bar{x}) \wedge \xi'_1)$ 
(respectively $\forall \bar{z}_X (Code(\bar{z}_X, \bar{x}) \rightarrow \xi'_1)$).
\end{itemize}
It is either evident, or follows from the definition of
$G_P(\mcA)$ and the assumption that each $S \subseteq P^\mcA$ has a code, that
the equivalence~\eqref{translation of subformulas} is preserved by each step in the translation.
It follows that after the final step we have $G_P(\mcA) \models \theta$
if and only if $\mcA \models \theta'(\bar{a})$.
\end{proof}

\noindent
Recall the monadic second-order sentence $\psi$ from~\eqref{psi conditioned on X}.
According to Lemma~\ref{good tuple gives translation} there is a first-order sentence $\psi'(\bar{x})$
such that if $n \in \mbbN^+$ and $\mcA \in \mbW_n$, then
\begin{equation}\label{psi and psi'}
\text{$\bar{a} \in [n]^{|\bar{x}|}$ and $\mcA \models Good(\bar{a})$ implies that }
\mcA \models \psi  \text{ if and only if }  \mcA \models \psi'(\bar{a}).
\end{equation}
Define
\[
\chi := \exists \bar{x} (Good(\bar{x}) \wedge \psi'(\bar{x})).
\]
By \eqref{psi and psi'}
and lemmas~\ref{square root of n-k is k}
and~\ref{a good tuple exists}  it follows that
\[
\lim_{k\to\infty} \mbbP_{n_k}\big(\big\{\mcA \in \mbW_{n_k} : 
G_P(\mcA) \models \psi \text{ if and only if } \mcA \models \chi \big\}\big) = 1.
\]
Thus \eqref{equivalence of psi and chi} holds and, and as we observed above, it follows
that $\mbbP_n\big(\chi\big)$ does not converge as $n \to \infty$.

Now suppose that $\Eff^{\max} = \{1\}$ and $\eff'(1) = 0$.
For every $n$ and $\mcA \in \mbW_n$ let $Q^\mcA = [n] \setminus P^\mcA$.
By case~(4) of Theorem~\ref{the colour distribution, exact formulation},
 there are constants $D > d > 0$ such that
\begin{equation}
\lim_{n\to\infty} \mbbP_n\big(\big\{\mcA \in \mbW_n :  d\ln n \leq |Q^\mcA| \leq D\ln n \big\}\big) = 1.
\end{equation}
Now we can reason just as we did above, except that we let $Q^\mcA$ take the role of $P^\mcA$
and we let $p := p_{00}$ and $q := p_{01}$.

\end{document}